\documentclass[preprint,12pt]{elsarticle}

\usepackage{ifpdf}
\usepackage{amssymb}
\usepackage{amsthm}
\usepackage{xcolor}
\usepackage{lipsum}
\usepackage{amsfonts}
\usepackage{graphicx}

\usepackage{algorithmic}
\usepackage{enumerate}
\usepackage{multirow}
\usepackage{array}
\usepackage{booktabs}  
\usepackage{multirow}  
\usepackage{algorithm}
\usepackage{amsmath}
\usepackage[hidelinks]{hyperref}
\usepackage[capitalize,nameinlink]{cleveref}
\usepackage{amsfonts}
\usepackage{mathrsfs}
\usepackage{tikz}
\usetikzlibrary{matrix}
\counterwithin{figure}{section}
\counterwithin{table}{section}

\usepackage{subcaption}

\ifpdf
  \DeclareGraphicsExtensions{.pdf,.png,.jpg}
\else
  \DeclareGraphicsExtensions{.eps}
\fi

\theoremstyle{plain}
\newtheorem{theorem}{Theorem}[section]
\newtheorem{lemma}[theorem]{Lemma}
\newtheorem{corollary}[theorem]{Corollary}
\newtheorem{proposition}[theorem]{Proposition}
\newtheorem{definition}[theorem]{Definition}

\theoremstyle{remark}
\newtheorem{remark}[theorem]{Remark}

\newtheorem{example}[theorem]{\textbf{Example}}
\crefname{hypothesis}{Hypothesis}{Hypotheses}
\crefname{fact}{Fact}{Facts}

\numberwithin{equation}{section}
\numberwithin{algorithm}{section}

\usepackage{amsopn}

\hypersetup{
  pdftitle={SC-NuF-CFEI},
  pdfauthor={X. Cai, and L. Kong}
}
\journal{Journal of Computational Physics}

\begin{document}
\begin{frontmatter}

\title{Self-consistent flow maps based on commutator-free exponential integrators\tnoteref{submitted}\\[5pt]
\small\textit{%
Dedicated to Professor Weiwei Sun on his 70th birthday.}}
\tnotetext[submitted]{Submitted to the editors on October 26, 2026.}

\author[aff1]{Xiaofeng Cai}
\ead{xfcai@bnu.edu.cn}
\address[aff1]{Research Center for Mathematics, Faculty of Arts and Sciences, Beijing Normal University, Zhuhai, 519087, P.R. China}

\author[aff2]{Linghui Kong}
\ead{linghuikong@mail.bnu.edu.cn}
\address[aff2]{Department of Mathematics, Faculty of Arts and Sciences, Beijing Normal University, Zhuhai, 519087, P.R. China}

\begin{abstract}
We develop a flow-map framework for nonlinear transport equations,
comprising a self-consistent exact frozen flow and its numerical realization.
An exponential--pullback identity represents each frozen transport exponential in a commutator-free exponential integrator 
as the pullback induced by its exact backward characteristic map.
Composing these maps defines the self-consistent exact frozen flow, 
a time-discrete construction with autonomous frozen fields formed from linear combinations of transport fields evaluated at the stage approximations.
A local freezing-error identity and a one-step estimate for the
Lie--Euler Vlasov--Poisson update separate field variation, field approximation,
and characteristic integration errors; these estimates do not assert
global PDE convergence.
We specialize this construction to the Vlasov--Poisson system 
and obtain the self-consistent numerical flow through velocity-space quadrature, a discrete Poisson solve, field interpolation, and numerical integration of the frozen characteristic equations.
Under suitable regularity assumptions, we prove that the self-consistent exact frozen flow for this system is a volume-preserving diffeomorphism.
Temporal self-convergence tests on Vlasov--Poisson and guiding-center
benchmarks yield rates consistent with the design orders 
and illustrate conservation behavior and the evolution of solution structures in long-time simulations with large time steps.
\end{abstract}


\end{frontmatter}

\section{Introduction}
\label{sec:Intro}

This paper develops a flow-map framework for transport equations of the form 
$\partial_t u + b[u]\cdot\nabla u = 0$, 
where the transport field $b[u]$ depends on the evolving unknown $u$.
Representative examples include the Vlasov--Poisson system in collisionless plasma physics \cite{Glassey1996},
where the phase-space distribution is transported under a self-consistent electrostatic field,
and the vorticity formulation of the two-dimensional incompressible Euler equations \cite{Lowengrub1993},
where the vorticity is advected by the velocity field it induces.
The solution dependence of the transport field couples the computation of the characteristic flow with the evolution of the transported quantity, 
requiring a consistent numerical treatment.
In addition, for sufficiently regular divergence-free transport fields, the associated characteristic flow preserves volume,
a property that a numerical flow approximation need not inherit.


A classical strategy for treating the self-consistent coupling is 
to construct frozen transport fields from numerical stage approximations,
thereby approximating the nonlinear evolution through a sequence of linear transport subproblems with prescribed coefficients.
Our formulation follows this approach using commutator-free Runge--Kutta exponential integrators (CF-RKEI) \cite{CMO2003}.
The resulting composition of frozen-field flows also provides a natural setting for geometric structure preservation:
if the exact subflows preserve a common volume form or symplectic form, their composition preserves the same form \cite{HLW2006,MQ2002}.
This property holds at the level of exact subflows; its preservation after numerical approximation depends on the integrators used to compute them.
In Lie-group methods based on a fixed group action, successive updates by group elements keep the numerical iterates on the group orbit of the initial state \cite{IMKNZ2000}.
Relevant developments include the frozen-flow methods of Crouch and Grossman \cite{CrouchGrossman1993},
the Runge-Kutta-Munthe-Kaas methods on manifolds \cite{MuntheKaas1999},
and the commutator-free Lie-group methods of Celledoni, Marthinsen, and Owren \cite{CMO2003}.
Related constructions were subsequently adapted to Eulerian and semi-Lagrangian schemes \cite{Celledoni2005, CK2009}.

The construction rests on a fundamental identity for autonomous vector fields.
Let $b$ be a sufficiently regular, complete autonomous vector field on a manifold $M$,
let $\varphi_\sigma$ denote its flow,
and define the transport generator on scalar functions by $C_bg:=-\mathcal L_bg$.
Then $C_b$ generates the one-parameter family of pullback operators
\[
  \exp(hC_b)\,g=g\circ\varphi_{-h},
  \qquad g\in C^1(M),\quad h\in\mathbb R.
\]
For a vector field whose flow is only locally defined, the same identity holds on the corresponding flow domains and time intervals.
This identity, established in Proposition~\ref{prop:autonomous-exponential-pullback},
supplies the elementary update used throughout the construction.
In the self-consistent setting, the evolution is governed by
\begin{equation}
  \frac{d}{dt}u(t)=C\bigl(u(t)\bigr)u(t),
  \qquad C(u)g:=-\mathcal L_{b[u]}g,
  \label{eq:intro-self-consistent}
\end{equation}
where $C(u)$ depends on $u$ but acts linearly on $g$ for fixed $u$.
The autonomous identity does not directly resolve this coupling:
the flow of $b_t^u:=b[u(t)]$ depends on the same solution trajectory it transports,
making the pullback representation implicit.
We resolve this at the time-discrete level by forming frozen transport generators as fixed linear combinations of stage generators computed self-consistently within a CF-RKEI scheme.
The exponential--pullback identity converts each transport exponential into the pullback induced by its exact backward characteristic map.
We call the resulting composition the \emph{self-consistent exact frozen flow}.
Here exact refers to the solution of the frozen characteristic equations within the time-discrete construction.
The corresponding realization after approximating the self-consistent fields and integrating the frozen characteristics numerically is termed the \emph{self-consistent numerical flow}.

Among the geometric structures, the one most directly relevant to the present construction is volume preservation.
If a sufficiently regular vector field is divergence-free with respect to the underlying volume form,
its exact flow preserves that form, as does a composition of such flows.
Mesh-based characteristic methods enforce discrete conservation through conservative remapping, 
which underlies the volume corrected characteristic method \cite{Arbogast2006,Arbogast2010,ArbogastWang2011}, 
and conservative semi-Lagrangian schemes \cite{XU2026,CaiGuoQiu2017,CBQ2021}.
In multidimensional schemes, high-order spatial accuracy requires suitable approximations of the upstream geometry and its associated integrals.
When CF-RKEI is coupled with a conservative semi-Lagrangian discretization,
these operations occur within the frozen transport updates at the intermediate stages \cite{CBQ2021,ZHENG2022,ZHENG2024}.
The present formulation represents each exponential factor directly through a frozen characteristic map and reconstructs the transported quantity from the initial data.
At the level of the self-consistent exact frozen flow, divergence-free frozen fields give volume-preserving subflows and a volume-preserving composition.
This geometric property and the design order specified by the CF-RKEI conditions are distinct aspects of the construction.
Their preservation in the self-consistent numerical flow must be assessed with the field approximation and characteristic integration included.

The electrostatic, collisionless Vlasov--Poisson system is the main example of \eqref{eq:intro-self-consistent} considered in this work. 
Let $M:=\mathbb T^d\times\mathbb R^d$ and $z=(x,v)\in M$, 
and denote by \(f(x,v,t)\ge0\) the phase-space distribution function. 
The VP equation is given by
\begin{equation}
  \partial_t f(x,v,t)
  +v\cdot\nabla_x f(x,v,t)
  -E[f](x,t)\cdot\nabla_v f(x,v,t)=0,
  \label{eq:vp}
\end{equation}
where the self-consistent electric field \(E[f]\) is determined through the Poisson equation
\begin{align}
-\Delta_x\phi(x,t)&=\rho[f](x,t),\qquad
E[f](x,t)=-\nabla_x\phi(x,t),
\label{eq:poisson}
\end{align}
with $\rho[f](x,t):=\bar\rho_0-\int_{\mathbb R^d}f(x,v,t)\,dv$,
denoting the charge density fluctuation relative to the constant neutralizing background \(\bar\rho_0\). 
In the notation of \eqref{eq:intro-self-consistent}, the associated transport field is
$
  b[f](x,v,t)=\bigl(v,-E[f](x,t)\bigr).
$
The second model considered in our numerical experiments is the guiding-center Vlasov model, which describes the transverse dynamics of highly magnetized plasmas \cite{CROUSEILLES2010,Cai2019}. 
It governs the evolution of \(\rho(x,t)\) on a periodic rectangle $\Omega_{\mathrm{gc}}$, with \(x=(x_1,x_2)\in\mathbb R^2\), according to
\begin{align}
  \partial_t\rho+E^\perp[\rho]\cdot\nabla\rho=0,
  \label{eq:GDcenter}
\end{align}
where $E[\rho]=-\nabla_x\phi$ is obtained from the periodic Poisson closure of the guiding-center model,
and $E^\perp=(E_{x_2},-E_{x_1})$.
Thus, the guiding-center model also takes the form of a self-consistent transport equation,
with transport field given by the perpendicular electric drift.

A principal numerical challenge in solving the Vlasov equation is the evolution of a distribution function in a high-dimensional phase space.
In addition to the dimensionality barrier, the formation of increasingly fine filamentary structures can severely degrade the efficiency of fixed-resolution discretizations.
A broad numerical literature has therefore developed around this difficulty. 
Particle-in-cell (PIC) methods \cite{BirdsallLangdon1991,Verboncoeur2005}, 
Eulerian approaches \cite{FILBET2003,BanksHittinger2010,Cheng2013}, 
and semi-Lagrangian methods \cite{CROUSEILLES2010,QiuShu2011,Kormann2019,CaiGuoQiu2017,CBQ2021} 
remain among the standard approaches.
A related strategy for reducing reconstruction costs is the fast
kinetic scheme \cite{DimarcoLoubere2013}.
A complementary direction addresses the phase-space complexity by exploiting low-rank representations of the distribution function. 
Early developments include semi-Lagrangian solvers based on tensor-train formats \cite{Kormann2015TT} 
and adaptive tensor decompositions \cite{EHRLACHER2017}. 
Subsequent methods have largely followed two paradigms: step-and-truncate (SAT) approaches 
and dynamical low-rank (DLR) approximations \cite{Einkemmer2018DLR,EINKEMMER2025,ZHENG2025}.
Structure-preserving variants have also been developed to improve the conservation of macroscopic moments and other invariants \cite{Guo2024CLR,Guo2024LoMaC,Guo2024LoMaCDG,CoughlinHuShumlak2024,Einkemmer2019,Einkemmer2021}. 
The efficiency of these approaches is tied to the compressibility of the evolving distribution function.

This observation motivates an alternative viewpoint: 
one may approximate the characteristic flow transporting the initial data.
Since the solution can be recovered through the pullback of the initial distribution, 
the approximation target is shifted from the increasingly oscillatory function \(f\) to the underlying flow map.
This perspective has led to a class of flow-map methods 
that reconstruct the solution from an approximation of the backward characteristic map.
Representative examples are the Characteristic Mapping Method (CMM) \cite{Krah2023} and Numerical Flow Iteration (NuFI) method \cite{KW2024},
which is most closely related to the present work. 
NuFI represents the numerical solution as
\begin{equation*}
  f_{\tau,h}^n=f_0\circ\Psi^{0}_{n\tau,h},
\end{equation*}
where \(\Psi^{0}_{n\tau,h}\) denotes the backward numerical flow reconstructed from the stored electric field history through compositions of Störmer--Verlet steps. 
Subsequent developments extended NuFI to multi-species systems and Dirichlet boundary conditions \cite{Wilhelm2025PPCF}.
For flow-map methods based on direct initial-data pullback, the reconstructed solution can be evaluated at arbitrary phase-space points and inherits the bounds of the initial datum without limiters.
For Vlasov systems, implementations that reconstruct the map from stored field histories, such as NuFI, reduce storage costs by retaining field data in physical space instead of an evolving phase-space distribution.
Once the field history is available, backtracking from different points is independent and readily parallelizable \cite{KW2024}.

The present work develops a systematic CF-RKEI construction within this initial-data pullback framework.
Its two principal objects are the self-consistent exact frozen flow and the self-consistent numerical flow.
An exponential--pullback identity provides the basis for constructing the exact frozen flow from CF-RKEI.
Using a first-order scheme, we identify the local error caused by field
freezing and derive a one-step VP error budget that separates field and
characteristic errors. These local estimates guide the numerical
realization without claiming global PDE convergence.
For the Vlasov--Poisson system, we prove that the divergence-free frozen fields yield a self-consistent exact frozen flow that is a volume-preserving diffeomorphism of phase space under suitable regularity assumptions.

Section~\ref{sec:flows} develops the geometric setting for prescribed fields (\S\ref{subsec:prescribed-vector-fields})
and the pullback representation of linear transport (\S\ref{sec:pullback-representation});
it then formulates the distinction between prescribed and self-consistent fields and derives
the frozen exponential--pullback identity (\S\ref{subsec:self-consistent-fields}).
Section~\ref{sec:rkei-vlasov} first reviews CF-RKEI methods for ODEs (\S\ref{sec:rkei:construction})
and derives their flow-map representation (\S\ref{sec:NuF:SCfield}),
including a general local freezing-error identity.
It then specializes the construction to the Vlasov--Poisson system,
derives a one-step Lie--Euler error budget at the end of
\S\ref{subsec:vp-exact-frozen-flow}, and defines the self-consistent
numerical flow in \S\ref{subsec:vp-numerical-flow},
before analyzing the diffeomorphism and volume-preservation properties of the former (\S\ref{sec:CP}).
The numerical realization, computational cost, and storage requirements are also discussed.
Section~\ref{sec:Result} assesses the design orders through temporal self-convergence tests
and studies conservation behavior and the evolution of solution structures in long-time simulations for the Vlasov--Poisson system and the guiding-center model.

\section{Exact Flows and Pullback Representations for Transport Equations}
\label{sec:flows}

We recall the characteristic-flow notation and pullback representation used below, following \cite{Lee2013}.
We then distinguish prescribed transport fields from fields that depend on the transported quantity, as in the Vlasov--Poisson system.

Throughout this section, $M$ denotes a smooth finite-dimensional manifold
and $\mathfrak X(M)$ denotes the space of smooth vector fields on $M$. 
We use $\varphi_\sigma$ for the flow of an autonomous vector field, 
$\Phi_s^t$ for the flow of a time-dependent vector field,
and $\Phi_s^t[u]$ for the flow of a self-consistent vector field.

\subsection{Prescribed Vector Fields}
\label{subsec:prescribed-vector-fields}
The primary geometric objects associated with smooth vector fields are their integral curves, 
which are smooth curves whose velocity at each point is equal to the value of the vector field there.
We begin with the autonomous case, in which the vector field does not depend on time.
\begin{definition}[Integral curve of an autonomous vector field]
  \label{def:autonomous-integral-curve}
  Let $b\in\mathfrak X(M)$. A smooth curve $Z:J\to M$, $J\subset\mathbb R$, 
  is an \emph{integral curve} of $b$ if
  \begin{equation}
    \dot Z(\sigma) = b\big(Z(\sigma)\big), \qquad \sigma\in J.
    \label{eq:autonomous-ic}
  \end{equation}
\end{definition}

Another way to encode the family of integral curves associated with a vector field is through its flow.
\begin{definition}[Flow of an autonomous vector field]
\label{def:flow-of-b}
Let \(b\in\mathfrak X(M)\). A \emph{flow of \(b\)} (or a flow \emph{generated by} \(b\)) is a map
  \begin{equation*} 
    \varphi : \mathcal D \longrightarrow M, 
  \end{equation*}
  where \(\mathcal D\subset \mathbb R\times M\) is an open set, 
  called the \emph{flow domain}, such that for every \(p\in M\) the slice
  \begin{equation*}
    \mathcal D_p := \{\sigma\in\mathbb R : (\sigma,p)\in\mathcal D\} 
  \end{equation*}
  is an interval containing \(0\), and \(\varphi\) satisfies:
  \begin{enumerate}
    \item \(\varphi(0,p) = p\) for all \(p\in M\);
    \item for each \(p\in M\), the curve \(\sigma\mapsto \varphi(\sigma,p)\), \(\sigma\in\mathcal D_p\), 
    is an integral curve of \(b\) starting at \(p\).
  \end{enumerate}
  For fixed \(\sigma\in\mathbb R\),
  we write \(\varphi_\sigma : \{p\in M:(\sigma,p)\in\mathcal D\}\to M\) for the map \(\varphi_\sigma(p):=\varphi(\sigma,p)\);
  and for fixed \(p\in M\), the curve $\sigma \mapsto \varphi_\sigma(p)$, $\sigma \in \mathcal D_p$, 
  is precisely the integral curve \(Z\) of \(b\) with \(Z(0)=p\).
\end{definition}

By the Fundamental Theorem on flows \cite[Thm.\ 9.12]{Lee2013},
for each $p\in M$ the integral curve can be taken to be \emph{maximal}, i.e., 
defined on the largest possible interval \(\mathcal D_p\) on which it cannot be extended;
this maximality is what makes \(\varphi_\sigma(p)\),
and hence \(\mathcal D\) and \(\varphi\), uniquely determined.
The theorem further ensures that \(\mathcal D\) is open 
and \(\varphi\) is smooth.

The preceding vector field is autonomous because its value depends only on
the state \(z\). We now allow the vector field to depend on time.
\begin{definition}[Integral curve of a time-dependent vector field]
\label{def:td-ic}
A smooth time-dependent vector field on $M$ is a smooth map
  \begin{equation*}
    b:I\times M\to TM, \qquad (t,z)\mapsto b_t(z):=b(t,z),
  \end{equation*}
satisfying
$b_t(z)\in T_zM$. Given $s\in I$ and $z\in M$, a smooth curve
$Z(\cdot\,;s,z):J\to M$ with $s\in J\subset I$ is an \emph{integral curve
of $b$ through $z$ at time $s$} if
\begin{equation}
  \frac{d}{dt}Z(t;s,z) = b_t\big(Z(t;s,z)\big), \qquad Z(s;s,z)=z.
  \label{eq:td-ic}
\end{equation}
\end{definition}

A time-dependent vector field might not generate a one-parameter flow,
because two integral curves starting at the same point 
but at different times might follow different paths,
whereas all integral curves of a flow through a given point have the same image.
By the standard existence-uniqueness-maximality theory for time-dependent vector fields,
obtained by applying the autonomous flow theorem to the suspended field $\tilde{b}(t,z) := (1,b_t(z))$ on $I \times M$,
for each $(s,z) \in I \times M$ there is a unique maximal integral curve $Z(\cdot;s,z)$;
we denote its domain by $J_{s,z} \subset I$.

Allowing both the initial point $z$ and the initial time $s$ 
to vary assembles the maximal integral curves into a two-parameter family of flow maps. 
Define the natural flow domain 
\begin{equation}
  \mathcal D_b
    :=
  \left\{
    (s,t,z)\in I\times I\times M:
    t\in J_{s,z}
  \right\}
  \label{eq:two-parameter-flow-domain}
\end{equation}
and, for fixed \(s,t\in I\), define
\begin{equation}
  \mathcal D_{s,t}
    :=
  \left\{
    z\in M:
    (s,t,z)\in\mathcal D_b
  \right\}
  =
  \left\{
    z\in M:
    t\in J_{s,z}
  \right\}.
  \label{eq:two-parameter-flow-spatial-domain}
\end{equation}

\begin{definition}[Two-parameter flow]
  \label{def:two-parameter-flow}
  The \emph{two-parameter flow} generated by \(b\) is the smooth map
  \begin{equation}
    \Phi\colon\mathcal D_b\longrightarrow M,
    \qquad
    \Phi(s,t,z):=Z(t;s,z).
    \label{eq:two-parameter-flow}
  \end{equation}
  For fixed \(s,t\in I\), the corresponding flow map is
  \begin{equation}
    \Phi_s^t\colon\mathcal D_{s,t}\longrightarrow M,
    \qquad
    \Phi_s^t(z):=\Phi(s,t,z)=Z(t;s,z).
    \label{eq:flow-ode}
  \end{equation}
  Thus \(\Phi_s^t\) maps a state at time \(s\) to its state at time \(t\).
\end{definition}

For every \(z\in\mathcal D_{s,t}\), the curve
\[
  t\longmapsto\Phi_s^t(z)
\]
is the maximal integral curve through \(z\) at time \(s\). In particular,
\begin{equation}
  \frac{\partial}{\partial t}\Phi_s^t(z)
    =
  b_t\bigl(\Phi_s^t(z)\bigr),
  \qquad
  \Phi_s^s(z)=z.
  \label{eq:flow-evolution-equation}
\end{equation}

\begin{proposition}[Fundamental properties of the two-parameter flow]
  \label{prop:flow-properties}
  The two-parameter flow generated by \(b\) has the following properties.
  \begin{enumerate}
    \item The set \(\mathcal D_b\) is open in \(I\times I\times M\), and
          \(\Phi\colon\mathcal D_b\to M\) is smooth.

    \item For every \(s,t\in I\), the set \(\mathcal D_{s,t}\) is open in
          \(M\).

    \item For every \(s\in I\),
          \begin{equation}
            \Phi_s^s=\operatorname{Id}_M.
            \label{eq:two-parameter-flow-identity}
          \end{equation}

    \item Whenever the compositions are defined,
          \begin{equation}
            \Phi_r^t\circ\Phi_s^r
              =
            \Phi_s^t.
            \label{eq:two-parameter-flow-composition}
          \end{equation}

    \item For every \(s,t\in I\), the map
          \[
            \Phi_s^t\colon
            \mathcal D_{s,t}\longrightarrow\mathcal D_{t,s}
          \]
          is a diffeomorphism, with
          \begin{equation}
            \bigl(\Phi_s^t\bigr)^{-1}
              =
            \Phi_t^s.
            \label{eq:two-parameter-flow-inverse}
          \end{equation}
  \end{enumerate}
\end{proposition}

\begin{proof}
  Openness of the natural domains and smoothness of \(\Phi\) follow from
  smooth dependence of the suspended autonomous flow on its initial data.

  The identity property follows directly from
  \(Z(s;s,z)=z\). For the composition property, let
  \(q=\Phi_s^r(z)\). The curves
  \[
    t\longmapsto\Phi_s^t(z)
    \qquad\text{and}\qquad
    t\longmapsto\Phi_r^t(q)
  \]
  pass through the same point \(q\) at the same time \(r\).
  Uniqueness therefore gives
  \[
    \Phi_r^t\bigl(\Phi_s^r(z)\bigr)=\Phi_s^t(z).
  \]
  Taking \(t=s\) shows that
  \(\Phi_t^s\circ\Phi_s^t=\operatorname{Id}\), and interchanging \(s\)
  and \(t\) gives the reverse identity. Hence \(\Phi_s^t\) is a
  diffeomorphism with inverse \(\Phi_t^s\).
\end{proof}

If $b_t\equiv b$ is autonomous with flow $\varphi_\sigma$, then
$\Phi_s^t=\varphi_{t-s}$: an autonomous flow depends only on the elapsed
time $t-s$, whereas a genuinely time-dependent flow generally depends on
$s$ and $t$ separately.

\subsection{Pullback Representation of Linear Transport}
\label{sec:pullback-representation}

We next connect the geometric flow with the evolution of a transported
scalar function. For $g\in C^1(M)$, the Lie derivative of $g$ along $b_t$
is $(\mathcal L_{b_t}g)(z):=dg_z(b_t(z))$, i.e.\ the directional derivative
$b_t(z)\cdot\nabla g(z)$ in coordinates. Consider the linear transport
equation
\begin{equation}
  \partial_t u(t,\cdot) + \mathcal L_{b_t}u(t,\cdot) = 0, \qquad
  u(s,\cdot)=u_s,
  \label{eq:transport}
\end{equation}
which is linear because $b_t$ is prescribed independently of $u$.

\begin{proposition}[Pullback representation]
\label{prop:pullback}
Let $\Phi_s^t$ be the flow generated by $b$. The classical solution of
\eqref{eq:transport} is
\begin{equation}
  u(t,z) = u_s\big(\Phi_t^s(z)\big), \qquad z\in \mathcal D_{t,s}.
  \label{eq:pullback}
\end{equation}
\end{proposition}

\begin{proof}
Fix $z_0\in \mathcal D_{s,t}$ and set $Z(t):=\Phi_s^t(z_0)$. The chain rule
and \eqref{eq:flow-ode} give
\[
  \tfrac{d}{dt}u\big(t,\Phi_s^t(z_0)\big)
    = \big(\partial_t u+\mathcal L_{b_t}u\big)\big(t,\Phi_s^t(z_0)\big)=0.
\]
Hence $u\big(t,\Phi_s^t(z_0)\big)=u_s(z_0)$. Substituting
$z_0=\Phi_t^s(z)$ gives \eqref{eq:pullback}.
\end{proof}

We record separately the specialization of Proposition~\ref{prop:pullback}
to an autonomous field, since this is the single fact to which all
frozen-field constructions of the paper reduce. When $b$ is autonomous,
let $\varphi_\sigma$ denote its flow, so that $\Phi_s^t=\varphi_{t-s}$, 
and define the transport generator
$C_b := -\mathcal L_b$.

\begin{proposition}[Autonomous exponential-pullback identity]
\label{prop:autonomous-exponential-pullback}
Let $b\in\mathfrak X(M)$ be complete and autonomous with flow $\varphi_\sigma$, and let
$C_b:=-\mathcal L_b$. For every $h\in\mathbb R$ and every $g\in C^1(M)$,
\begin{equation}
  \exp(hC_b)\,g = g\circ\varphi_{-h}.
  \label{eq:autonomous-exponential-pullback}
\end{equation}
That is, the exponential of an autonomous transport operator acts by
pullback along the corresponding backward characteristic flow.
\end{proposition}

\begin{proof}
Fix $g\in C^1(M)$ and apply Proposition~\ref{prop:pullback} to the
autonomous case $b_t\equiv b$ with initial time $s=0$ and initial data
$u_0=g$: the classical solution of \eqref{eq:transport} is
\[
  u(t,z) = g\big(\Phi_t^0(z)\big) = g\big(\varphi_{-t}(z)\big),
  \qquad t\in\mathbb R,
\]
using $\Phi_t^0=\varphi_{0-t}=\varphi_{-t}$. Set $w(t):=u(t,\cdot) =
g\circ\varphi_{-t}$. By \eqref{eq:transport}, $w$ solves
\[
  \frac{d}{dt}w(t) = -\mathcal L_b\,w(t) = C_b\,w(t), \qquad w(0)=g,
\]
which is precisely the defining initial value problem for the operator
exponential: $\exp(tC_b)g$ is, by definition, the unique solution $v(t)$
of $\dot v(t)=C_b v(t)$, $v(0)=g$. By uniqueness, $w(t)=\exp(tC_b)g$, i.e.
$g\circ\varphi_{-t}=\exp(tC_b)g$ for every $t\in\mathbb R$. Relabelling
$t=h$ gives \eqref{eq:autonomous-exponential-pullback}.
\end{proof}

Here $\exp(tC_b)$ denotes the transport solution operator determined by characteristics, 
rather than a power series of bounded operators on an unspecified Banach space.
If $b$ is not complete, the identity is restricted to the points
and time intervals on which the corresponding characteristics exist.

Proposition~\ref{prop:autonomous-exponential-pullback} is a classical identity used in splitting and composition
methods \cite{HLW2006,MQ2002}. 
We include the proof to fix the sign and composition conventions for Section~\ref{sec:rkei-vlasov}.

\subsection{Prescribed versus Self-Consistent Vector Fields}
\label{subsec:self-consistent-fields}
Sections~\ref{subsec:prescribed-vector-fields}--\ref{sec:pullback-representation} recall, 
in a form adapted to our later notation, 
the classical theory of flows generated by autonomous and prescribed time-dependent vector fields (see e.g.\ \cite{Lee2013}); 
no originality is claimed there. Section~\ref{subsec:self-consistent-fields} identifies why this classical theory does not directly apply once the field is generated by the transported function itself, 
and introduces the frozen-field device around which the paper's numerical method, developed in Section~\ref{sec:rkei-vlasov}, is built.

The preceding construction assumes that $b_t$ is known independently of
the transported function. We now distinguish this setting from nonlinear
transport equations in which the vector field is determined by the
solution itself.

Let $\mathcal Y$ be a suitable class of scalar functions on $M$ and let
\begin{equation}
  \mathcal B\colon\mathcal Y\longrightarrow\mathfrak X(M),
  \qquad
  u\longmapsto b[u]:=\mathcal B(u),
  \label{eq:field-functional}
\end{equation}
be a vector-field-valued functional. 
The corresponding self-consistent transport equation is
\begin{equation}
  \partial_t u(t,\cdot) + \mathcal L_{b[u(t)]}u(t,\cdot) = 0.
  \label{eq:self-consistent-transport}
\end{equation}
For fixed $u\in\mathcal Y$, define the linear operator
\begin{equation}
  C(u)g
    :=
  -\mathcal L_{b[u]}g.
  \label{eq:self-consistent-operator}
\end{equation}
Equation~\eqref{eq:self-consistent-transport} can then be written in
the abstract form
\begin{equation}
  \frac{\mathrm d}{\mathrm dt}u(t)
    =
  C\bigl(u(t)\bigr)u(t).
  \label{eq:nonlinear-operator-evolution}
\end{equation}
Although $C(u)$ is linear in the function on which it acts once $u$
has been fixed, the map $u\mapsto C(u)$ is generally nonlinear.

Along an exact solution $u(t)$, define
\begin{equation}
  b_t^u
    :=
  b[u(t)].
  \label{eq:self-consistent-induced-field}
\end{equation}
This is a self-consistent vector field on $M$, and we denote its exact
flow by
\begin{equation}
  \frac{\partial}{\partial t}\Phi_s^t[u](z)
    =
  b[u(t)]\bigl(\Phi_s^t[u](z)\bigr),
  \qquad
  \Phi_s^s[u](z)=z.
  \label{eq:self-consistent-flow}
\end{equation}
Applying Proposition~\ref{prop:pullback} along the exact
solution gives
\begin{equation}
  u(t,z)
    =
  u\bigl(s,\Phi_t^s[u](z)\bigr).
  \label{eq:self-consistent-pullback}
\end{equation}
When no ambiguity can arise, we suppress the explicit dependence on $u$
and write simply $\Phi_s^t$.

\begin{remark}[Implicit nature of the self-consistent representation]
\label{rem:non-closure}
The representation \eqref{eq:self-consistent-pullback} is exact but implicit: 
evaluating $\Phi_t^s[u]$ requires the solution-dependent field $b[u(\tau)]$ between $s$ and $t$. 
Freezing this field at stage approximations yields the autonomous subproblems used in Section~\ref{sec:rkei-vlasov}.
\end{remark}

To obtain computable subproblems, fix a state $u_\star\in\mathcal Y$ and
define the frozen autonomous vector field and its generator by
\begin{equation}
  b_\star
    :=
  b[u_\star],
  \qquad
  C_\star
    :=
  C(u_\star)
    =
  -\mathcal L_{b_\star}.
  \label{eq:frozen-state-vector-field}
\end{equation}
Let $\varphi_\sigma^{b_\star}$ be the autonomous flow generated by
$b_\star$:
\begin{equation}
  \frac{\mathrm d}{\mathrm d\sigma}
    \varphi_\sigma^{b_\star}(z)
    =
  b_\star
    \bigl(\varphi_\sigma^{b_\star}(z)\bigr),
  \qquad
  \varphi_0^{b_\star}(z)=z.
  \label{eq:frozen-state-forward-flow}
\end{equation}
For a signed flow increment \(\delta\in\mathbb{R}\), define the associated
backward map by
\begin{equation}
  \zeta_\delta^{b_\star}
    :=
  \varphi_{-\delta}^{b_\star}.
  \label{eq:frozen-state-backward-flow}
\end{equation}
For \(\delta>0\), this is equivalently the terminal-value map
\begin{equation}
  \begin{aligned}
    \frac{\mathrm d}{\mathrm d\sigma}Z(\sigma)
      &=
    b_\star\bigl(Z(\sigma)\bigr),
    \qquad \sigma\in[0,\delta],\\
    Z(\delta)&=z,
    \qquad
    \zeta_\delta^{b_\star}(z)=Z(0).
  \end{aligned}
  \label{eq:frozen-state-terminal-value-flow}
\end{equation}

Once $u_\star$ is fixed, $b_\star$ is autonomous and its characteristics can be computed
independently of the unknown solution trajectory. The following corollary
identifies its transport exponential with the pullback along the backward flow.

\begin{corollary}[Frozen exponential-pullback identity]
\label{prop:frozen-pullback}
For every \(u_\star\in\mathcal Y\) and every \(\delta\in\mathbb R\) for
which \(\zeta_\delta^{b_\star}\) is defined,
\begin{equation}
  \exp\bigl(\delta\, C(u_\star)\bigr)g
    =
  g\circ\zeta_\delta^{b_\star}
  \label{eq:frozen-exponential-pullback}
\end{equation}
for every \(g\in C^1(M)\).
\end{corollary}

\begin{proof}
By construction \(C(u_\star)=C_\star=-\mathcal L_{b_\star}\) is the
transport generator of the \emph{autonomous} vector field \(b_\star\).
The autonomous case of the pullback representation
(equivalently~\eqref{eq:autonomous-exponential-pullback}, applied to
\(b_\star\) with flow \(\varphi_\sigma^{b_\star}\)) gives, for any time
increment \(\delta\),
\[
  \exp(\delta C_\star)g = g\circ\varphi_{-\delta}^{b_\star}.
\]
The claim follows from the definition
\(\zeta_\delta^{b_\star}:=\varphi_{-\delta}^{b_\star}\) in
\eqref{eq:frozen-state-backward-flow}.
\end{proof}

Corollary~\ref{prop:frozen-pullback} converts each frozen transport exponential into evaluation along a backward characteristic map. 
For CF-RKEI factors containing linear combinations of stage generators, 
the associated field is the same linear combination of the stage fields, 
by linearity of the Lie derivative in the vector field.

The role of the CF-RKEI is not to solve an already prescribed time-dependent characteristic equation.
Instead, it approximates the nonlinear, self-consistent evolution
\eqref{eq:nonlinear-operator-evolution} by composing exponentials of
operators frozen at suitable stage approximations; 
by Corollary~\ref{prop:frozen-pullback}, each such frozen exponential acts by pullback along an exact backward characteristic map.
Section~\ref{sec:rkei-vlasov} develops this composition at the operator level 
and specializes it to the self-consistent field of the Vlasov dynamics.

\section{A Self-consistent Exact Frozen Flow and its numerical realization}
\label{sec:rkei-vlasov}
CF-RKEI provide a systematic way to construct time-stepping
schemes from exponentials of frozen operators. For nonlinear transport,
each frozen transport exponential can be defined directly through its
characteristic flow. This gives a flow-based construction without first
discretizing the transport operator in space.

The construction and its convergence analysis are separate questions.
The ODE order conditions guide the choice of coefficients; they do not
by themselves establish convergence for a transport PDE or its fully
discrete realization. Our analysis has a more limited purpose: we
identify the error introduced by freezing a field and, for the
Vlasov--Poisson system, separate the effects of field approximation and
characteristic integration in a single step. These estimates explain
accuracy requirements without asserting global PDE convergence.

Section~\ref{sec:rkei:construction} recalls the ODE construction, and
Section~\ref{sec:NuF:SCfield} gives its characteristic-flow formulation
and a general local freezing-error estimate.
Section~\ref{sec:SC-NuF:VP} specializes the construction to
Vlasov--Poisson, derives a one-step error budget, and describes
the numerical realization.
Section~\ref{sec:CP} establishes diffeomorphism and volume-preservation
properties of the exact frozen flows and their compositions.

Unless otherwise stated, we consider a finite time interval
$I=[0,T]$ and denote the time step by $\tau$ in this section.
\subsection{A commutator-free Runge-Kutta exponential integrators approximation}
\label{sec:rkei:construction}
Consider the finite-dimensional nonlinear ODE
\begin{equation}
  \frac{dY(t)}{dt}=C(Y(t))Y(t),
  \qquad Y(0)=Y_0,
  \label{eq:ode}
\end{equation}
where $C$ is a matrix-valued function of the state.
The Lie--Euler method \cite{CrouchGrossman1993} freezes $C$ at the
current numerical state and advances the solution by
\begin{equation}
  Y^{n+1}=\exp\bigl(\tau C(Y^n)\bigr)Y^n.
  \label{eq:lie-euler-ode-update}
\end{equation}
Here the matrix is frozen during a step, while its dependence on the
solution is retained through the value at the beginning of that step.

Higher-order approximations can be constructed using
commutator-free Runge--Kutta exponential integrators
(CF-RKEI), as described below.

\begin{definition}[CF-RKEI tableau]\label{def:tableau}
A CF-RKEI method with nodes
\[
0=c_1<c_2<\cdots<c_s\le 1
\]
is represented by the tableau
\begin{equation}
  \begin{array}{c|cccc}
    0 & & & & \\
    c_2 & \boldsymbol\alpha_{2,1} & \boldsymbol\alpha_{2,2} & \cdots & \boldsymbol\alpha_{2,J_2} \\
    c_3 & \boldsymbol\alpha_{3,1} & \boldsymbol\alpha_{3,2} & \cdots & \boldsymbol\alpha_{3,J_3} \\
    \vdots & \vdots & \vdots & & \vdots \\
    c_s & \boldsymbol\alpha_{s,1} & \boldsymbol\alpha_{s,2} & \cdots & \boldsymbol\alpha_{s,J_s} \\
    \hline
    & \boldsymbol\beta_{1} & \boldsymbol\beta_{2} & \cdots & \boldsymbol\beta_{J}
  \end{array}
  \label{eq:general-tableau}
\end{equation}
where
\[
\boldsymbol{\alpha}_{r,j}
:=
\bigl(\alpha^1_{r,j},\ldots,\alpha^{r-1}_{r,j}\bigr)
\in\mathbb{R}^{r-1},
\qquad
\boldsymbol{\beta}_{j}
:=
\bigl(\beta^1_j,\ldots,\beta^s_j\bigr)
\in\mathbb{R}^{s}.
\]
Here \(J_r\ge1\) is the number of exponential factors used in stage \(r\),
and \(J\ge1\) is the number used in the final update. The coefficients
satisfy the consistency conditions
\begin{align*}
  \left\{
    \begin{aligned}
        &\sum_{j=1}^{J_r}\sum_{k=1}^{r-1}\alpha^k_{r,j}=c_r,\quad r=2,\ldots,s,\\
        &\sum_{j=1}^{J}\sum_{k=1}^{s}\beta^k_j=1.
    \end{aligned}
  \right.
\end{align*}
For design order \(p\), the coefficients \(\alpha^k_{r,j}\) and
\(\beta^k_j\) are required to satisfy the corresponding CF-RKEI
order conditions (see e.g.\ \cite{CMO2003,Owren2006,CK2009}).
\end{definition}

\begin{algorithm}[htbp]
\caption{The CF-RKEI method.}
\label{alg:RKEI}
\begin{algorithmic}[1]
\STATE $Y_1 \leftarrow Y^n$
\FOR{$r = 2,\dots,s$}
    \STATE
    $Y_r \leftarrow
    \exp\Bigl(
        \tau\textstyle\sum_{k=1}^{r-1}
        \alpha^k_{r,J_r} C(Y_k)
    \Bigr)
    \cdots
    \exp\Bigl(
        \tau\textstyle\sum_{k=1}^{r-1}
        \alpha^k_{r,1} C(Y_k)
    \Bigr)Y^n$
\ENDFOR
\STATE
$Y^{n+1} \leftarrow
\exp\Bigl(
    \tau\textstyle\sum_{k=1}^{s}
    \beta^k_{J} C(Y_k)
\Bigr)
\cdots
\exp\Bigl(
    \tau\textstyle\sum_{k=1}^{s}
    \beta^k_{1} C(Y_k)
\Bigr)Y^n$
\end{algorithmic}
\end{algorithm}

With the above definition, the CF-RKEI method is given by Algorithm~\ref{alg:RKEI}, 
which computes the stage values \(Y_r\) and the final update \(Y^{n+1}\) using the exponential of linear combinations of the matrices \(C(Y_k)\). 

\subsection{Construction of the self-consistent exact frozen flow}
\label{sec:NuF:SCfield}
We define each frozen transport exponential by its exact characteristic
flow, using Corollary~\ref{prop:frozen-pullback}. Thus the exponential
notation below denotes the transport evolution operator, rather than
a matrix exponential or a power series of bounded operators.
This distinction matters because $C(u)=-\mathcal L_{b[u]}$
differentiates its argument and is generally unbounded on $L^\infty$.

Related commutator-free schemes, including nonlinear Vlasov applications,
have been developed
in \cite{Celledoni2005,CK2009,CBQ2021}.
Here we use the characteristic-flow representation to construct the
method and to give local error estimates. Global stability and
convergence of the coupled PDE scheme are beyond these estimates.
For comparison, convergence analyses for particular Vlasov--Poisson
splitting schemes are available in
\cite{BesseMehrenberger2008,EinkemmerOstermann2014Strang,EinkemmerOstermann2014DG};
those results concern different schemes and do not directly establish
convergence of the present construction.

To rewrite the CF-RKEI scheme in terms of flows, 
we associate each frozen operator appearing in the CF-RKEI composition 
with the flow generated by the corresponding frozen vector field. 
More precisely, for the \(j\)th exponential in stage \(r\), define
\begin{equation}\label{eq:r-stage}
    b_{r,j}
    :=
    \sum_{k=1}^{r-1}
    \alpha_{r,j}^k\, b[u_k],
    \qquad
    j=1,\ldots,J_r,
\end{equation}
and, for the final update, define
\begin{equation}\label{eq:final-stage}
    b_j
    :=
    \sum_{k=1}^{s}
    \beta_j^k\, b[u_k],
    \qquad
    j=1,\ldots,J,
\end{equation}
where \(b[u_k]\) denotes the frozen vector field associated with the stage value \(u_k\).
Hence, by Corollary~\ref{prop:frozen-pullback},
\begin{equation*}
    \exp\left(
        \tau
        \sum_{k=1}^{r-1}
        \alpha_{r,j}^k C(u_k)
    \right)g
    =
    g\circ\zeta_\tau^{b_{r,j}},
\end{equation*}
for every sufficiently regular \(g\).
Likewise,
\begin{equation*}
    \exp\left(
        \tau
        \sum_{k=1}^{s}
        \beta_j^k C(u_k)
    \right)g
    =
    g\circ\zeta_\tau^{b_j}.
\end{equation*}
Moreover, care must be taken with the order of composition.
Writing $C_j:=-\mathcal L_{b_j}$, the operator product
\begin{equation*}
    \exp(\tau C_{J})
    \cdots
    \exp(\tau C_{1})g
\end{equation*}
corresponds to
\begin{equation*}
    g\circ
    \zeta_\tau^{b_1}
    \circ\cdots\circ
    \zeta_\tau^{b_J}.
\end{equation*}
This reversal is a direct consequence of the contravariant nature of pullbacks. 
Indeed, for two frozen vector fields \(b_1\) and \(b_2\),
\begin{equation*}
  \begin{aligned}
      \exp(\tau C_2)\exp(\tau C_1)g
      &=
      \exp(\tau C_2)
      \bigl(g\circ\zeta_\tau^{b_1}\bigr) \\
      &=
      g\circ
      \zeta_\tau^{b_1}
      \circ
      \zeta_\tau^{b_2}.
  \end{aligned}
\end{equation*}
Applying these identities successively to each exponential factor yields the flow representation
\begin{equation}
    u_r
    =
    u^n
    \circ
    \zeta_\tau^{b_{r,1}}
    \circ\cdots\circ
    \zeta_\tau^{b_{r,J_r}},
    \qquad r=2,\ldots,s,
\end{equation}
and
\begin{equation}
    u^{n+1}
    =
    u^n
    \circ
    \zeta_\tau^{b_1}
    \circ\cdots\circ
    \zeta_\tau^{b_J}.
\end{equation}
Thus, each commutator-free exponential is realized as the pullback induced by a flow associated with a frozen linear combination of stage vector fields.
Consequently, although the pullbacks are written from left to right as
\(\zeta_\tau^{b_1},\ldots,\zeta_\tau^{b_J}\), their action on a point
\(z\in M\) is evaluated from right to left:
\begin{equation*}
    u^{n+1}(z)
    =
    u^n\left(
        \zeta_\tau^{b_1}
        \circ\cdots\circ
        \zeta_\tau^{b_J}(z)
    \right).
\end{equation*}

\begin{algorithm}[htbp]
\caption{The CF-RKEI method in flow form.}
\label{alg:RKEI-flow}
\begin{algorithmic}[1]

\STATE $u_1 \leftarrow u^n$

\FOR{$r=2,\dots,s$}

    \FOR{$j=1,\dots,J_r$}
        \STATE
        $b_{r,j}
        \leftarrow
        \displaystyle
        \sum_{k=1}^{r-1}
        \alpha_{r,j}^k\, b[u_k]$
    \ENDFOR

    \STATE
    $u_r
    \leftarrow
    u^n
    \circ
    \zeta_\tau^{b_{r,1}}
    \circ\cdots\circ
    \zeta_\tau^{b_{r,J_r}}$

\ENDFOR

\FOR{$j=1,\dots,J$}
    \STATE
    $b_j
    \leftarrow
    \displaystyle
    \sum_{k=1}^{s}
    \beta_j^k\,b[u_k]$
\ENDFOR

\STATE
$u^{n+1}
\leftarrow
u^n
\circ
\zeta_\tau^{b_1}
\circ\cdots\circ
\zeta_\tau^{b_J}$

\end{algorithmic}
\end{algorithm}

We next estimate the local error caused by replacing a varying
transport field with a prescribed frozen field. The estimate does not
require a Lipschitz bound for the map $u\mapsto b[u]$ on a set
containing all numerical iterates.

Consider
\begin{equation}
  \partial_tu+b[u]\cdot\nabla u=0
  \label{eq:lie-euler-transport}
\end{equation}
on $M=\mathbb T^\ell\times\mathbb R^{m-\ell}$, equipped with its
flat metric. Write $\|\cdot\|_\infty=\|\cdot\|_{L^\infty(M)}$.

\begin{lemma}
\label{lem:lie-euler-flow-defect}
Let $u$ be a classical solution of \eqref{eq:lie-euler-transport}
on $[t_n,t_n+\tau]$, with bounded spatial gradient.
Let $c$ be a prescribed $C^1$ vector field with a complete flow,
and suppose the integrand below is integrable in the supremum norm.
Define
\begin{equation}
  d_c^{n+1}:=u(t_n+\tau)-u(t_n)\circ\zeta_\tau^c.
  \label{eq:lie-euler-flow-defect}
\end{equation}
Then
\begin{equation}
  d_c^{n+1}
  =-\int_0^\tau
    \Bigl[(b[u(t_n+s)]-c)\cdot\nabla u(t_n+s)\Bigr]
      \circ\zeta_{\tau-s}^c\,ds,
  \label{eq:lie-euler-flow-defect-identity}
\end{equation}
and consequently
\begin{equation}
  \|d_c^{n+1}\|_\infty
  \leq\int_0^\tau
    \|(b[u(t_n+s)]-c)\cdot\nabla u(t_n+s)\|_\infty\,ds.
  \label{eq:lie-euler-flow-local-bound}
\end{equation}
In particular, for $c=b[u(t_n)]$, if
$\|b[u(t_n+s)]-b[u(t_n)]\|_\infty\leq H_n s$ and
$G_n^*:=\sup_{0\leq s\leq\tau}\|\nabla u(t_n+s)\|_\infty$,
then $\|d_c^{n+1}\|_\infty\leq G_n^*H_n\tau^2/2$.
\end{lemma}

\begin{proof}
For a fixed terminal point $z$, set
$\gamma(s)=\zeta_{\tau-s}^c(z)$. Then
$\dot\gamma(s)=c(\gamma(s))$, and the transport equation gives
\[
  \frac{d}{ds}u(t_n+s,\gamma(s))
  =-\bigl[(b[u(t_n+s)]-c)\cdot\nabla u(t_n+s)\bigr](\gamma(s)).
\]
Integration gives \eqref{eq:lie-euler-flow-defect-identity}.
Taking suprema gives \eqref{eq:lie-euler-flow-local-bound}.
The final estimate follows by integrating $G_n^*H_n s$.
\end{proof}

The integral estimate does not require time-Lipschitz continuity of the
field. That additional condition is only one way to obtain an
$O(\tau^2)$ one-step freezing defect. This is a local consistency
statement; controlling the accumulation of such defects requires
additional analysis of the self-consistent coupling.
We apply the identity after specifying the field in
Section~\ref{subsec:vp-exact-frozen-flow}.

\subsection{A self-consistent exact frozen flow and numerical flow for the Vlasov--Poisson system}
\label{sec:SC-NuF:VP}
We now specialize the general construction to the Vlasov--Poisson system.
We first define the corresponding self-consistent exact frozen
flow and derive a one-step error budget for its Lie--Euler
specialization. We then describe the field reconstruction and
characteristic integration that yield the self-consistent numerical flow.
The guiding-center model can be treated analogously within the
same framework and is considered in the numerical experiments
in Section~\ref{sec:Result}.

\subsubsection{A self-consistent exact frozen flow for the Vlasov--Poisson system}
\label{subsec:vp-exact-frozen-flow}
Let $q\geq1$ be an integer. For an electric field
$E[f]\in C^q(\mathbb T^d;\mathbb R^d)$,
define the phase-space vector field and transport operator
\begin{align}\label{eq:CE}
  &b^{E[f]}(x,v) := \begin{pmatrix} v \\ -E[f](x) \end{pmatrix} \in \mathfrak X(M),\\
  &C(f) := -\mathcal L_{b^{E[f]}} = -v\cdot\nabla_x + E[f](x)\cdot\nabla_v,
\end{align}
where the operator's nonlinearity is confined to the
dependence of the field $E[f]$ on $f$ through \eqref{eq:poisson}.
For a fixed $\mathcal E$, the vector field is autonomous, i.e.,
\begin{align*}
  &b^{\mathcal E}(x,v) := \begin{pmatrix} v \\ -\mathcal E(x) \end{pmatrix} 
\end{align*}
We write $\xi^{\mathcal E}_c$, $c \in \mathbb R\setminus\{0\}$, 
for elementary backward map of
\begin{equation}
  \frac{dX}{d\sigma}(\sigma) = c\,V(\sigma), \qquad
  \frac{dV}{d\sigma}(\sigma) = -c\,\mathcal{E}\big(X(\sigma)\big),\qquad \sigma\in[0,\tau],
  \label{eq:frozen-characteristics}
\end{equation}
i.e.\ $\xi^{\mathcal E}_c(z):=(X(0),V(0))$ for $(X(\tau),V(\tau))=z$.
The case $c=0$ is trivial:
both right-hand sides in \eqref{eq:frozen-characteristics} vanish,
and we therefore set $\xi^{\mathcal E}_0=\operatorname{Id}$.

\begin{lemma}
\label{lem:rate-interval}
Let $\varphi^{b^{\mathcal E}}_\delta$ denote the autonomous flow 
generated by $b^{\mathcal E}$ with an increment $\delta=c\tau$,
and let $\xi^{\mathcal E}_c$ denote the elementary frozen backward map
determined by the problem \eqref{eq:frozen-characteristics}.
Let $\zeta^{b^{\mathcal E}}_{\delta} :=
\varphi^{b^{\mathcal E}}_{-\delta}$ denote the standard backward map. 
Then
\begin{equation}
  \xi^{\mathcal E}_c = \zeta^{b^{\mathcal E}}_{\delta}.
  \label{eq:rate-interval-equivalence}
\end{equation}
That is, absorbing the constant $c$ into the vector field on the fixed
interval $[0,\tau]$, as in \eqref{eq:frozen-characteristics}, produces exactly the same
map as absorbing it into the length of a backward flow increment applied
to the unscaled field $b^{\mathcal E}$.
\end{lemma}

\begin{proof}
Let $\varphi_s^{b^{\mathcal E}}$ denote the complete flow
generated by $b^{\mathcal E}$.
For any terminal point $z$, the curve
\[
  Z_c(\sigma;z)
  :=
  \varphi_{c(\sigma-\tau)}^{b^{\mathcal E}}(z)
\]
satisfies
$\dot Z_c=c\,b^{\mathcal E}(Z_c)$ and $Z_c(\tau;z)=z$.
By uniqueness, this is the characteristic curve defining
$\xi_c^{\mathcal E}$.
Hence
\[
  \xi_c^{\mathcal E}(z)
  =
  \varphi_{-c\tau}^{b^{\mathcal E}}(z)
  =
  \zeta_{c\tau}^{b^{\mathcal E}}(z).
\]
\end{proof}

We now identify each factor of Algorithm~\ref{alg:RKEI-flow}, applied to
$u=f$. We restrict the Vlasov--Poisson construction below to CF-RKEI
tableaux for which the sum of the coefficients in every exponential
factor is nonzero. All CF-RKEI tableaux used in this work satisfy
this restriction.
For stage $r=2,\dots,s$ and factor $j=1,\dots,J_r$, assume
$a^\alpha_{r,j} := \sum_{k=1}^{r-1}\alpha^k_{r,j} \neq 0$ and set
\begin{equation}
  \bar E_{r,j} := \big(a^\alpha_{r,j}\big)^{-1}\sum_{k=1}^{r-1}\alpha^k_{r,j}\,E^{(k-1)},
  \qquad
  \xi_{r,j} := \xi^{\bar E_{r,j}}_{a^\alpha_{r,j}},
\end{equation}
and define the $r$-th stage composition 
\begin{equation}
  \Xi_r := \xi_{r,1}\circ\cdots\circ\xi_{r,J_r}.
\end{equation}
For the final stage and factor $j=1,\cdots, J$, assume
$a^\beta_{j} := \sum_{k=1}^{s}\beta^k_{j} \neq 0$ and set
\begin{equation}
  \bar E^\beta_j := \big(a^\beta_j\big)^{-1}\sum_{k=1}^{s}\beta^k_j\,E^{(k-1)},
  \qquad
  \Psi^{n\tau}_{(n+1)\tau} := \xi^{\bar E^\beta_{1}}_{a^\beta_1} \circ\cdots\circ \xi^{\bar E^\beta_{J}}_{a^\beta_J}.
  \label{eq:finalmap}
\end{equation}
so that, the $r$-th stage update in
Algorithm~\ref{alg:RKEI-flow} is $f_r = f^n\circ\Xi_r$ and final update is $f^{n+1} = f^n\circ\Psi^{n\tau}_{(n+1)\tau}$. 
Starting from $f^0=f_0$ and $\Psi^0_0:=\operatorname{Id}$,
define the history map recursively by
\[
  \Psi^0_{(n+1)\tau}
  :=
  \Psi^0_{n\tau}\circ\Psi^{n\tau}_{(n+1)\tau}.
\]
Then $f^{n+1}=f_0\circ\Psi^0_{(n+1)\tau}$.
The stage and accepted fields are determined self-consistently by
$E^{(0)}=E^n=E[f^n]$,
$E^{(r-1)}=E[f_r]$ for $r=2,\ldots,s$,
and $E^{n+1}=E[f^{n+1}]$.

Before describing the numerical realization, we use
Lemma~\ref{lem:lie-euler-flow-defect} to examine the Lie--Euler
specialization. Its exact frozen update is
\begin{equation}
  f^{n+1}=f^n\circ\xi_1^{E[f^n]}.
  \label{eq:vp-lie-euler-frozen-update}
\end{equation}
We next allow an approximate input distribution, a prescribed frozen
electric field, and an approximate backward map. Their construction
is left unspecified here; the following estimate prepares the
numerical realization in Section~\ref{subsec:vp-numerical-flow}.
Since freezing leaves the kinetic component $v$ unchanged, the
freezing defect involves only the electric-field difference and
the velocity gradient of the distribution.

\begin{proposition}
\label{prop:vp-one-step-budget}
Let $f$ be a classical solution of \eqref{eq:vp} on
$[t_n,t_n+\tau]$, with bounded spatial and velocity gradients.
Let $g_n$ be a bounded approximation of $f(t_n)$ on
$M=\mathbb T^d\times\mathbb R^d$, and let
$\widetilde E_n\in C^1(\mathbb T^d;\mathbb R^d)$ be any
prescribed frozen electric field. Set
\[
  c_n:=b^{\widetilde E_n},\qquad
  e_n=\|g_n-f(t_n)\|_\infty,
\]
\[
  G_n=\|\nabla_{x,v}f(t_n)\|_\infty,\qquad
  G_{v,n}=\sup_{0\leq s\leq\tau}
                  \|\nabla_v f(t_n+s)\|_\infty.
\]
Let $K\subset M$ be a bounded evaluation set, and let
$\widehat X_n:K\to M$ approximate the exact frozen backward map
$\xi_1^{\widetilde E_n}=\zeta_\tau^{c_n}$, with
\[
  \eta_n=\sup_{z\in K}
    d_M\bigl(\widehat X_n(z),\zeta_\tau^{c_n}(z)\bigr)<\infty.
\]
Here $d_M$ is the flat-metric distance, and
$\|h\|_{\infty,K}:=\sup_{z\in K}|h(z)|$.
Then the update $\widehat f^{n+1}=g_n\circ\widehat X_n$ satisfies
\begin{equation}
  \|f(t_n+\tau)-\widehat f^{n+1}\|_{\infty,K}
  \leq e_n+
    G_{v,n}\int_0^\tau
       \|E[f(t_n+s)]-\widetilde E_n\|_{L^\infty(\mathbb T^d)}\,ds
    +G_n\eta_n.
  \label{eq:vp-one-step-budget}
\end{equation}
If, additionally,
\[
  \|E[f(t_n+s)]-E[f(t_n)]\|_{L^\infty(\mathbb T^d)}\leq H_n s,
  \qquad 0\leq s\leq\tau,
\]
and
$\varepsilon_n:=\|\widetilde E_n-E[f(t_n)]\|_{L^\infty(\mathbb T^d)}$,
then
\begin{equation}
  \|f(t_n+\tau)-\widehat f^{n+1}\|_{\infty,K}
  \leq e_n+G_{v,n}\bigl(H_n\tau^2/2+\tau\varepsilon_n\bigr)
            +G_n\eta_n.
  \label{eq:vp-one-step-budget-lipschitz}
\end{equation}
\end{proposition}

\begin{proof}
Since $\widetilde E_n$ is $C^1$ and periodic, $c_n$ generates a
complete flow. We apply Lemma~\ref{lem:lie-euler-flow-defect}
with $u=f$ and $c=c_n$. The field difference has zero $x$-component,
so the resulting bound contains only $\nabla_v f$.
Next decompose, writing $f_n=f(t_n)$ and
$\zeta=\zeta_\tau^{c_n}$,
\begin{equation*}
\begin{aligned}
  f(t_n+\tau)-g_n\circ\widehat X_n
  ={}&[f(t_n+\tau)-f_n\circ\zeta]\\
     &+[f_n\circ\zeta-f_n\circ\widehat X_n]\\
     &+[(f_n-g_n)\circ\widehat X_n].
\end{aligned}
\end{equation*}
The second and third terms are bounded on $K$ by $G_n\eta_n$
and $e_n$, respectively. This proves \eqref{eq:vp-one-step-budget}.
The triangle inequality for the electric-field difference and
integration in $s$ give \eqref{eq:vp-one-step-budget-lipschitz}.
\end{proof}

Proposition~\ref{prop:vp-one-step-budget} separates the incoming
distribution error, the error from a varying or approximated field,
and the error in tracing the frozen characteristics. It does not
require bounds on derivatives of $g_n$, or a common regularity set
containing all numerical iterates. It also shows why field accuracy
becomes more demanding when $G_{v,n}$ grows through filamentation.
For exact initialization and an exact step-initial field, the bound
reduces to an $O(\tau^2)$ local freezing defect plus the characteristic
integration error, when the stated time-Lipschitz bound holds.

This is a one-step error estimate for the Lie--Euler specialization.
In a self-consistent computation, $\varepsilon_n$ generally also
depends on the incoming distribution error. We do not estimate that
dependence here; thus the coefficient $1$ in front of $e_n$ does not
establish nonlinear nonexpansiveness, stability, or global convergence.
Nor does this first-order example prove the PDE order of the
higher-order CF-RKEI compositions. Their coefficients provide a
design order, which is assessed numerically below.

\subsubsection{A self-consistent numerical flow for the Vlasov--Poisson system}
\label{subsec:vp-numerical-flow}
We denote by $\Psi^0_{(n+1)\tau,h}$ the self-consistent numerical
flow obtained by numerically realizing the exact frozen-flow
construction.
The subscript $h$ indicates the discrete field reconstruction
and characteristic integration.
Writing $\Psi^{n\tau}_{(n+1)\tau,h}$ for the numerical one-step
backward map and setting $\Psi^0_{0,h}:=\operatorname{Id}$, we define
\begin{equation}
  \Psi^0_{(n+1)\tau,h}
  :=
  \Psi^0_{n\tau,h}\circ\Psi^{n\tau}_{(n+1)\tau,h}.
  \label{eq:numerical-history-flow}
\end{equation}
The numerical distribution is evaluated as
\begin{equation}
  f_{\tau,h}^{n+1}
  :=
  f_0\circ\Psi^0_{(n+1)\tau,h}.
  \label{eq:numerical-pullback}
\end{equation}

For a fixed field $\mathcal E$, let
$\widehat\xi_c^{\mathcal E}$ denote the numerical approximation
of $\xi_c^{\mathcal E}$ obtained by integrating
\eqref{eq:frozen-characteristics} backward from $\sigma=\tau$
to $\sigma=0$.
We use an $s_{\mathrm{in}}$-stage explicit Runge--Kutta method of order
$p_{\mathrm{in}}$.
Under the regularity assumptions required for its local error
estimate, the characteristic integration error on a bounded
phase-space set $K$ is measured using the flat-metric distance $d_M$.
For maps $F,G:K\to M$, set
\[
  \|F-G\|_K:=\sup_{z\in K}d_M\bigl(F(z),G(z)\bigr).
\]
Then
\begin{equation}
  \big\|
    \xi_c^{\mathcal E}
    -
    \widehat\xi_c^{\mathcal E}
  \big\|_K
  \leq
  C_K\tau^{p_{\mathrm{in}}+1}
  \label{eq:inner-order-general}
\end{equation}
for sufficiently small $\tau>0$.
The constant $C_K$ is independent of $\tau$ but may depend
on $c$ and the relevant bounds on $\mathcal E$.
This estimate compares exact and numerical characteristic maps
for the same reconstructed field.
For an outer method with design order $p$, choosing
$p_{\mathrm{in}}\geq p$ is a natural requirement for targeting an
$O(\tau^{p+1})$ local characteristic error, provided that the required
bounds hold uniformly across the frozen factors and time steps.
This choice alone does not establish an order-$p$ PDE error estimate.
Preserving the overall order additionally requires stable and
consistent stage-field reconstruction and control of velocity
quadrature, Poisson discretization, field interpolation.

Let $E_h^{(k-1)}$ denote the reconstructed stage fields,
with $E_h^{(0)}:=E_h^n$.
Using the coefficients of the exact frozen-flow construction,
we define
\begin{equation}
  \bar E_{r,j,h}
  :=
  \big(a^\alpha_{r,j}\big)^{-1}
  \sum_{k=1}^{r-1}\alpha^k_{r,j}E_h^{(k-1)},
  \qquad
  \widehat\xi_{r,j,h}
  :=
  \widehat\xi^{\bar E_{r,j,h}}_{a^\alpha_{r,j}}.
\end{equation}
The numerical stage composition is
\begin{equation}
  \Xi_{r,h}
  :=
  \widehat\xi_{r,1,h}\circ\cdots\circ
  \widehat\xi_{r,J_r,h}.
\end{equation}
Similarly, the final frozen fields and the numerical one-step
map are
\begin{align}
  \bar E^\beta_{j,h}
  &:=
  \big(a^\beta_j\big)^{-1}
  \sum_{k=1}^{s}\beta^k_j E_h^{(k-1)},\\
  \Psi^{n\tau}_{(n+1)\tau,h}
  &:=
  \widehat\xi^{\bar E^\beta_{1,h}}_{a^\beta_1}
  \circ\cdots\circ
  \widehat\xi^{\bar E^\beta_{J,h}}_{a^\beta_J}.
  \label{eq:numerical-one-step-flow}
\end{align}

Given the accepted numerical history map $\Psi^0_{n\tau,h}$,
we reconstruct the electric fields required by the current
time step through velocity quadrature, a discrete Poisson
solve, and field interpolation.
We use a uniform tensor-product spatial grid
$\{x_i\}_{i=1}^{N_x}$ and velocity-cell midpoints
$\{v_\ell\}_{\ell=1}^{N_v}$ in a truncated velocity domain.
For a finite composition $\Theta_h$ of numerical elementary
backward maps within the current time step, the charge density
is approximated by
\begin{equation}
  \rho_{i,h}[\Theta_h]
  :=
  \bar\rho_0
  -
  \sum_{\ell=1}^{N_v}w_\ell\,
  f_0\!\left(
    \Psi^0_{n\tau,h}
    \bigl(\Theta_h(x_i,v_\ell)\bigr)
  \right),
  \label{eq:numerical-charge-density}
\end{equation}
where $w_\ell$ is the volume of the corresponding velocity cell.

The grid values $\{\rho_{i,h}[\Theta_h]\}_{i=1}^{N_x}$
are transformed by FFT.
Each nonzero Fourier mode of the electric field is recovered
through \eqref{eq:poisson}, and the field values
$\{E_{i,h}[\Theta_h]\}_{i=1}^{N_x}$ are obtained by the inverse FFT.
The zero charge mode is omitted, so the periodic Poisson solver
acts on the zero-mean component of the computed charge density.
The zero electric-field mode is set to zero.
Periodic cubic spline interpolation of these field values
provides a continuous field that can be evaluated at the
off-grid points encountered during characteristic integration.
Together, quadrature, the Poisson solve, and interpolation
define a discrete field functional $E_h[\cdot]$.
We introduce the field-reconstruction functional
\begin{equation}
  \mathcal N_{n,h}[\Theta_h]
  :=
  E_h\!\left[
    f_0\circ\Psi^0_{n\tau,h}\circ\Theta_h
  \right].
  \label{eq:numerical-field-functional}
\end{equation}
Its argument $\Theta_h$ is a within-step composition;
the accepted history is already included in this definition.

The first stage uses the accepted field
$E_h^{(0)}:=E_h^n$.
For each subsequent stage $r=2,\ldots,s$, the previously
computed fields determine $\Xi_{r,h}$, and the reconstructed
stage field is
\begin{equation}
  E_h^{(r-1)}
  :=
  \mathcal N_{n,h}[\Xi_{r,h}].
  \label{eq:numerical-stage-fields}
\end{equation}
These fields are then used to form the frozen linear
combinations for subsequent stages and the final update.
Once the numerical one-step map has been constructed,
the field accepted at the next time level is
\begin{equation}
  E_h^{n+1}
  :=
  \mathcal N_{n,h}
  [\Psi^{n\tau}_{(n+1)\tau,h}]
  =
  E_h[f_{\tau,h}^{n+1}].
  \label{eq:numerical-accepted-field}
\end{equation}

We refer to this construction as the self-consistent numerical
flow (SC-NuF).
The tableaux for SC-NuF-CF1, SC-NuF-CF2, and SC-NuF-CF3G
are given in Appendix~\ref{app:low-order-tableaux}.
Temporal self-convergence tests in Section~\ref{sec:Result} give
rates consistent with the design orders over the tested range.

\begin{remark}[Scope of the local error estimates]
\label{rem:scnuf-fully-discrete-convergence}
For SC-NuF-CF1, the quantities in
Proposition~\ref{prop:vp-one-step-budget} correspond to
\[
  g_n=f_{\tau,h}^n,\qquad
  \widetilde E_n=E_h^n,\qquad
  \widehat X_n=\widehat\xi_1^{E_h^n}.
\]
The proposition separates the discrepancy in the reconstructed field
from the characteristic integration error for that same field.
The field discrepancy includes
the effects of the incoming distribution, velocity truncation and
quadrature, the Poisson solve, and field interpolation.
Periodic cubic splines provide the $C^1$ regularity needed to define
the frozen flow, but do not by themselves provide uniform error bounds
for the coupled numerical evolution.
Higher-order compositions additionally require control of their
intermediate stage errors. We do not derive a fully discrete
convergence estimate for SC-NuF in this work.
In related nonlinear problems, temporal and spatial errors have been
analyzed separately by introducing an intermediate time-discrete
system. This approach is developed for linearized Galerkin methods
for nonlinear parabolic equations in
\cite{LiSun2013Linearized} and for characteristics-mixed finite element
methods for miscible displacement in
\cite{WangSiSun2014Characteristics}.
These studies suggest a possible approach to a future fully discrete
analysis of SC-NuF, for which estimates of the coupled field
reconstruction and numerical characteristic maps remain to be
established.
\end{remark}

\subsection{Diffeomorphism and volume preservation}
\label{sec:CP}

We next investigate the geometric properties of the self-consistent exact frozen flow. 
We first establish the diffeomorphism property of the exact frozen-field characteristic flows $\xi_c^{\mathcal{E}}$
and the integral invariance of the associated exponential evolution operators. 
The exponential-pullback identity then allows us 
to establish volume preservation of the exact frozen-field flows. 
Finally, we show that this property is retained under the finite compositions defining the stage, final, and history maps.

Throughout this subsection, let $M:=\mathbb T^d\times\mathbb R^d$,
with the integer $q\geq1$ fixed above.
The case $q=1$ suffices for volume preservation and the Jacobian
determinant identities below; higher $q$ gives the corresponding
regularity of the flow maps.
Unless otherwise stated, we assume that the fixed electric field $\mathcal{E} \in C^q(\mathbb T^d;\mathbb R^d)$
and $c\in\mathbb R\setminus\{0\}$,
and consider the frozen vector field
\begin{align*}
  b(x,v) = b^{\mathcal E}(x,v)
  :=
  \begin{pmatrix}
    \,v\\
    -\,\mathcal E(x)
  \end{pmatrix}.
\end{align*}

\begin{lemma}[Diffeomorphism of exact frozen-field flows]
\label{lem:frozen-diffeomorphism}
The flow map $\varphi_\delta^b$ generated by $b$ is a
$C^q$-diffeomorphism of $M$ for every $\delta\in\mathbb R$.
In particular, the backward characteristic map
$\xi_c^{\mathcal E}$ defined in
\eqref{eq:frozen-characteristics} is a
$C^q$-diffeomorphism of $M$.
\end{lemma}

\begin{proof}
Since $\mathbb T^d$ is compact and $q\geq 1$,
the derivative $D\mathcal E$ is bounded.
Consequently, $b=b^{\mathcal E}$ is a
globally Lipschitz $C^q$ vector field on $M$.

Standard ODE theory therefore guarantees that
its flow $\varphi_\delta^b$ exists uniquely
for all $\delta\in\mathbb R$ and depends $C^q$-smoothly
on the initial data. Moreover, the flow property
implies
\begin{align*}
  \bigl(\varphi_\delta^b\bigr)^{-1}
  =
  \varphi_{-\delta}^b.
\end{align*}
Thus $\varphi_\delta^b$ is a $C^q$-diffeomorphism of $M$.
By Lemma~\ref{lem:rate-interval},
\[
  \xi_c^{\mathcal E} = \zeta_{c\tau}^b = \varphi_{-c\tau}^b.
\]
Hence $\xi_c^{\mathcal E}$ is also a
$C^q$-diffeomorphism.
\end{proof}

We now turn to the analytic counterpart of this geometric fact: 
the invariance of the phase-space integral 
under the exponential evolution operator generated by the transport equation associated with \(b\). 
This invariance will subsequently be combined with an exponential-pullback identity 
to establish volume preservation.

\begin{lemma}[Integral invariance of the exponential operator]
\label{lem:exponential-integral-invariance}
Let
\begin{align*}
  C_{\mathcal E}g:=-\mathcal L_b g=-b\cdot\nabla_{x,v}g.
\end{align*}
Denote by $\exp(\delta C_{\mathcal E})$ the exact evolution
operator associated with the transport equation
\begin{align*}
  \partial_{\delta} u=C_{\mathcal E}u.
\end{align*}
Then, for every $g\in C_c^\infty(M)$ and
$\delta\in\mathbb R$,
\begin{align}
  \int_M\exp(\delta C_{\mathcal E})g(z)\,dz
  =
  \int_Mg(z)\,dz.
  \label{eq:exponential-integral-invariance}
\end{align}
\end{lemma}

\begin{proof}
Fix $g\in C_c^\infty(M)$ and set
\begin{align*}
  u(\delta,z):=\exp(\delta C_{\mathcal E})g(z).
\end{align*}
By definition of the evolution operator,
\begin{align*}
  \partial_{\delta}u
  =
  C_{\mathcal E}u
  =
  -b\cdot\nabla_{x,v}u,
  \qquad
  u(0,z)=g(z).
\end{align*}

If $g$ is supported in $|v|\leq R$, then $u(\delta)$ is supported
in $|v|\leq R+|\delta|\|\mathcal E\|_\infty$, since
$|\dot V|\leq\|\mathcal E\|_\infty$ along the characteristics of $b$.
This justifies differentiation under the integral and integration by
parts below. Using the transport equation, we obtain
\begin{align*}
  \frac{d}{d\delta}\int_Mu(\delta,z)\,dz
  &=
  -\int_M
  b(z)\cdot\nabla_{x,v}u(\delta,z)\,dz\\
  &=
  \int_M
  (\nabla_{x,v}\cdot b)(z)u(\delta,z)\,dz,
\end{align*}
where integration by parts is justified by
periodicity in $x$ and compact support in $v$.

For the present phase-space vector field,
\begin{align*}
  \nabla_{x,v}\cdot b(x,v)
  =
  \nabla_x\cdot v
  +
  \nabla_v\cdot \left(-\mathcal E(x)\right)
  =0.
\end{align*}
Consequently,
\begin{align*}
  \frac{d}{d\delta}\int_Mu(\delta,z)\,dz=0.
\end{align*}
Since $u(0,z)=g(z)$, integration in $\delta$ yields
\begin{align*}
  \int_M\exp(\delta C_{\mathcal E})g(z)\,dz
  =
  \int_Mg(z)\,dz.
\end{align*}
\end{proof}

With the integral invariance of $\exp(\delta C_{\mathcal E})$, 
we use the frozen exponential-pullback identity (Corollary~\ref{prop:frozen-pullback}) 
to obtain volume preservation of the exact frozen-field flow.
\begin{theorem}[Volume preservation of exact frozen-field flows]
\label{thm:frozen-flow-volume-preservation}
Under the assumptions of
Lemma~\ref{lem:frozen-diffeomorphism},
the exact backward characteristic flow
$\zeta_\delta^b$ preserves
phase-space volume for every $\delta\in\mathbb R$.
More precisely,
\begin{align}
  \int_M
  g\bigl(\zeta_\delta^b(z)\bigr)\,dz
  =
  \int_Mg(z)\,dz
  \label{eq:frozen-flow-volume-invariance}
\end{align}
for every $g\in C_c^\infty(M)$. In particular, $\delta= c\tau$ gives that $\xi_c^{\mathcal E}$ is volume-preserving.
\end{theorem}

\begin{proof}
Combining the frozen exponential-pullback identity 
\begin{align*}
  \exp(\delta C_{\mathcal E})g
  =
  g\circ\zeta_{\delta}^b,
\end{align*}
with Lemma~\ref{lem:exponential-integral-invariance} gives
\begin{align*}
  \int_M
  g\bigl(\zeta_\delta^b(z)\bigr)\,dz
  =
  \int_M\exp(\delta C_{\mathcal E})g(z)\,dz
  =
  \int_Mg(z)\,dz.
\end{align*}

As this equality holds for every
$g\in C_c^\infty(M)$, the pushforward of the phase-space Lebesgue measure under
$\zeta_\delta^b$ coincides with the original measure.
Thus $\xi_c^{\mathcal E} = \zeta_{c \tau}^b$ is volume-preserving.
\end{proof}

Theorem~\ref{thm:frozen-flow-volume-preservation} establishes volume preservation for a single exact frozen-field flow \(\xi_c^{\mathcal E}\). 
Since the stage maps \(\Xi_r\), the final map \(\Psi^{n\tau}_{(n+1)\tau}\), and the history map \(\Psi^0_{n\tau}\) are all built from finitely many such elementary factors composed together, 
it remains to show that volume preservation, together with the associated Jacobian determinant, survives finite composition. 
This is the content of the following corollary.
\begin{corollary}[Volume preservation and Jacobian determinants of composed flows]
\label{cor:rkei-volume-preservation}
Assume that, at every time step, all accepted fields $E^n$ and
stage fields $E^{(k-1)}$, $k=1,\ldots,s$, in the exact frozen-flow
construction of Section~\ref{sec:SC-NuF:VP} belong to
$C^q(\mathbb T^d;\mathbb R^d)$.
For the appropriate field
$\mathcal E\in\{\bar E_{r,j},\bar E^\beta_j\}$
and rate
$c\in\{a^\alpha_{r,j},a^\beta_j\}$, 
every elementary factor
$\xi_c^{\mathcal E}$, every stage map
$\Xi_r$, and the final map
$\Psi^{n\tau}_{(n+1)\tau}$ is a volume-preserving
$C^q$-diffeomorphism of $M$.
In particular,
\begin{align*}
  \det D_z\xi_c^{\mathcal E}(z)=1, \qquad
  \det D_z\Xi_r(z)&=1, \qquad
  \det D_z\Psi^{n\tau}_{(n+1)\tau}(z)=1
\end{align*}
for every $z\in M$.
Furthermore, the history flow
\begin{align*}
  \Psi^0_{n\tau}
  =
  \Psi^0_{\tau}\circ
  \Psi^{\tau}_{2\tau}\circ\cdots\circ
  \Psi^{(n-1)\tau}_{n\tau}
\end{align*}
is also a volume-preserving
$C^q$-diffeomorphism and satisfies
\begin{align*}
  \det D_z\Psi^0_{n\tau}(z)=1
  \qquad\text{for all }z\in M.
\end{align*}
\end{corollary}

\begin{proof}
We first establish the Jacobian determinant
of an exact frozen-field flow.
By Theorem~\ref{thm:frozen-flow-volume-preservation},
\begin{align*}
  \int_M
  g\bigl(\zeta_\delta^b(z)\bigr)\,dz
  =
  \int_Mg(z)\,dz
\end{align*}
for every $g\in C_c^\infty(M)$.

Since $\zeta_{\delta}^b$ is a
$C^q$-diffeomorphism,
\begin{align*}
  \left(\zeta_{\delta}^b\right)^{-1} = \zeta_{-\delta}^b = \varphi_{\delta}^b.
\end{align*}
Then the change of variables
$y=\zeta_{\delta}^b(z)$ gives
\begin{align*}
  \int_M
  g(y)
  \left|\det D_y\varphi_{\delta}^b(y)\right|
  \,dy
  =
  \int_Mg(y)\,dy.
\end{align*}
As this equality holds for every
$g\in C_c^\infty(M)$, we obtain
\begin{align*}
  \left|\det D_y\varphi_\delta^b(y)\right|
  =1
\end{align*}
almost everywhere in $M$.
By continuity of the Jacobian determinant,
the equality holds everywhere.

Moreover, $\varphi_0^b
=\operatorname{Id}$, and
$\delta\mapsto\det D_y\varphi_\delta^b(y)$
is continuous and never vanishes.
Hence the flow preserves orientation, and
\begin{align*}
  \det D_y\varphi_\delta^b(y)
  =1
\end{align*}
for all $\delta\in\mathbb R$ and $y\in M$.
In particular,
\begin{align*}
  \det D_z\xi_c^{\mathcal E}(z)
  =
  \det D_z\varphi_\delta^b(z)
  =
  1,\quad \delta=-c\tau.
\end{align*}

Next, the frozen fields $\bar E_{r,j}$
and $\bar E^\beta_j$ are finite linear
combinations of the accepted or stage fields
and therefore belong to
$C^q(\mathbb T^d;\mathbb R^d)$.
Consequently, every elementary factor is a volume-preserving $C^q$-diffeomorphism.
Every stage map $\Xi_r$ and the final map
$\Psi^{n\tau}_{(n+1)\tau}$ is a finite composition
of such elementary factors.
For a generic composition
\begin{align*}
  \Xi=\xi_1\circ\cdots\circ\xi_J,
\end{align*}
let $z_J=z$ and
$z_{j-1}=\xi_j(z_j)$ for $j=J,\ldots,1$.
The chain rule yields
\begin{align*}
  \det D_z\Xi(z)
  =
  \prod_{j=1}^{J}
  \det D\xi_j(z_j)
  =1.
\end{align*}
Since a finite composition of
$C^q$-diffeomorphisms is again a
$C^q$-diffeomorphism, every stage map
and the final map have the stated properties.

Finally, applying the same argument
recursively to
\begin{align*}
  \Psi^0_{n\tau}
  =
  \Psi^0_{(n-1)\tau}
  \circ
  \Psi^{(n-1)\tau}_{n\tau}
\end{align*}
shows that the history flow is also a
volume-preserving $C^q$-diffeomorphism,
with
\begin{align*}
  \det D_z\Psi^0_{n\tau}(z)=1.
\end{align*}
\end{proof}

If $f_0\geq0$, direct evaluation gives
$f_{\tau,h}^{n+1}=f_0\circ\Psi^0_{(n+1)\tau,h}\geq0$ whenever the
numerical history map is defined. This positivity statement requires
neither volume preservation nor invertibility, provided that the initial
data are evaluated without a reconstruction that introduces negative values.

A general explicit Runge--Kutta approximation of the exact frozen-field flow $\xi_c^{\mathcal E}$ need not preserve volume.

\section{Numerical Experiments}
\label{sec:Result}

The tests consider one spatial and one velocity dimension, 
the two-dimensional guiding-center model, and two spatial and two velocity dimensions.
We examine the numerical temporal convergence rates 
and compare them with the expected orders of accuracy. 
We also investigate the evolution of relevant macroscopic quantities 
to assess the conservation properties of the method. 
In addition, representative contour plots are presented 
to illustrate the evolution of the numerical solution 
in phase space for the Vlasov--Poisson system 
and in physical space for the guiding-center model.

We refer to the proposed construction as the self-consistent numerical flow (SC-NuF) method.
SC-NuF-CF1, SC-NuF-CF2 and SC-NuF-CF3G are used for comparison,
and the corresponding tableaux are given in
Appendix~\ref{app:low-order-tableaux} for completeness.

Throughout this section, the time step \(\tau\) is fixed a priori for
the duration of each simulation. Given a final time \(T\) and a
prescribed number of time steps \(N_{\mathrm{steps}}\), we set
\(\tau = T/N_{\mathrm{steps}}\), 
and this convention is used identically for all SC-NuF variants. 
To report the step size relative to the grid and initial field, 
we define an equivalent Courant--Friedrichs--Lewy (CFL) number.
For the 1D1V Vlasov--Poisson system, defined by
\begin{align}\label{eq:CFL}
\tau = \frac{\mathrm{CFL}}
  {\bigl(v_{\max}/\Delta x\bigr) + \bigl({|E^{0}|}_{\max}/\Delta v\bigr)},
\end{align}
where \(v_{\max}\) is the largest velocity resolved by the velocity
domain, and \(E^{0}_{\max}\) is the maximum magnitude of the electric
field at the initial time \(t = 0\);
for the two-dimensional guiding-center model, it is defined by
\begin{align}\label{eq:CFL-GC}
\tau = \frac{\mathrm{CFL}}
  {\bigl({|E_{x_2}^0|}_{\max}/\Delta x\bigr) + \bigl({|E_{x_1}^0|}_{\max}/\Delta y\bigr)}.
\end{align}

These tests cover weakly and strongly nonlinear regimes. 
The results are interpreted for the stated numbers of points and time intervals.

\subsection{1D1V Vlasov--Poisson system tests}
In this subsection, we monitor the total mass, momentum, and energy,
\begin{align}
    M(t)
    &:= \int_{M} f(x,v,t)\,\mathrm{d}x\,\mathrm{d}v,\\
    P(t)
    &:= \int_{M} v f(x,v,t)\,\mathrm{d}x\,\mathrm{d}v,\\
    \mathcal{H}(t)
    &:= \frac{1}{2}\int_{M} v^2 f(x,v,t)\,\mathrm{d}x\,\mathrm{d}v
    + \frac{1}{2}\int_{\mathbb T^d} |E(x,t)|^2\,\mathrm{d}x.
\end{align}
These quantities are invariants of the continuous Vlasov--Poisson system.
The $L^2$ norm of the electric field is given by
\begin{equation}
    \|E(t)\|_2
    := \left(\int_{\mathbb T^d} \frac{1}{2}|E(x,t)|^2\,\mathrm{d}x\right)^{1/2}.
\end{equation}
For the numerical solution, the conservation errors are measured by the
absolute deviations from the corresponding initial values,
\begin{equation}
    \Delta M(t) := |M(t)-M(0)|,\quad
    \Delta P(t) := |P(t)-P(0)|,\quad
    \Delta\mathcal{H}(t) := |\mathcal{H}(t)-\mathcal{H}(0)|.
\end{equation}

\begin{figure}[htbp]
    \centering
    \begin{subfigure}[b]{0.45\textwidth}
        \centering
        \includegraphics[width=\textwidth]{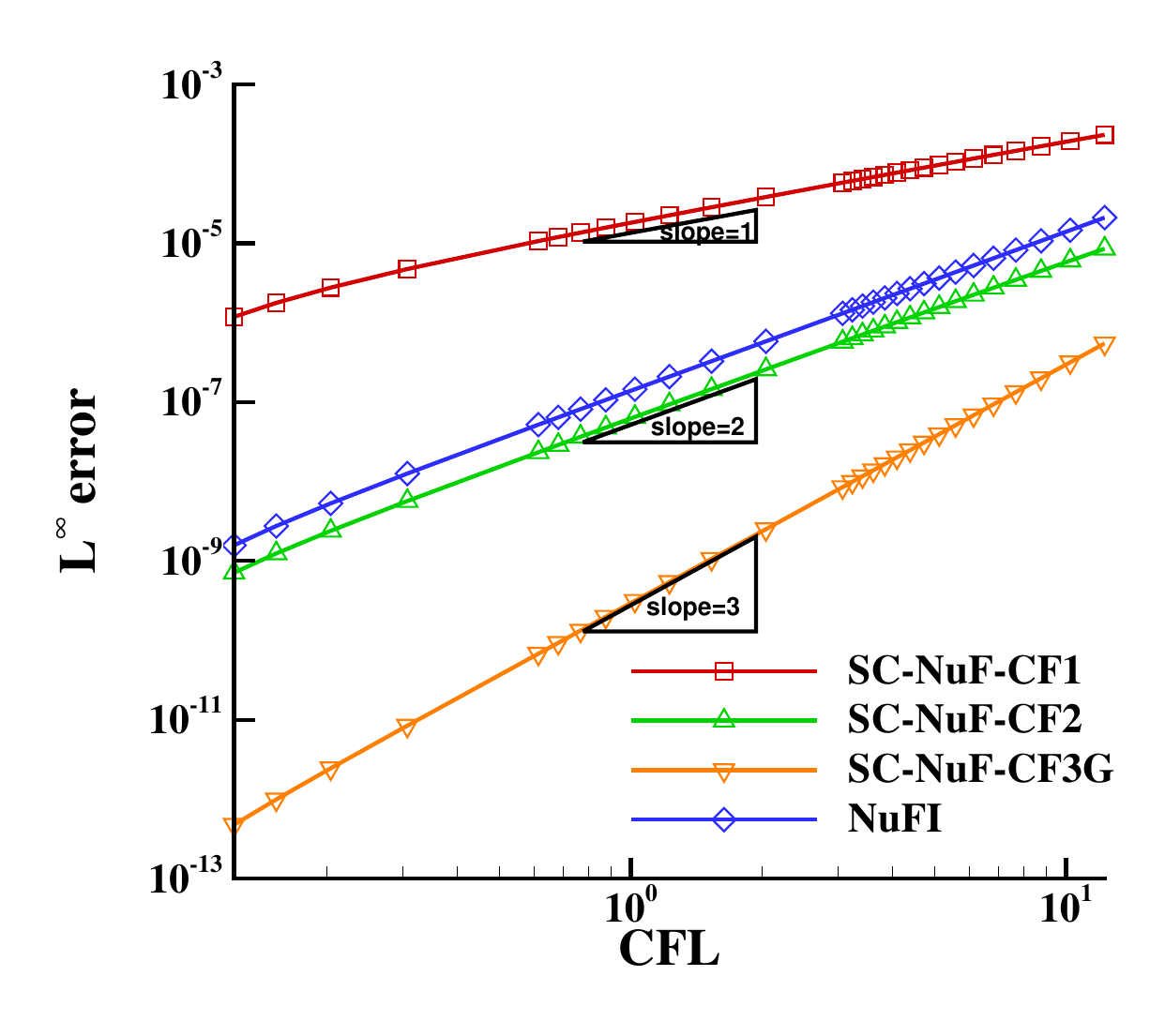}
        \caption{CFL v.s. $L^\infty$ error.}
        \label{fig:weak_landau_cfl_error}
    \end{subfigure}
    \hfill
    \begin{subfigure}[b]{0.45\textwidth}
        \centering
        \includegraphics[width=\textwidth]{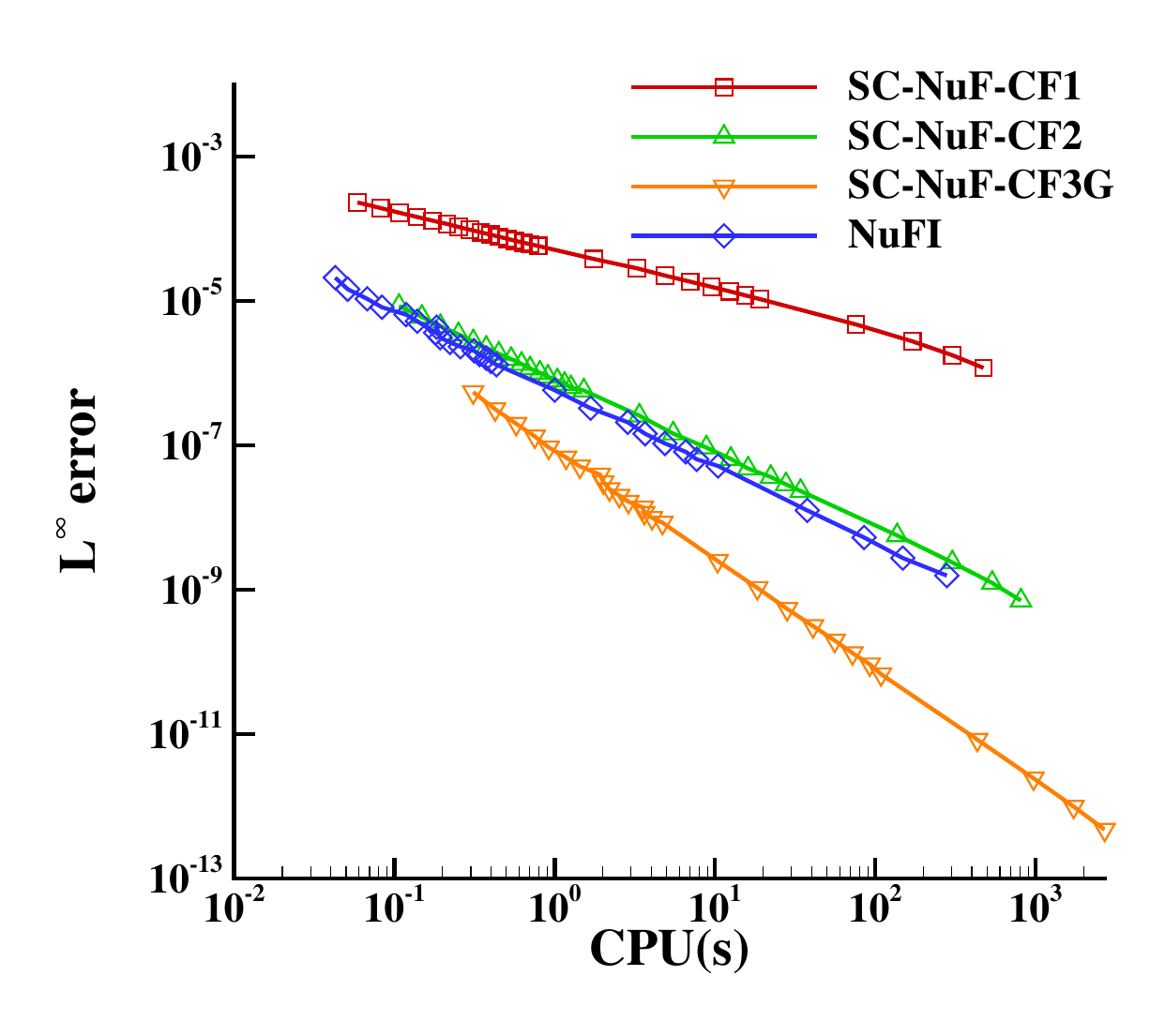}
        \caption{CPU times v.s. $L^\infty$ error.}
        \label{fig:weak_landau_CPU_error}
    \end{subfigure}

    \caption{Plots of $L^\infty$ error versus the CFL number (left) and error versus CPU times (right) for solving weak Landau damping at $T = 1$.
            Temporal self-convergence of SC-NuF coupled with CF-RKEI is assessed 
            with a reference solution from the corresponding scheme with $\tau=\frac{1}{1000}$.}
    \label{fig:weak_landau_error}
\end{figure}

\begin{figure}[htbp]
    \centering
    \begin{subfigure}[b]{0.48\textwidth}
        \centering
        \includegraphics[width=\linewidth]{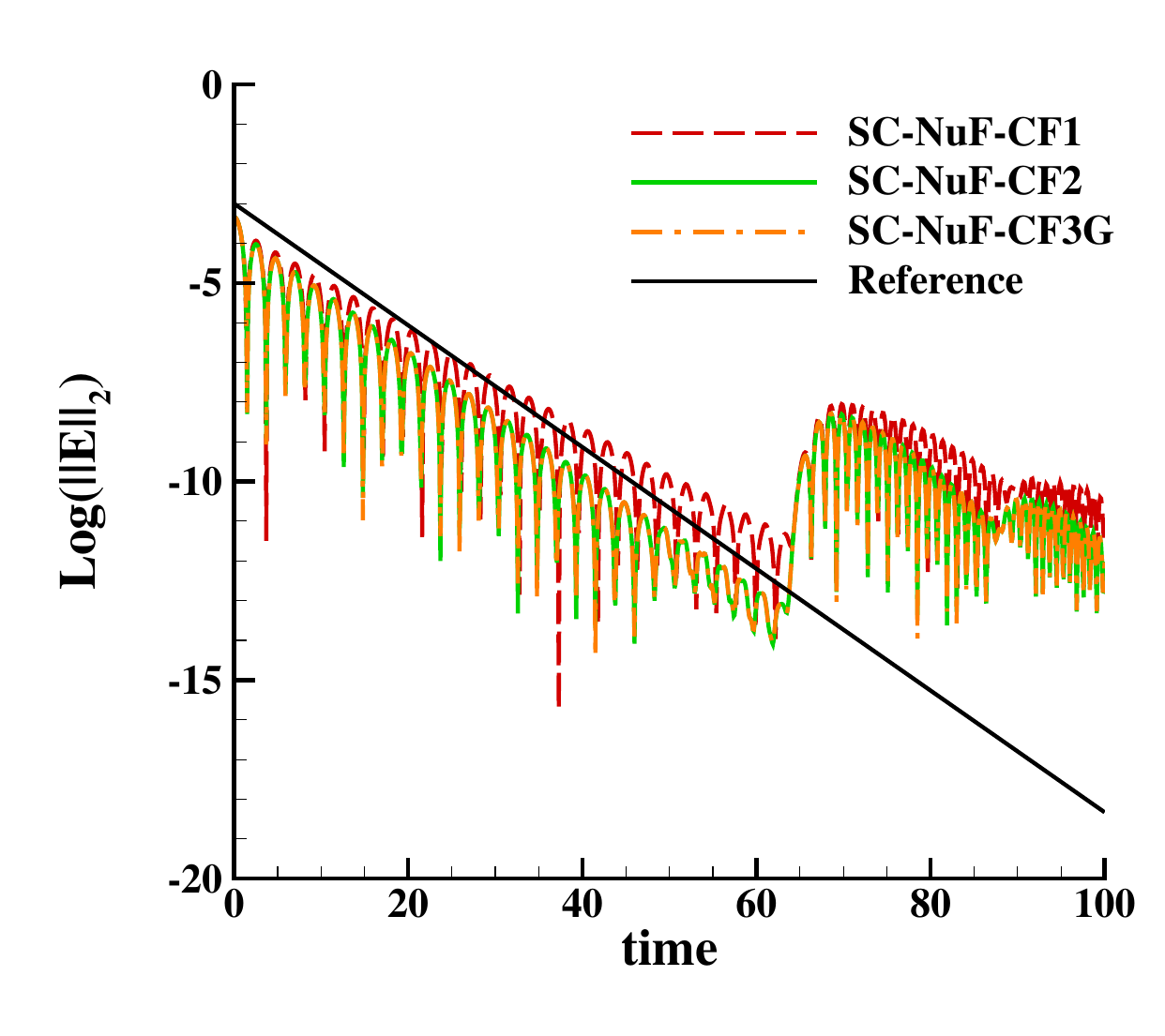}
    \end{subfigure}
    \hfill
    \begin{subfigure}[b]{0.48\textwidth}
        \centering
        \includegraphics[width=\linewidth]{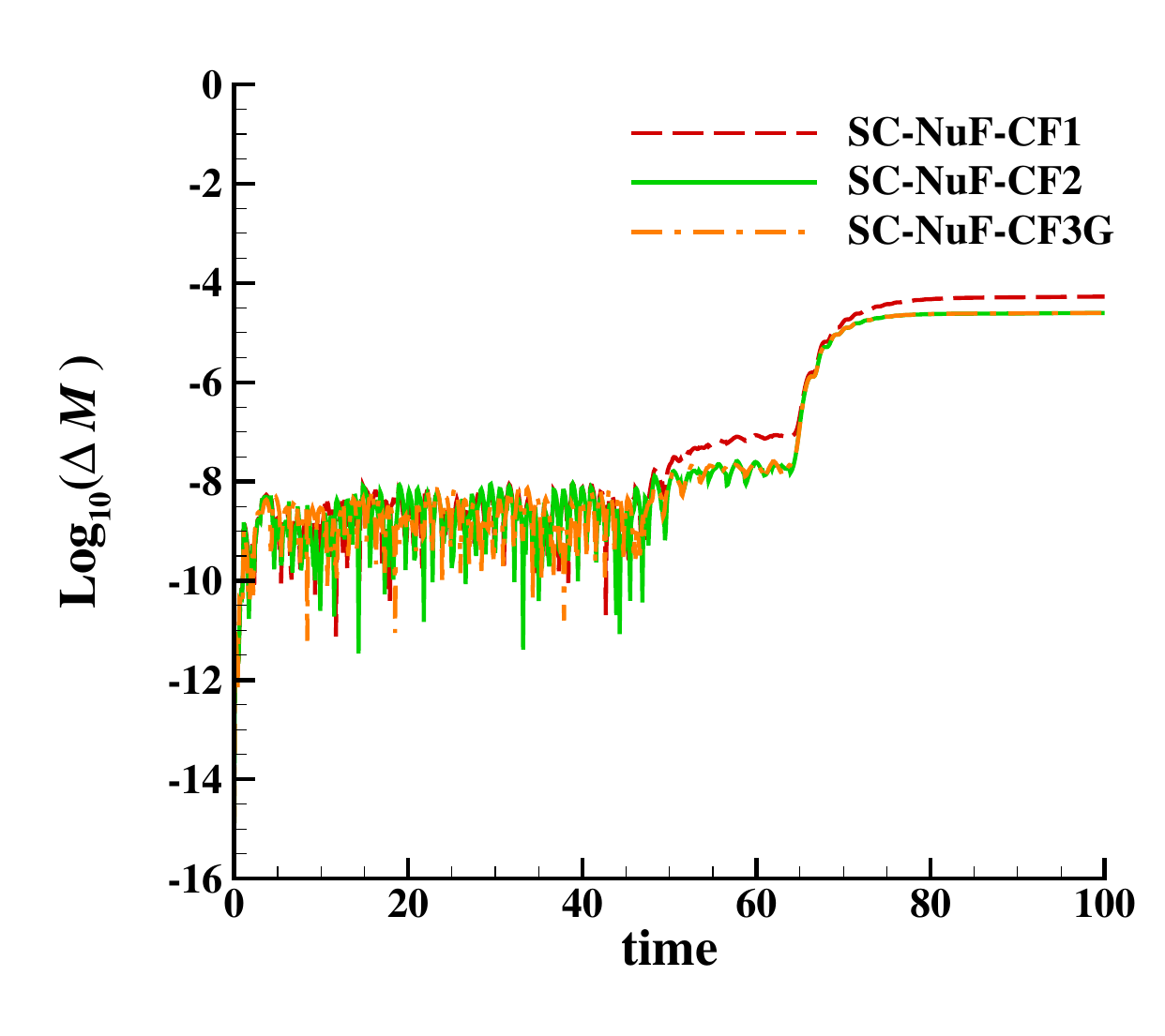}
    \end{subfigure} 
    
    \begin{subfigure}[b]{0.48\textwidth} 
        \centering
        \includegraphics[width=\linewidth]{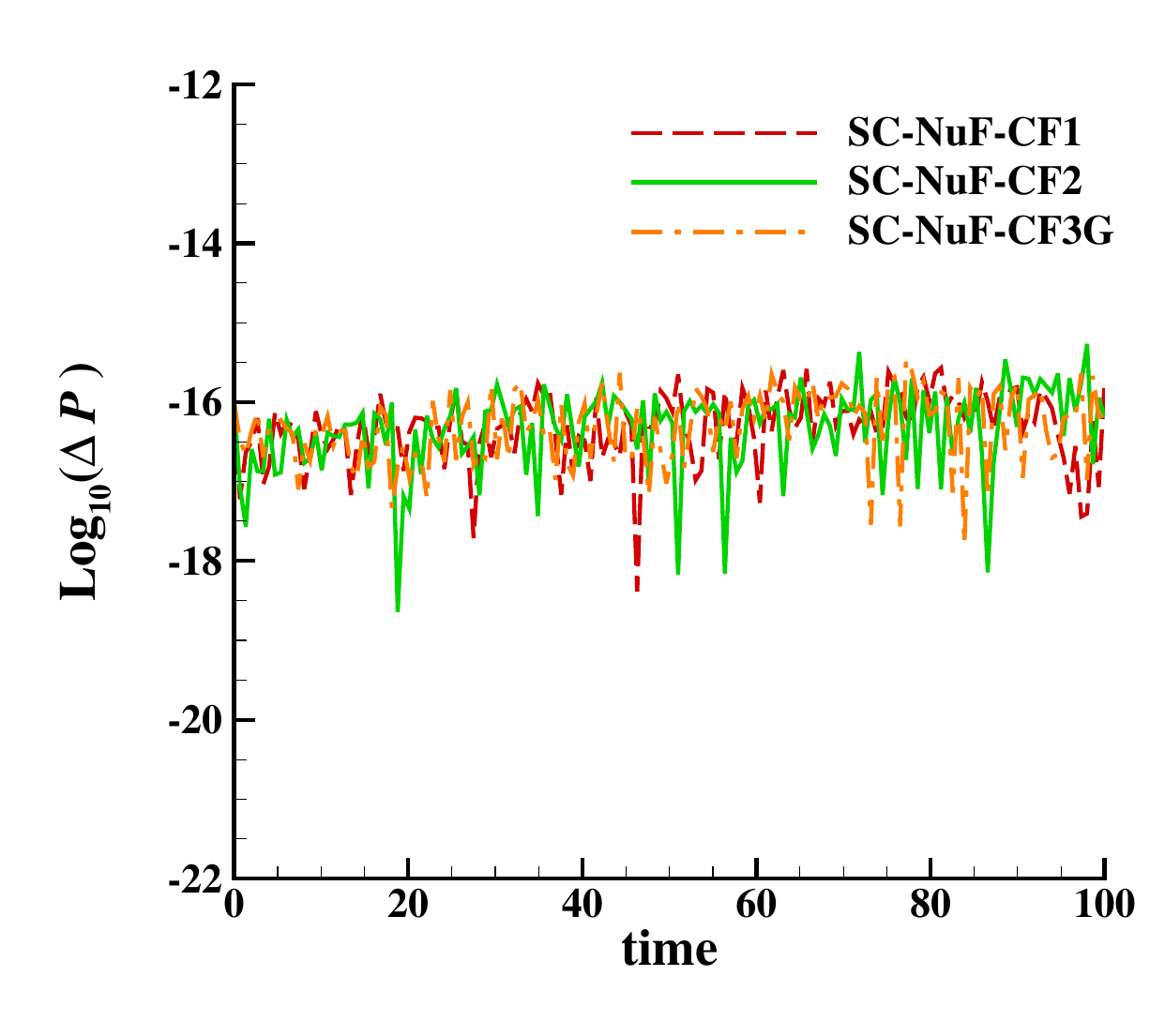}
  
    \end{subfigure}
    \hfill
    \begin{subfigure}[b]{0.48\textwidth} 
        \centering
        \includegraphics[width=\linewidth]{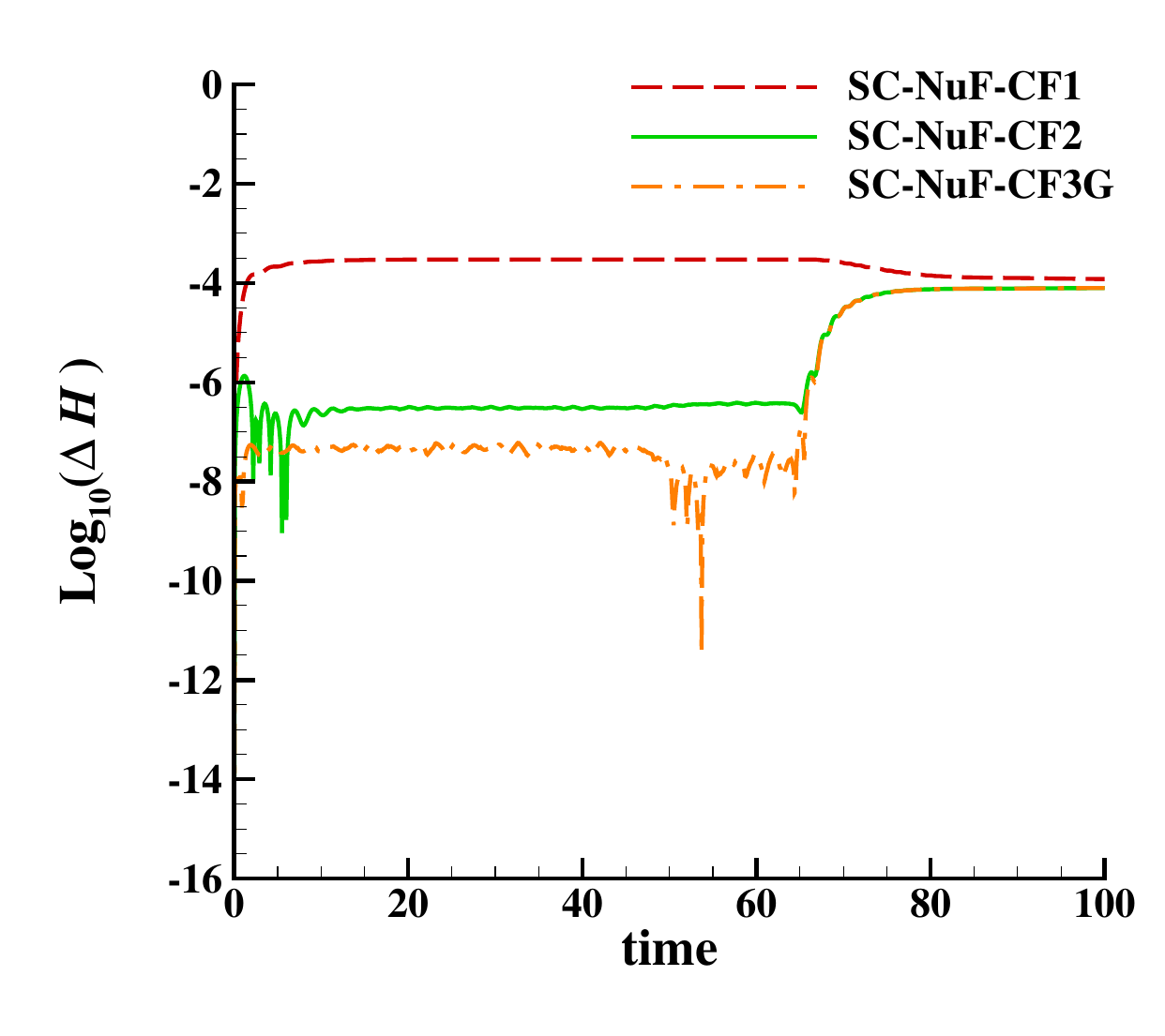}
    \end{subfigure}
    
    \caption{Weak Landau damping. Time evolution of $\log(||E||_2)$, $\log_{10}(\Delta M(t))$,
    $\log_{10}(\Delta P(t))$ and $\log_{10}(\Delta \mathcal{H}(t))$. 
    We use $128^2$ phase-space points, with $128$ points in each of
    the $x$ and $v$ directions, and set the time step to $\tau=0.1$.}
    \label{fig:weak_landau}
\end{figure}

\begin{example}[Weak Landau damping]\label{ex:weak-Landau}
We consider the initial condition
    \begin{align}\label{Landau}
        f_0(x,v) = \frac{1}{\sqrt{2\pi}}(1+\alpha\cos(kx))\exp(-\frac{v^2}{2}),
    \end{align}
where $\alpha=0.01$ and $k=0.5$. 
The phase space is $\Omega_x \times \Omega_v = [0,4\pi] \times [-6,6]$.

Using $128\times512$ phase-space quadrature points, we measure temporal
errors at $T=1$ against a reference solution computed for each scheme with
$\tau=\frac{1}{1000}$. Figure~\ref{fig:weak_landau_cfl_error} shows
$L^\infty$ errors consistent with the formal orders over the displayed range,
with NuFI \cite{KW2024} included for comparison;
Figure~\ref{fig:weak_landau_CPU_error} plots these errors against CPU time.

For the longer integration, we use $128^2$ phase-space points and
$\tau=0.1$ (CFL$\approx 6$). Figure~\ref{fig:weak_landau} shows the
electric-field history and conservation errors.
The decay envelopes of $\log (\|E(t)\|_2)$ are in good agreement with the reference line,
whose slope is $-0.1533$ \cite{FILBET2003}.
\end{example}

\begin{figure}[htbp]
    \centering
    \begin{subfigure}[b]{0.45\textwidth}
        \centering
        \includegraphics[width=\textwidth]{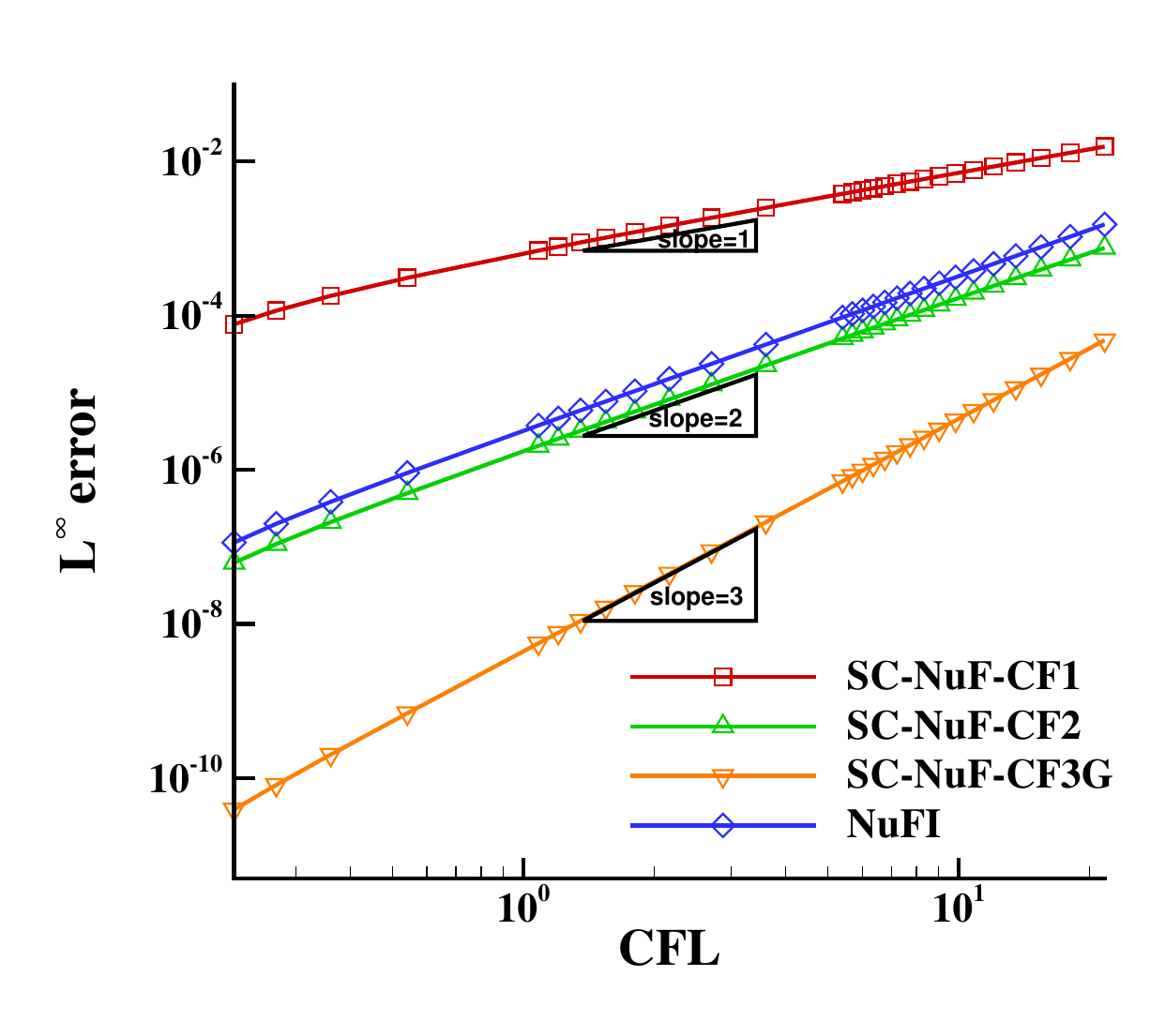}
        \caption{CFL v.s. $L^\infty$ error.}
        \label{fig:Strong_landau_cfl_error}
    \end{subfigure}
    \hfill
    \begin{subfigure}[b]{0.45\textwidth}
        \centering
        \includegraphics[width=\textwidth]{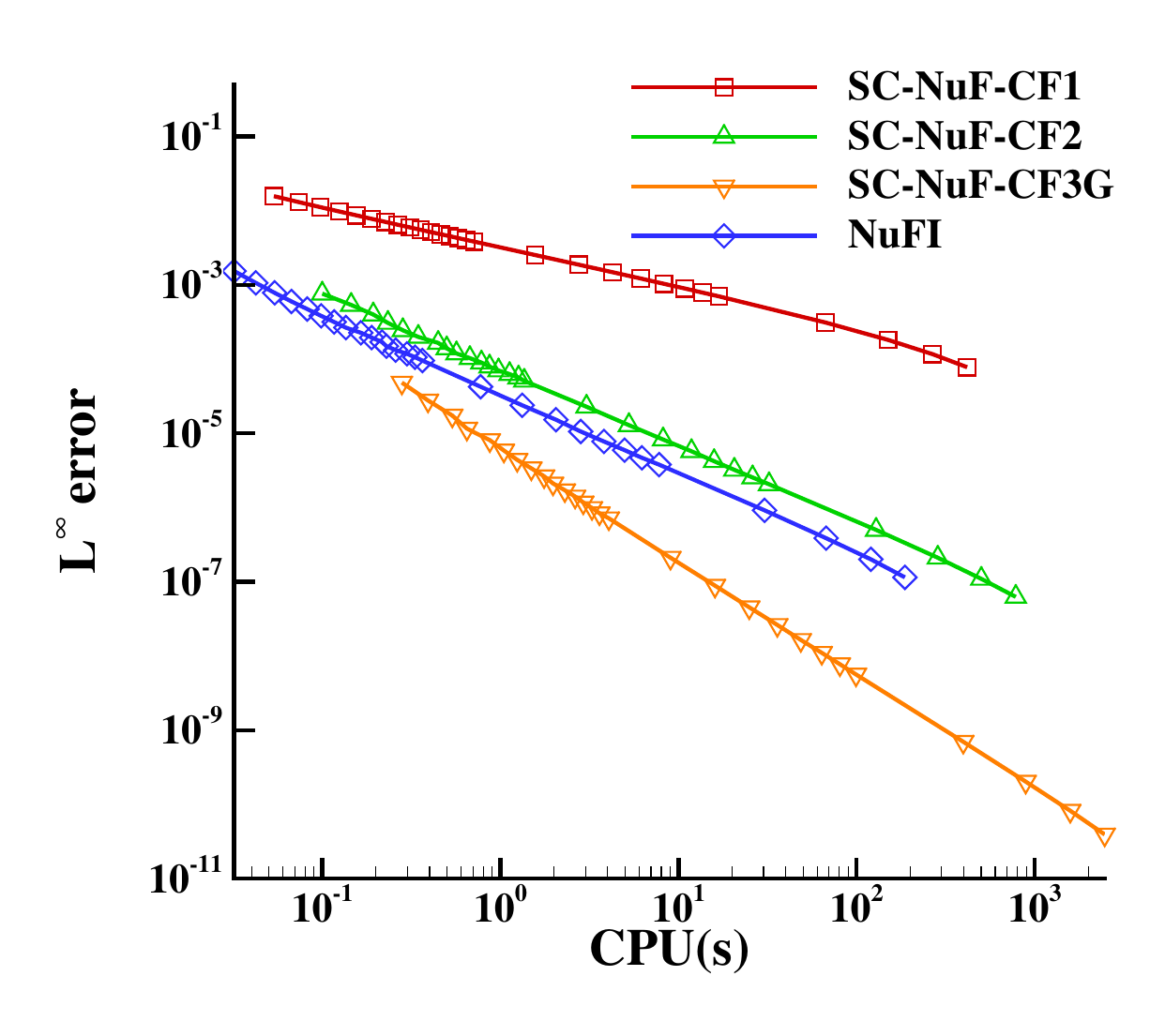}
        \caption{CPU times v.s. $L^\infty$ error.}
        \label{fig:Strong_landau_CPU_error}
    \end{subfigure}

    \caption{Plots of $L^\infty$ error versus the CFL number (left) and error versus CPU times (right) for solving strong Landau damping at $T = 1$.
            Temporal self-convergence of SC-NuF coupled with CF-RKEI is assessed 
            with a reference solution from the corresponding scheme with $\tau=\frac{1}{1000}$.}
    \label{fig:Strong_landau_error}
\end{figure}

\begin{figure}[htbp]
    \centering
    \begin{subfigure}[b]{0.48\textwidth}
        \centering
        \includegraphics[width=\linewidth]{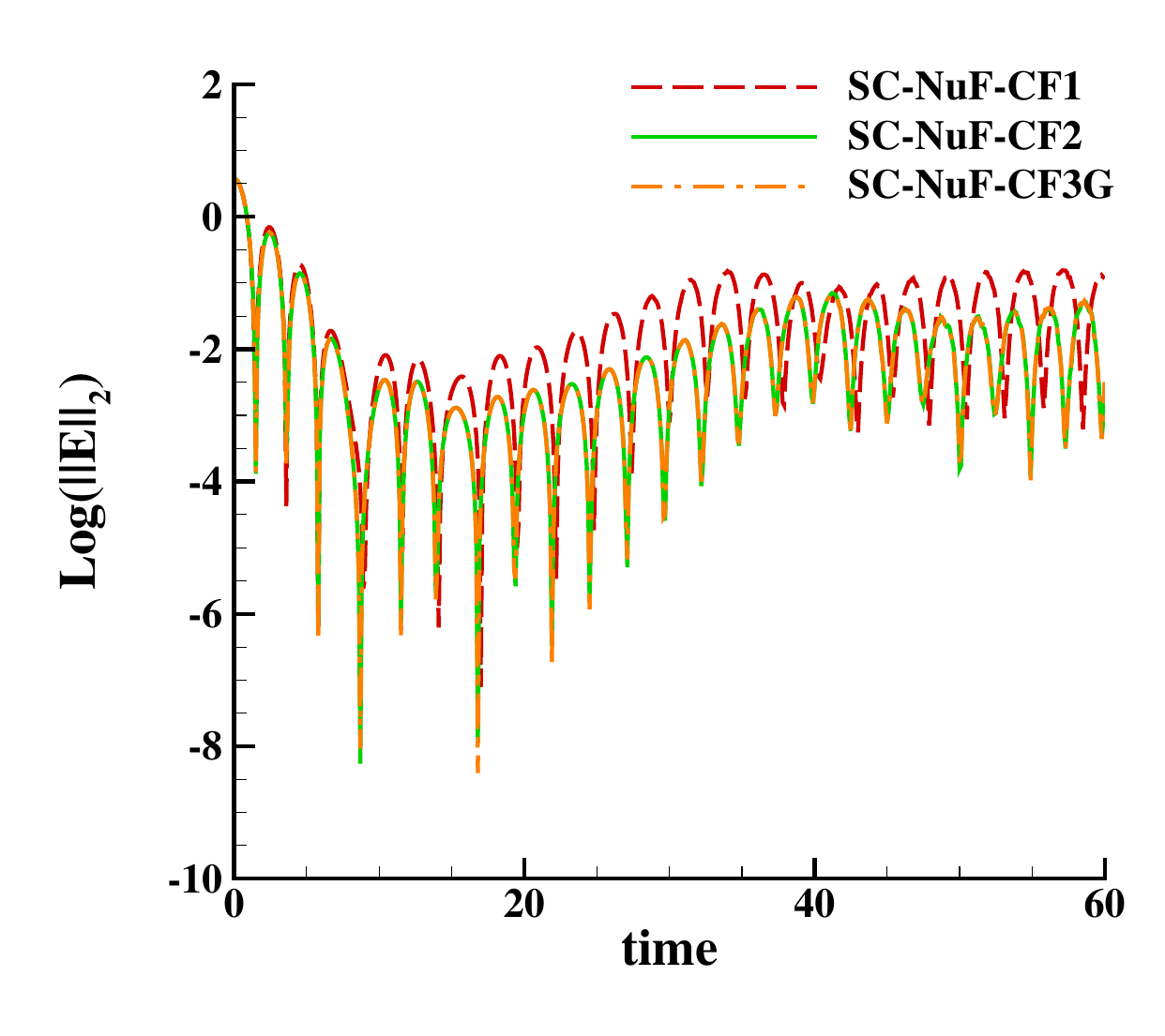}
    \end{subfigure}
    \hfill
    \begin{subfigure}[b]{0.48\textwidth}
        \centering
        \includegraphics[width=\linewidth]{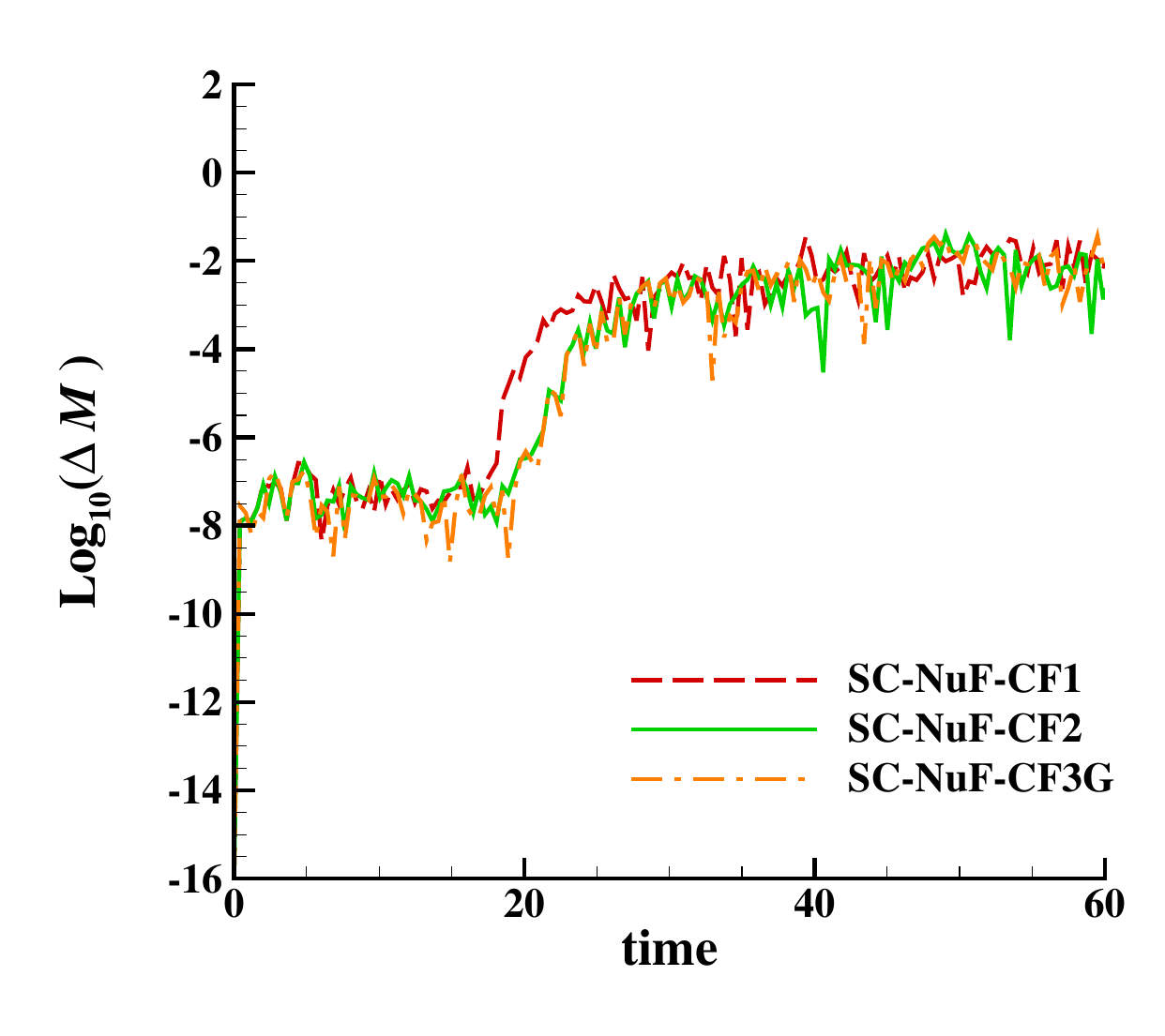}
    \end{subfigure}
    
    \begin{subfigure}[b]{0.48\textwidth} 
        \centering
        \includegraphics[width=\linewidth]{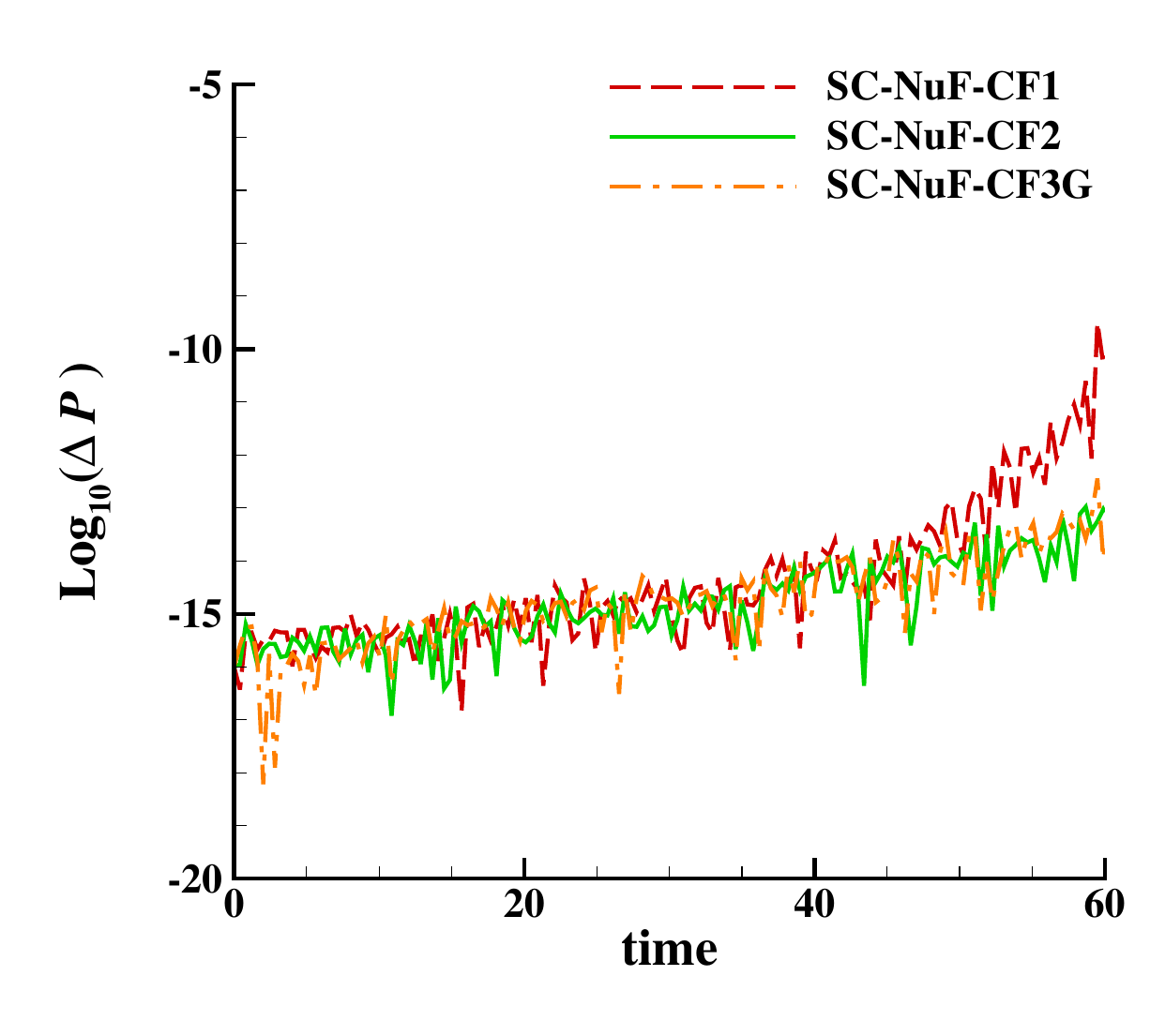}
    \end{subfigure}
    \hfill
    \begin{subfigure}[b]{0.48\textwidth} 
        \centering
        \includegraphics[width=\linewidth]{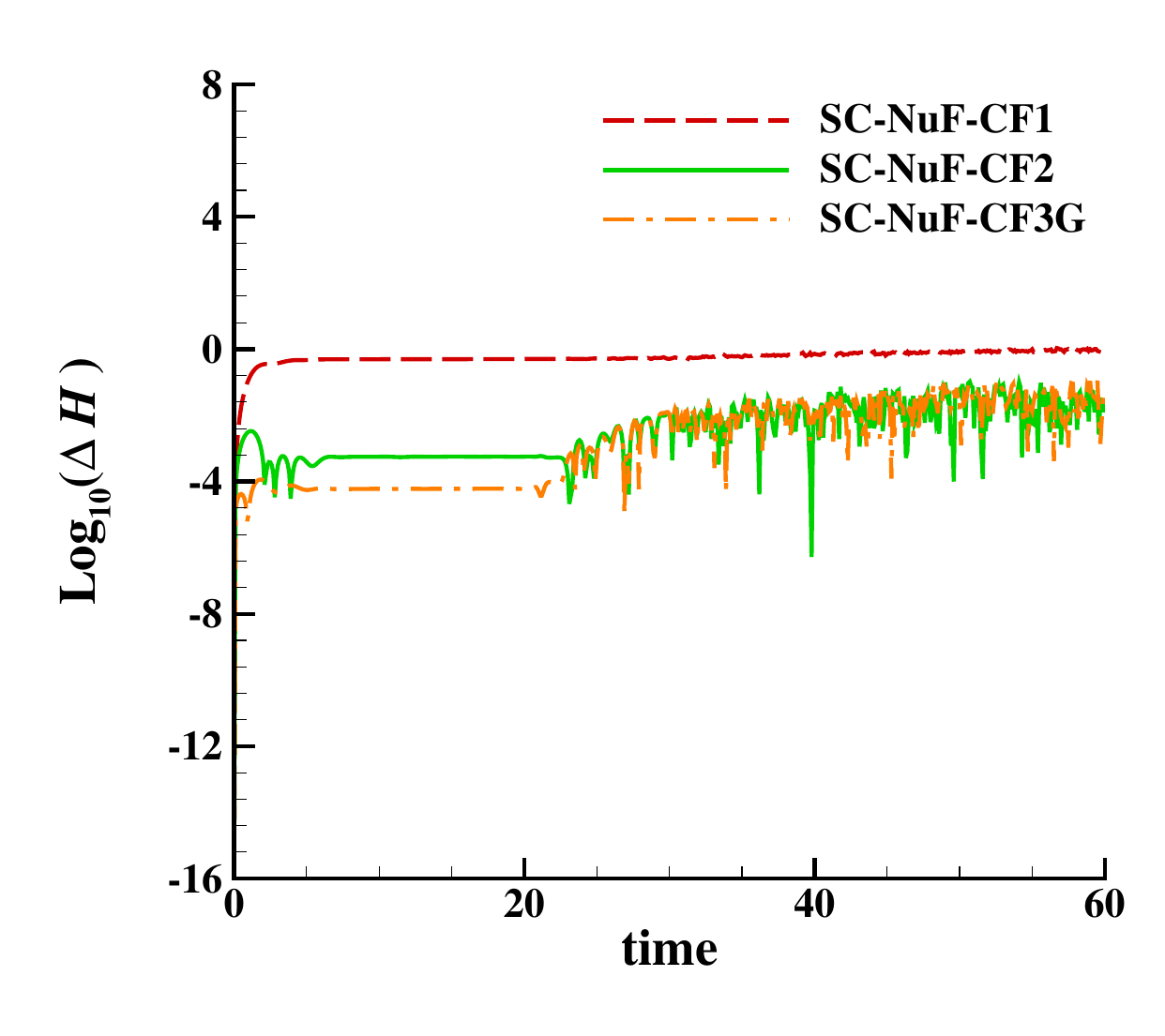}
    \end{subfigure}
    
    \caption{Strong Landau damping. Time evolution of $\log(||E||_2)$, $\log_{10}(\Delta M(t))$,
    $\log_{10}(\Delta P(t))$ and $\log_{10}(\Delta \mathcal{H}(t))$. 
    We use $128^2$ phase-space points, with $128$ points in each of
    the $x$ and $v$ directions, and set the time step to $\tau=0.1$.}
    \label{fig:strong_landau}
\end{figure}

\begin{example}[Strong Landau damping]\label{ex:strong_landau}
We set $\alpha=0.5$ in \eqref{Landau} to enter the strongly nonlinear
regime and use $\Omega_x\times\Omega_v=[0,4\pi]\times[-10,10]$.

With the same temporal-accuracy setup as in Example~\ref{ex:weak-Landau},
Figure~\ref{fig:Strong_landau_cfl_error} shows errors consistent with the
formal orders over the displayed range; Figure~\ref{fig:Strong_landau_CPU_error}
plots the errors against CPU time.

The longer integration uses $128^2$ phase-space points and $\tau=0.1$
(CFL$\approx10.83$). Figure~\ref{fig:strong_landau} shows the
electric-field history and conservation errors.

\end{example}

The mass, momentum, and total-energy diagnostics in
Figures~\ref{fig:weak_landau} and~\ref{fig:strong_landau} highlight
limitations of the fully discrete implementation. Achieving volume
preservation and conservation in SC-NuF requires compatible treatment of
velocity quadrature, electric-field computation and interpolation, and the
inner high-order Runge--Kutta integration of the frozen characteristic equations.
The present implementation does not yet fully retain these properties;
addressing this remains a direction for future work. The small momentum
errors may be due to the symmetry of the initial data, since less satisfactory
momentum conservation is observed in other tests; this also requires further analysis.

\begin{figure}[htbp]
    \centering
    \begin{subfigure}[b]{0.48\textwidth}
        \centering
        \includegraphics[width=\linewidth]{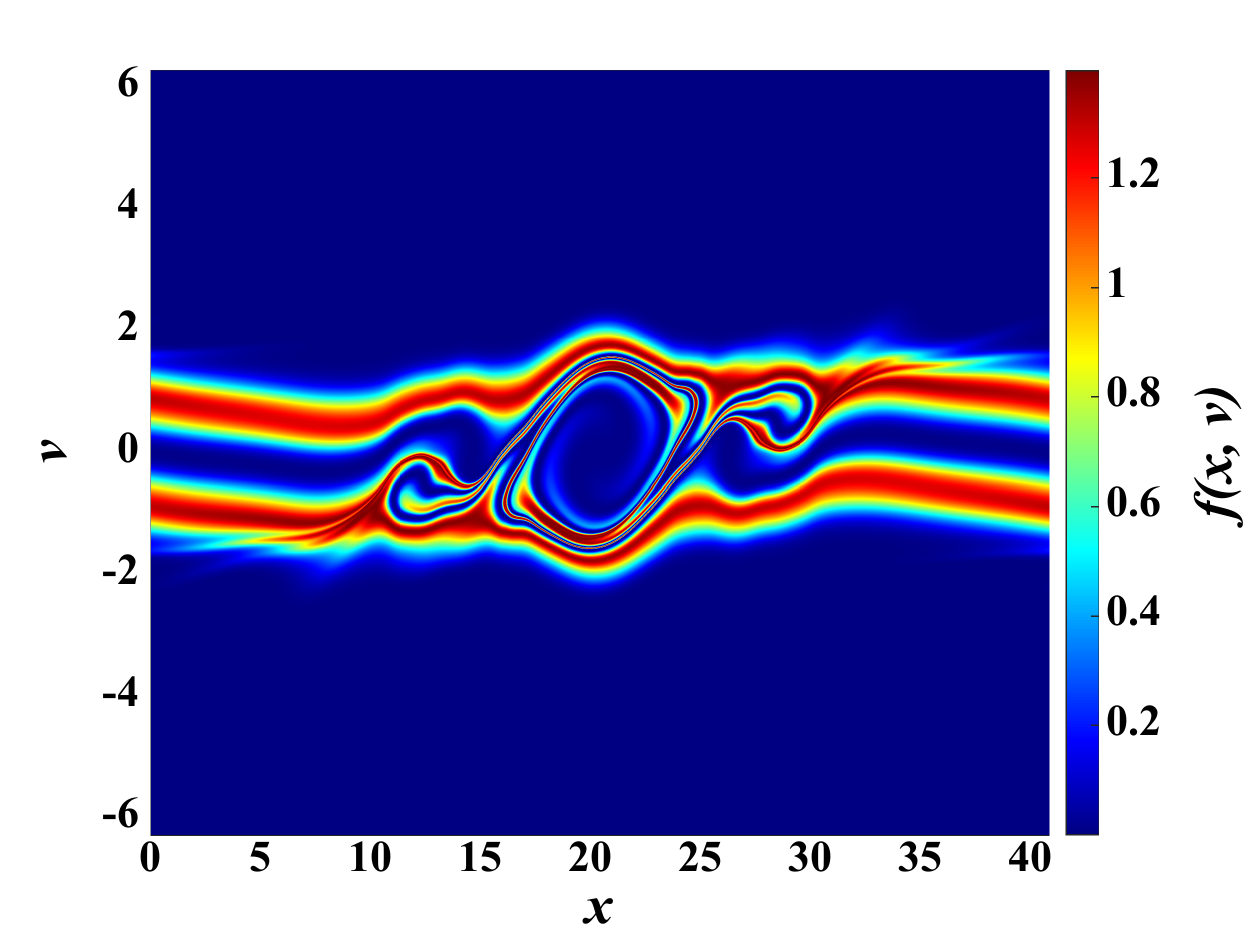}
        \caption{SC-NuF-CF2, $T=30$.}
        \label{fig:SLF-CF2_ex5_T30}
    \end{subfigure}
    \hfill
    \begin{subfigure}[b]{0.48\textwidth}
        \centering
        \includegraphics[width=\linewidth]{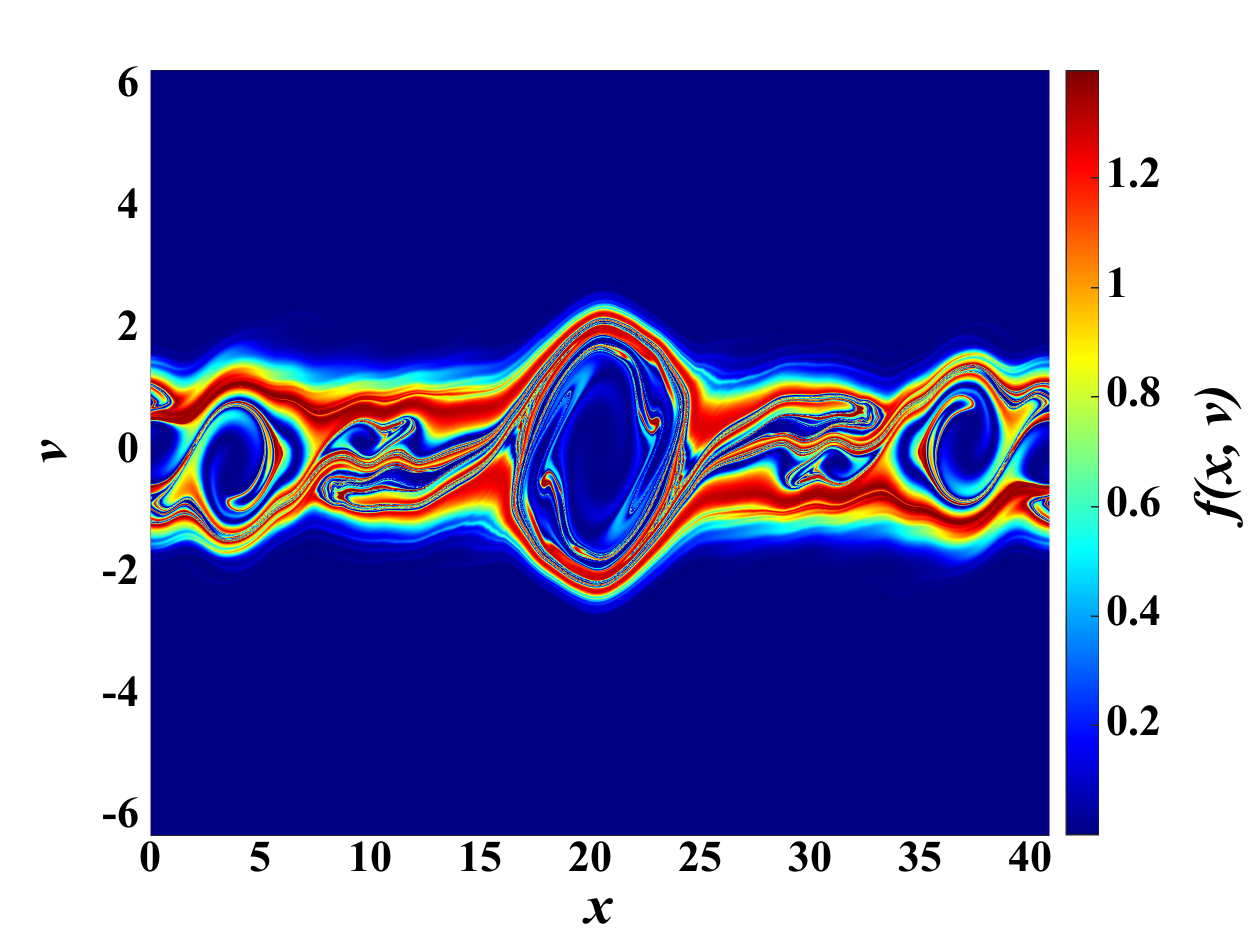}
        \caption{SC-NuF-CF2, $T=50$.}
        \label{fig:SLF-CF2_ex5_T50}
    \end{subfigure}

    \begin{subfigure}[b]{0.48\textwidth} 
        \centering
        \includegraphics[width=\linewidth]{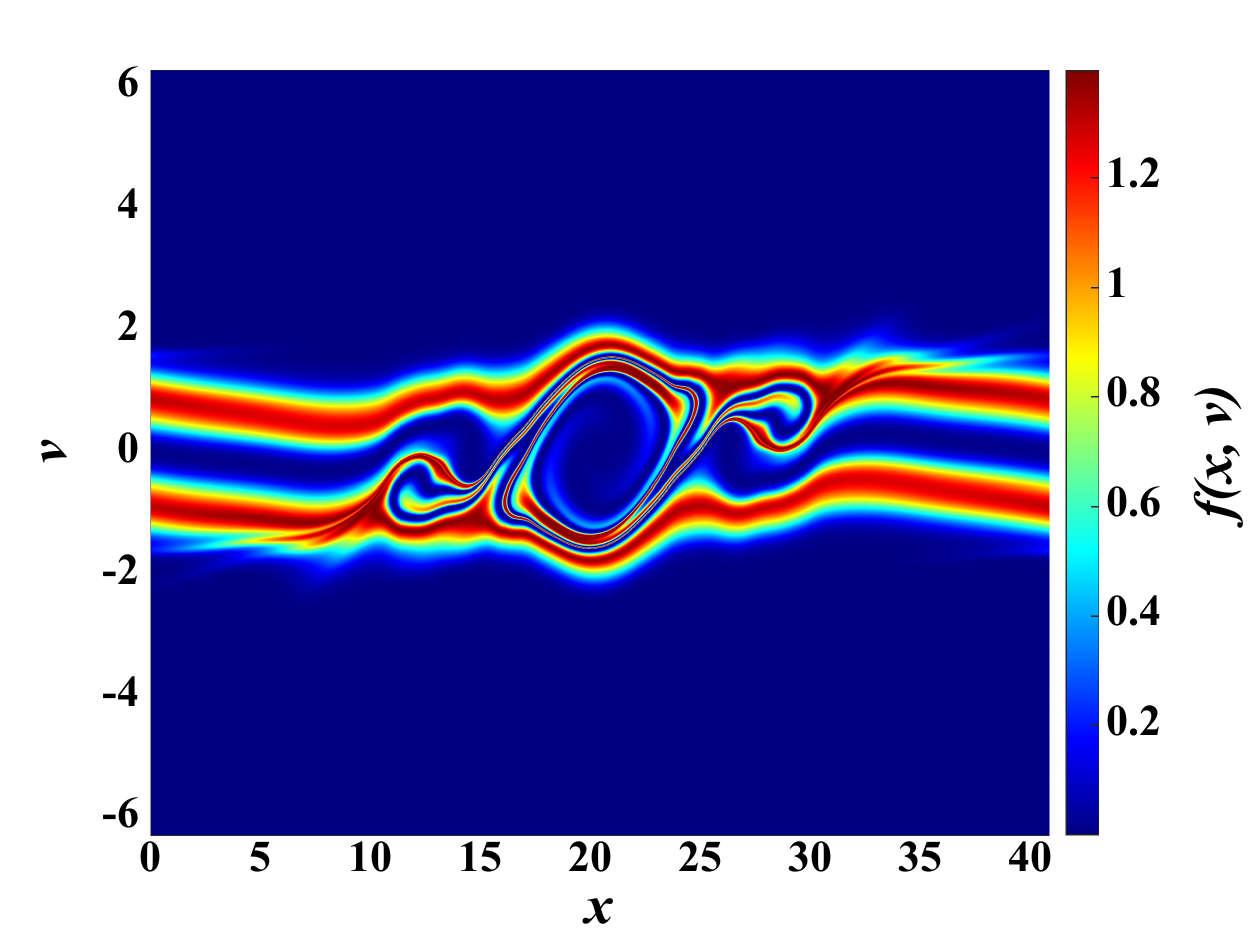}
        \caption{SC-NuF-CF3G, $T=30$.}
        \label{fig:SLF-CF3G_ex5_T30}
    \end{subfigure}
    \hfill
    \begin{subfigure}[b]{0.48\textwidth} 
        \centering
        \includegraphics[width=\linewidth]{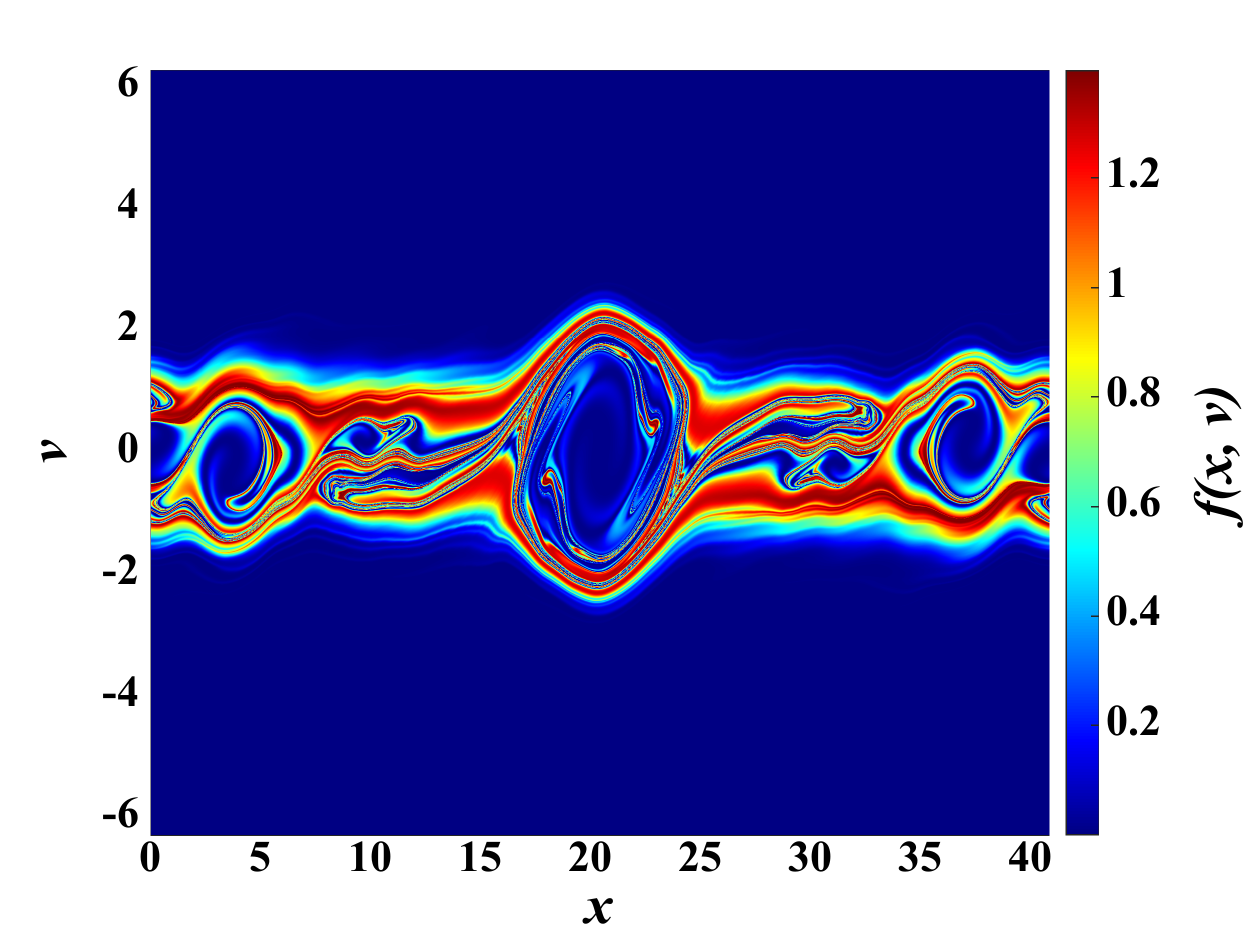}
        \caption{SC-NuF-CF3G, $T=50$.}
        \label{fig:SLF-CF3G_ex5_T50}
    \end{subfigure}
    
    \caption{Two-stream instability. 
    Surface plots of numerical sulotions at $T=30$ (left) and $T=50$ (right). 
    We use $512^2$ phase-space points, with $512$ points in each of
    the $x$ and $v$ directions, and set the time step to $\tau=1/10$.
    Top: SC-NuF-CF2. Bottom: SC-NuF-CF3G.}
    \label{fig:Two_streamII_solution_T30}
\end{figure}

\begin{example}[Two-stream instability]
We next consider the two-stream instability test \cite{CROUSEILLES2010}. The initial
distribution function consists of two symmetric drifting Maxwellian
populations and is given by
\begin{equation}
    f_0(x,v)=\frac{2}{0.6\sqrt{2\pi}}
    \bigl(1+\alpha\cos(kx)\bigr)
    \left[
        \exp\left(-\frac{(v-u)^2}{2(0.3)^2}\right)
        +
        \exp\left(-\frac{(v+u)^2}{2(0.3)^2}\right)
    \right],
\end{equation}
where
$
    \alpha=0.05,~
    k=\frac{2}{13},~
    u=0.99.
$
Accordingly, the spatial period is $L=2\pi/k=13\pi$.
The computation is performed on the phase-space domain
$
    \Omega_x\times\Omega_v
    =
    [0,13\pi]\times[-2\pi,2\pi],
$
with periodic boundary conditions in the $x$-direction.
The simulation uses $512^2$ phase-space points, with $512$ points
in each of the $x$ and $v$ directions, and a time step $\tau=0.1$.

For the two-stream instability test, we focus on the performance of the higher-order schemes with relatively large time step sizes. 
Figure~\ref{fig:Two_streamII_solution_T30} shows the phase-space distributions 
obtained by the SC-NuF-CF2 and SC-NuF-CF3G schemes at \(T=30\) and \(T=50\). 
Both plots display filamentary structures. 
\end{example}

\begin{figure}[htbp]
    \centering
    \begin{subfigure}[b]{0.48\textwidth}
        \centering
        \includegraphics[width=\linewidth]{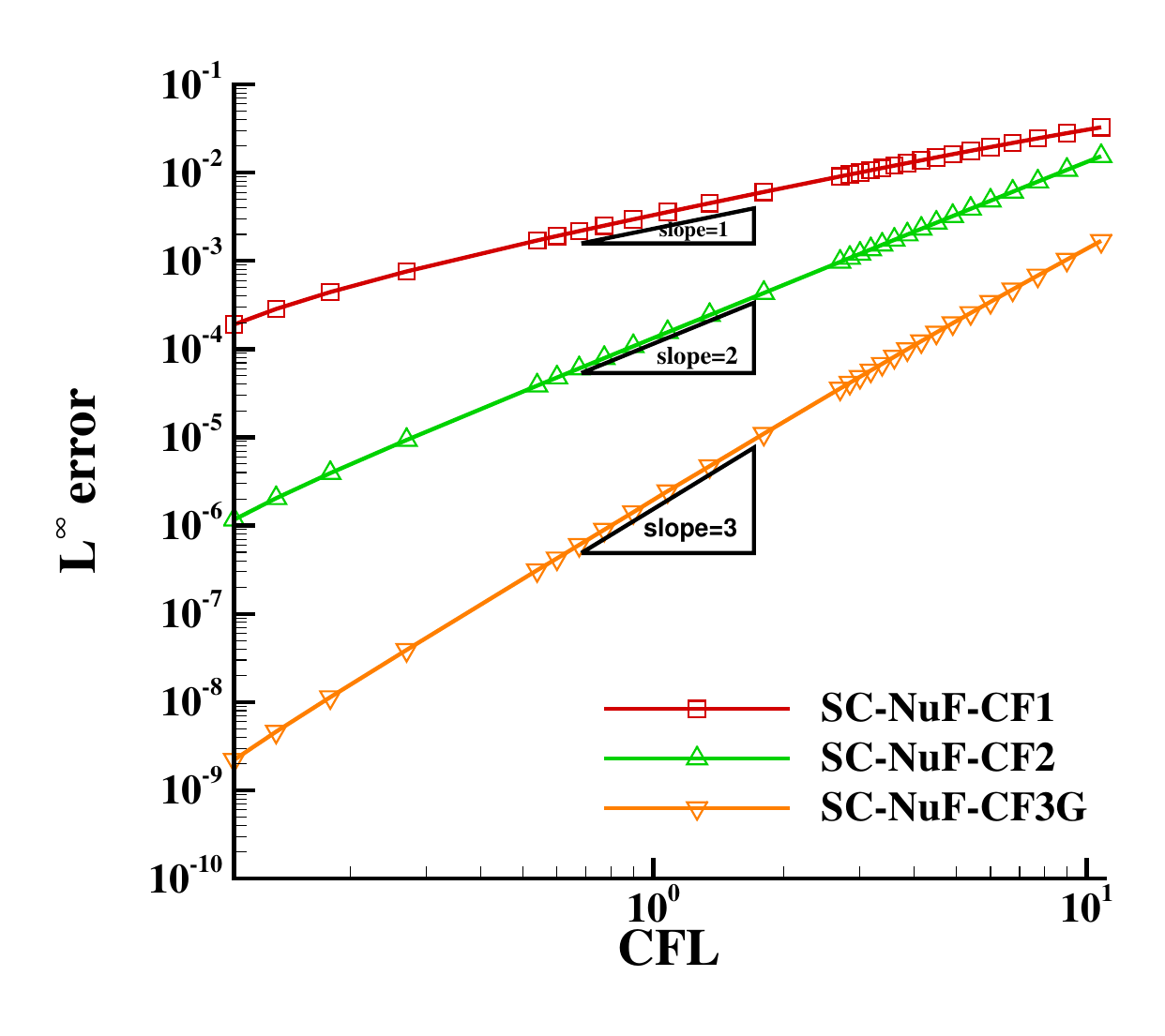}
        \caption{CFL v.s. $L^\infty$ error.}
        \label{fig:SC-NuF-CF2_KH_converge}
    \end{subfigure}
    \hfill
    \begin{subfigure}[b]{0.48\textwidth}
        \centering
        \includegraphics[width=\linewidth]{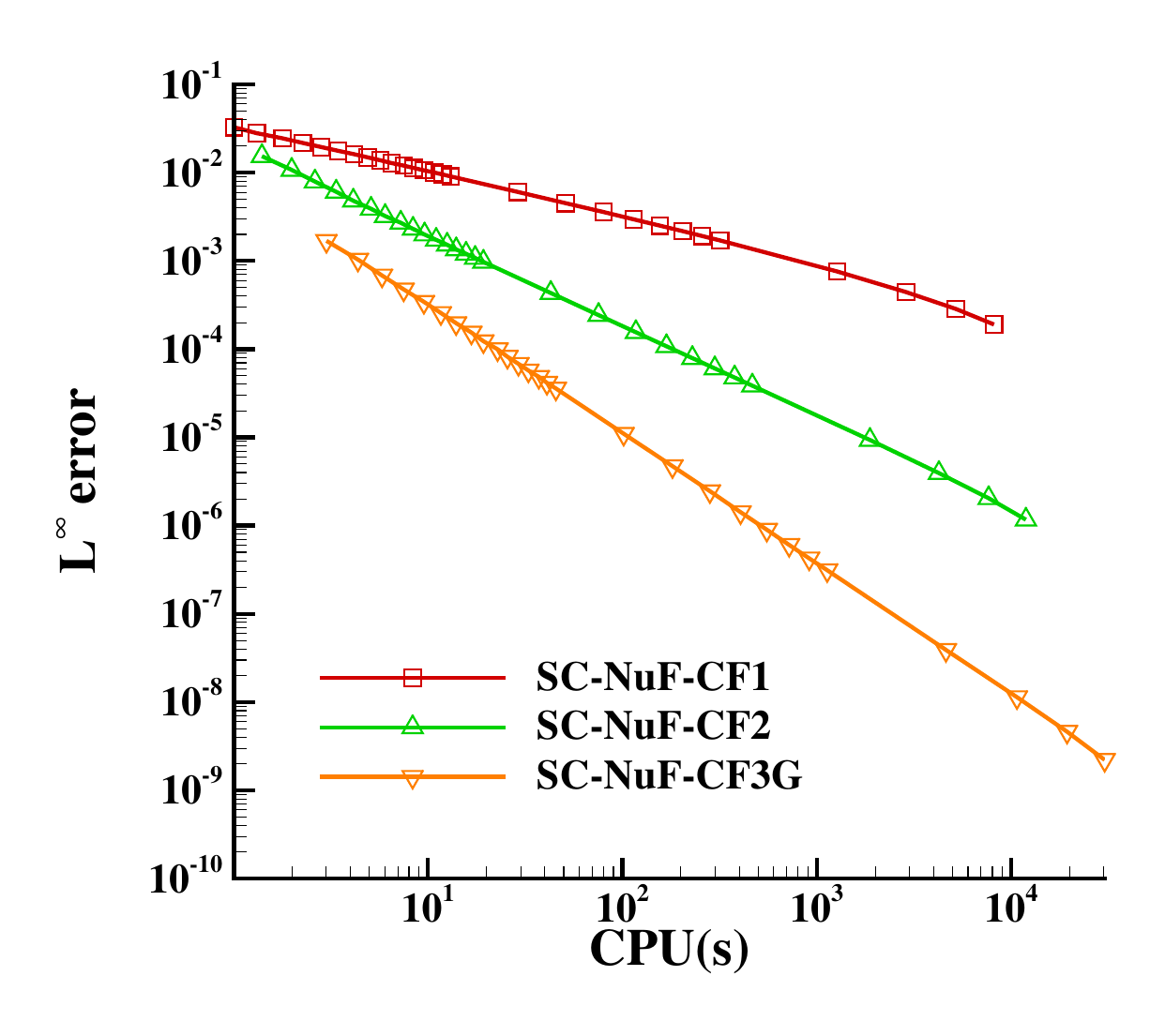}
        \caption{CPU times v.s. $L^\infty$ error.}
        \label{fig:SC-NuF-CF3G_KH_CPUs}
    \end{subfigure}
    \caption{Plots of $L^\infty$ error versus the CFL number (left) and error versus CPU time (right) 
            for Kelvin-Helmholtz instability at $T = 5$.
            Temporal convergence is measured using a separate reference solution for each scheme with $\tau=\frac{1}{1000}$.
            }

    \begin{subfigure}[b]{0.48\textwidth}
        \centering
        \includegraphics[width=\linewidth]{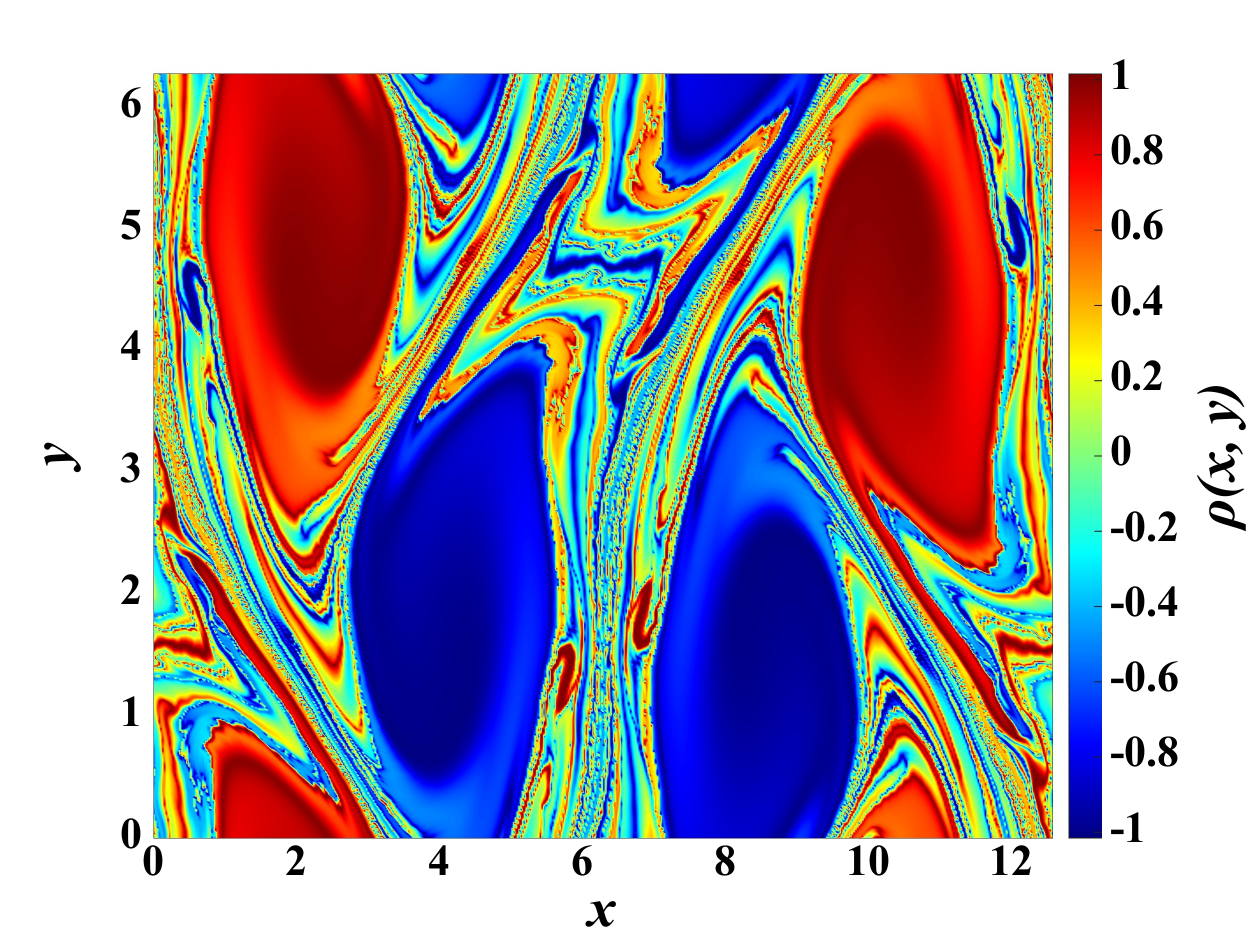}
    \end{subfigure}
    \hfill
    \begin{subfigure}[b]{0.48\textwidth}
        \centering
        \includegraphics[width=\linewidth]{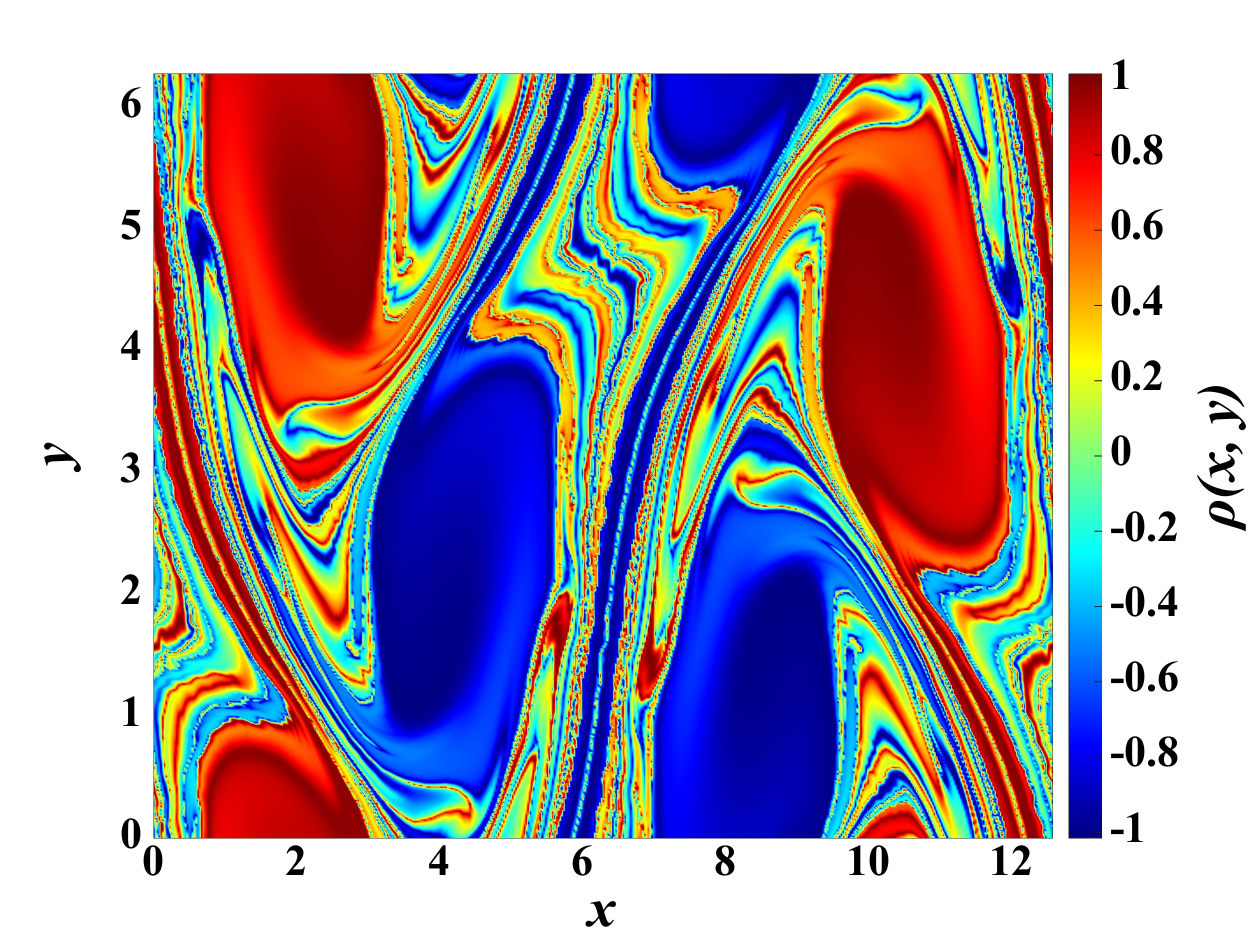}
    \end{subfigure}
    \caption{Kelvin--Helmholtz instability.
    Surface plots of the numerical solutions at $T=40$.
    The electric field is constructed using $128^2$ spatial points,
    with $128$ points in each spatial direction, and the time step is
    $\tau=0.4$.
    For visualization, the numerical solution is evaluated at
    $512^2$ sampling points, with $512$ points in each spatial direction.
    Left: SC-NuF-CF2. Right: SC-NuF-CF3G.}
    \label{fig:KH_plots}
\end{figure}

\subsection{Guiding-center model}

\begin{example}[Kelvin-Helmholtz instability problem]
We consider the guiding-center model \eqref{eq:GDcenter} with the initial condition
\begin{align}
    \rho_0(x_1,x_2) = \sin(x_2) + 0.015~\cos (kx_1),
\end{align}
and periodic boundary conditions on the domain $[0,4\pi]\times[0,2\pi]$. 
The choice $k=0.5$ gives the Kelvin--Helmholtz test studied in \cite{Shoucri1981,Zhu2017,Cai2019,CBQ2021}.

To assess the temporal accuracy, we use $128^2$ spatial quadrature
points, with $128$ points in each spatial direction.
For each scheme, a reference solution is computed with time step $\tau=\frac{1}{1000}$
over \([0,T]\), \(T = 5\).
Figure~\ref{fig:SC-NuF-CF2_KH_converge} shows 
the plotted slopes are consistent with the formal orders over the displayed range.
We report plots of $L^{\infty}$ error versus CPU time in Figure~\ref{fig:SC-NuF-CF3G_KH_CPUs}.

For the long-time simulation of the Kelvin--Helmholtz instability,
we construct the electric field using $128^2$ spatial points,
with $128$ points in each spatial direction, and set
$\tau=0.4$ (CFL$\approx4$).
Figure~\ref{fig:KH_plots} shows surface plots of the numerical
solutions at $T=40$, evaluated at $512^2$ sampling points,
with $512$ points in each spatial direction.
\end{example}

\subsection{2D2V Vlasov--Poisson system tests}
We next consider weak Landau damping and bump-on-tail instability
for the 2D2V Vlasov--Poisson system.
To address the computational cost associated with the four-dimensional
phase space, we use a hybrid CPU--GPU implementation of SC-NuF
for the experiments in this subsection.
Characteristic tracing, spline evaluation, and velocity-moment
integration are parallelized on the GPU, while the Poisson solve
and spline-coefficient construction are performed on the CPU.

The experiments are conducted on a server running 64-bit Windows
Server 2019, equipped with two Intel Xeon Silver 4216 processors
($2.10\,\mathrm{GHz}$, 16 physical cores per processor),
approximately $128\,\mathrm{GiB}$ of system memory,
and one NVIDIA A30 GPU with $24\,\mathrm{GB}$ of device memory.
The solver is implemented in Fortran with CUDA C++ kernels.
The software environment comprises GNU Fortran~15.1.0,
\texttt{fpm}~0.10.0, CUDA Toolkit~11.7,
the MSVC~14.29.30133 toolset, and NVIDIA driver~516.94.
Double-precision floating-point arithmetic is used throughout.
The phase-space domain, resolution, and time step are specified
for each test below.

\begin{figure}[htbp]
    \centering
    \begin{subfigure}[b]{0.48\textwidth}
        \centering
        \includegraphics[width=\linewidth]{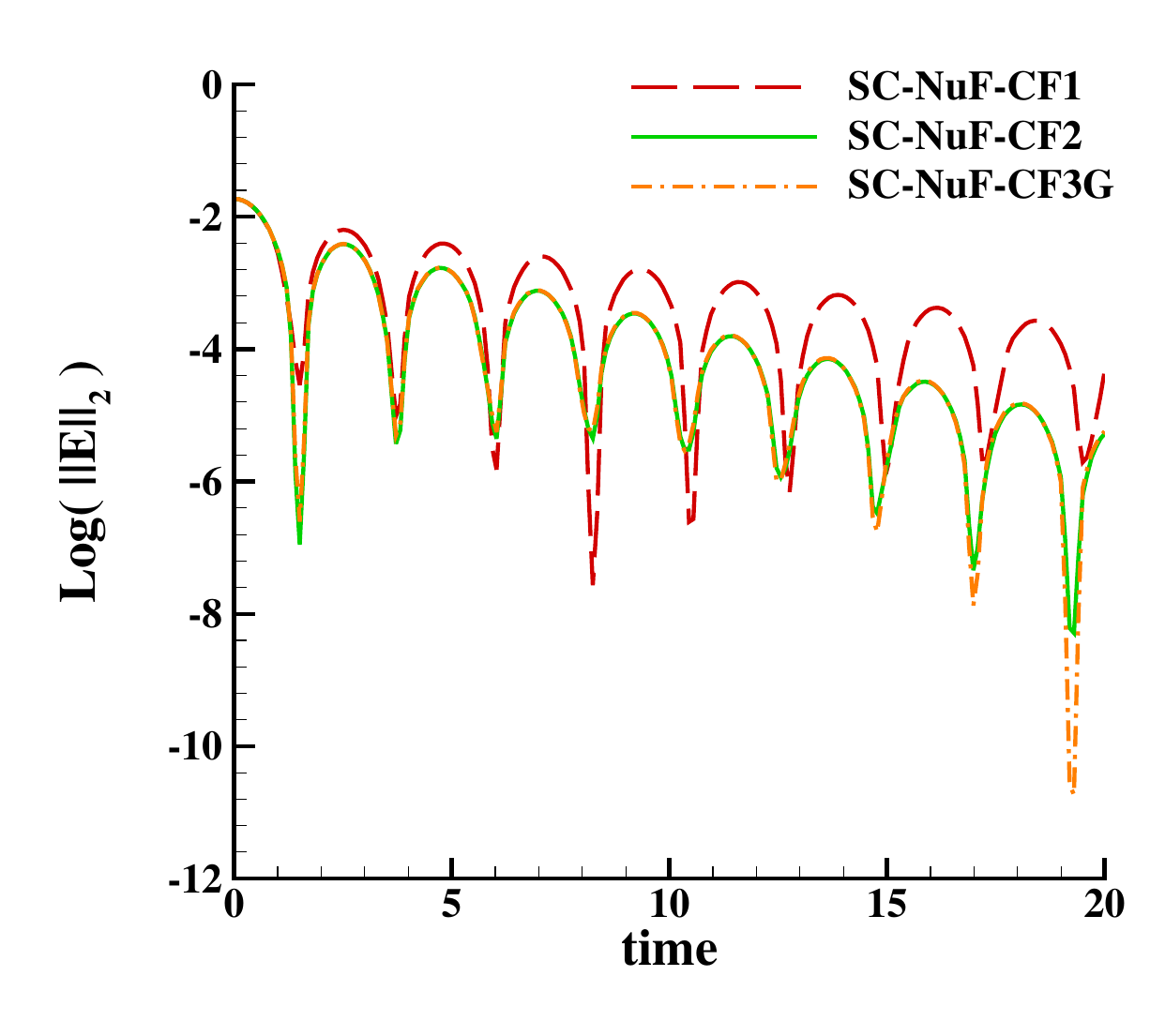}
    \end{subfigure}
    \hfill
    \begin{subfigure}[b]{0.48\textwidth}
        \centering
        \includegraphics[width=\linewidth]{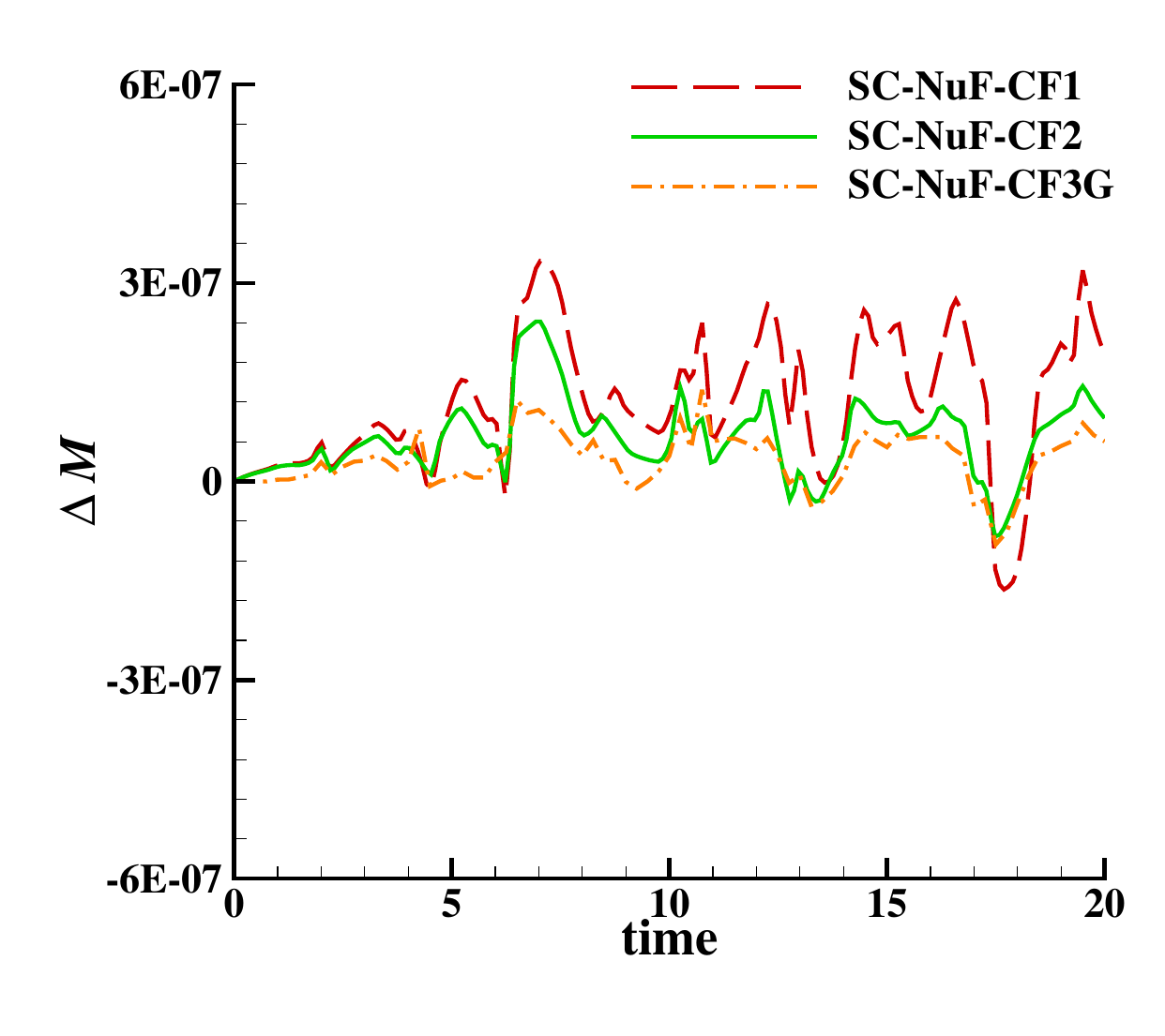}
    \end{subfigure}

    \begin{subfigure}[b]{0.48\textwidth}
        \centering
        \includegraphics[width=\linewidth]{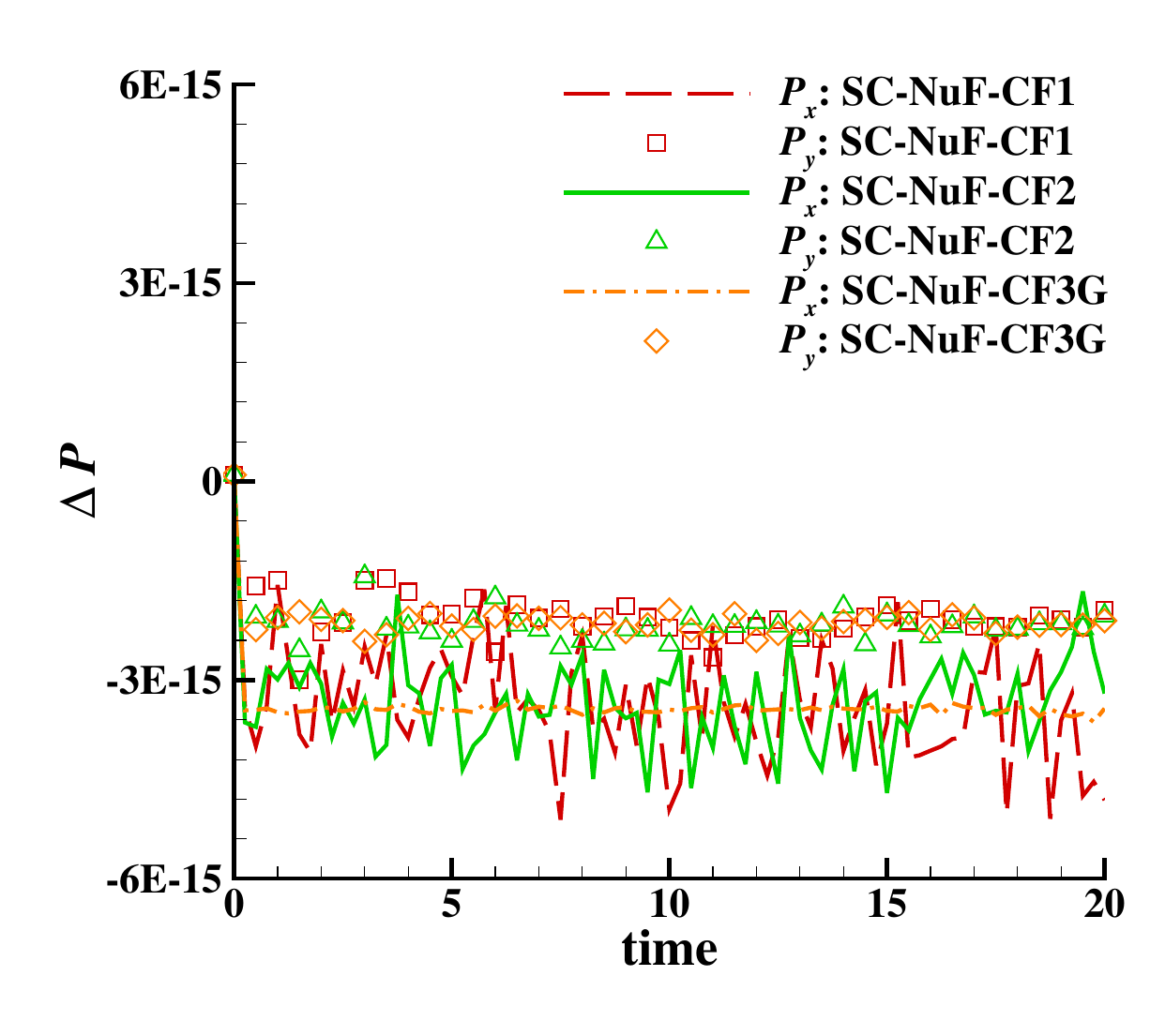}
    \end{subfigure}
    \hfill
    \begin{subfigure}[b]{0.48\textwidth}
        \centering
        \includegraphics[width=\linewidth]{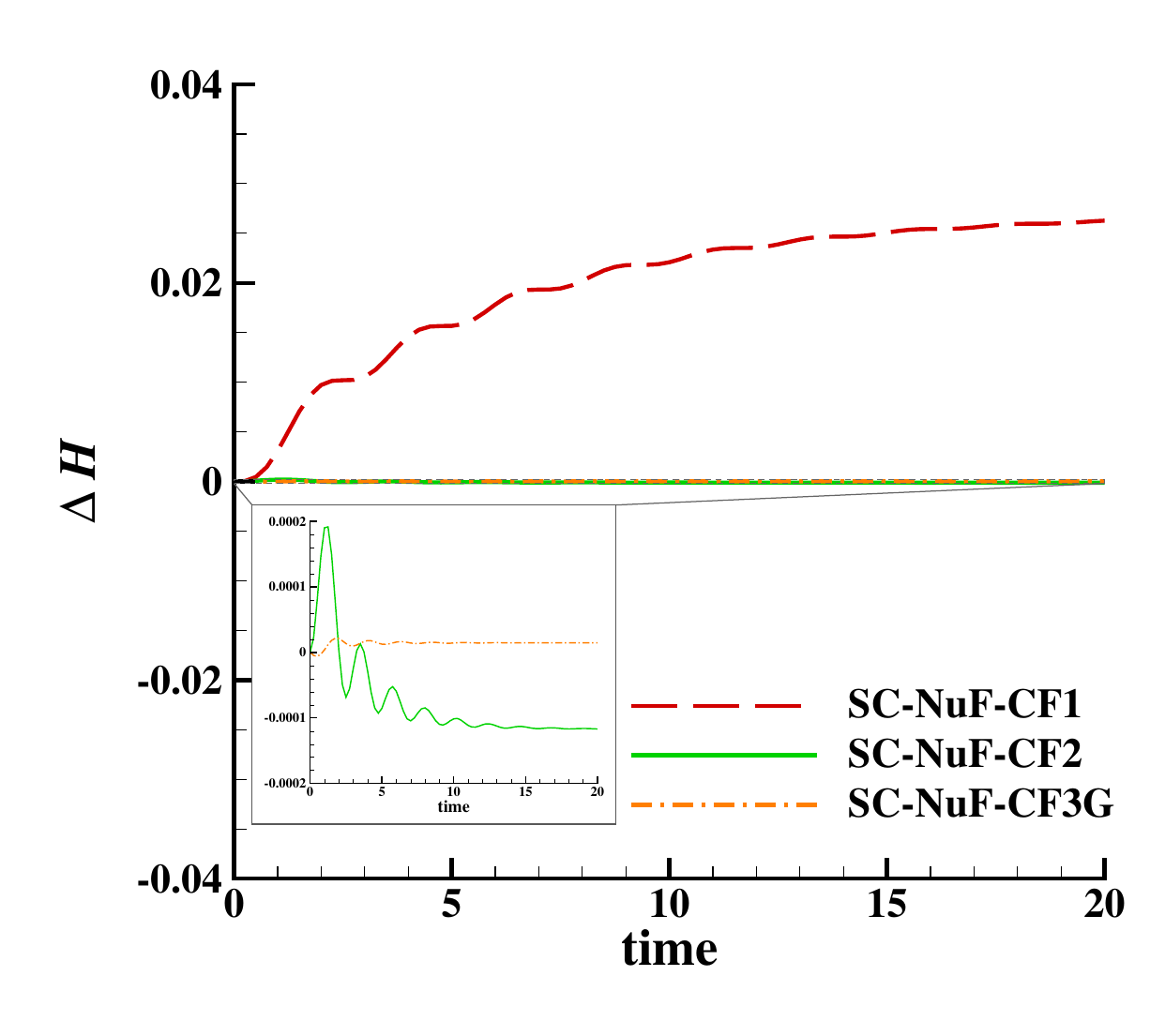}
    \end{subfigure}
    \caption{2D2V weak Landau damping. Evolution of $\log(\|E(t)\|_2)$
    and the signed deviations $M(t)-M(0)$, $P_i(t)-P_i(0)$ ($i=x,y$),
    and $\mathcal{H}(t)-\mathcal{H}(0)$.
    We use $32$ points in each spatial direction and $64$ points in each
    velocity direction, with $\tau=0.25$.}
    \label{fig:2D2V_weak_landau}
\end{figure}

\begin{example}[2D2V Landau damping]
We first consider the 2D2V weak Landau damping problem on the
phase-space domain $[0,L_x]^2\times[-L_v,L_v]^2,$ with $L_x=\frac{2\pi}{k}$ and $L_v=6$.
The initial distribution function is given by
\begin{equation}
    f_0(x_1,x_2,v_1,v_2)
    =
    \frac{1}{2\pi}
    \left[1+\alpha\cos(kx_1)+\alpha\cos(kx_2)\right]
    \exp\left(-\frac{v_1^2+v_2^2}{2}\right),
\end{equation}
where $\alpha=0.01$ and $k=0.5$.

We use $32$ quadrature points in each spatial direction and $64$ in each
velocity direction, with $\tau=0.25$, and integrate to $T=20$.
Figure~\ref{fig:2D2V_weak_landau} shows the electric-field history and
conservation diagnostics. SC-NuF-CF2 and SC-NuF-CF3G produce closely
matching decay envelopes, while SC-NuF-CF1 exhibits a more slowly
decaying envelope.
As in the 1D1V tests, fully discrete volume preservation and conservation
require compatible treatment of velocity quadrature, electric-field
computation and interpolation, and the inner high-order Runge--Kutta
integration of the frozen characteristic equations. Achieving these properties
remains a direction for future work. The small deviations in both momentum
components may again reflect the symmetry of the initial data and do not
establish momentum conservation for general data.
\end{example}

\begin{figure}[htbp]
    \centering
        \begin{subfigure}[b]{0.48\textwidth}
        \centering
        \includegraphics[width=\linewidth]{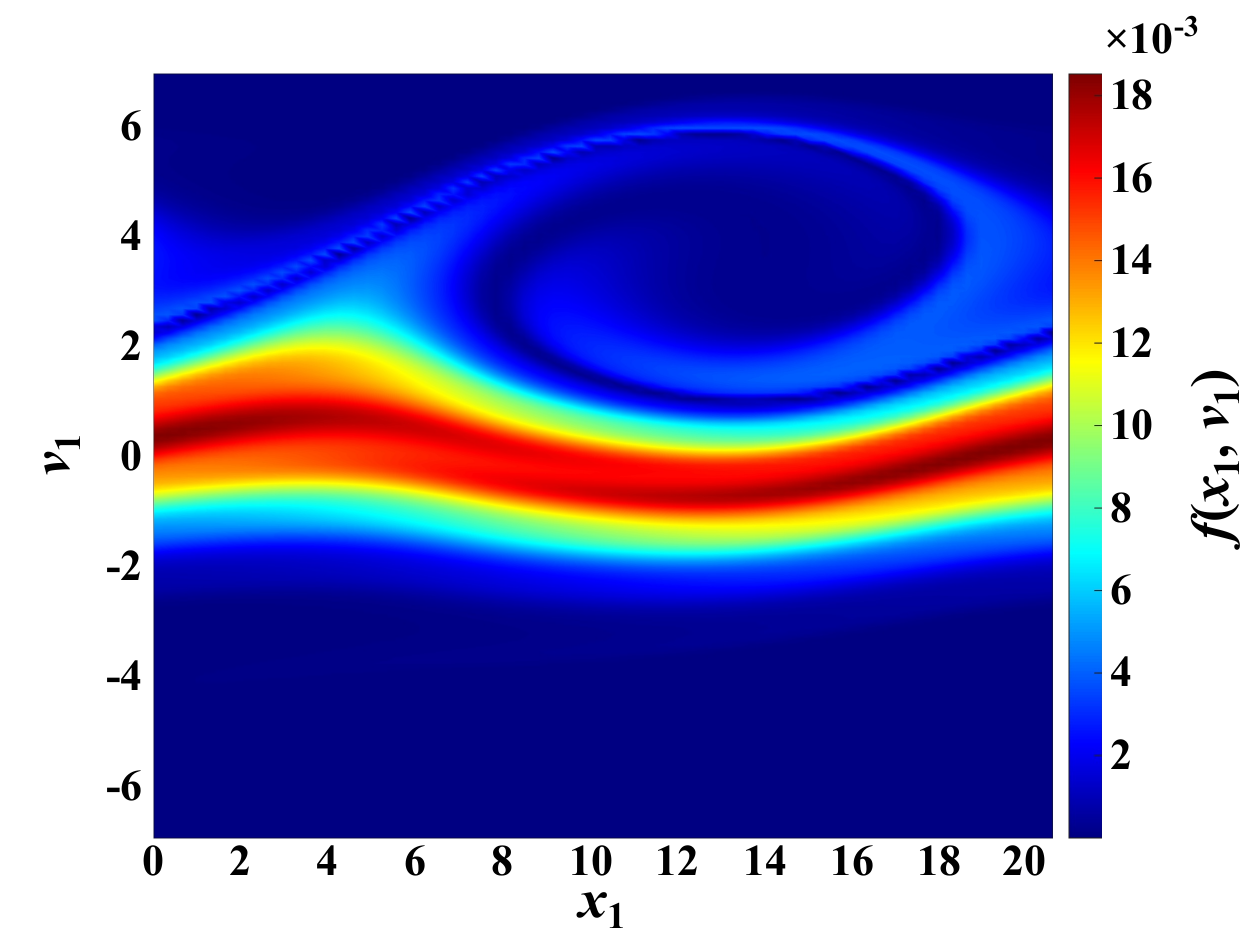}
        \caption{SC-NuF-CF2, $x_1$--$v_1$ slice.}
    \end{subfigure}
    \hfill
    \begin{subfigure}[b]{0.48\textwidth}
        \centering
        \includegraphics[width=\linewidth]{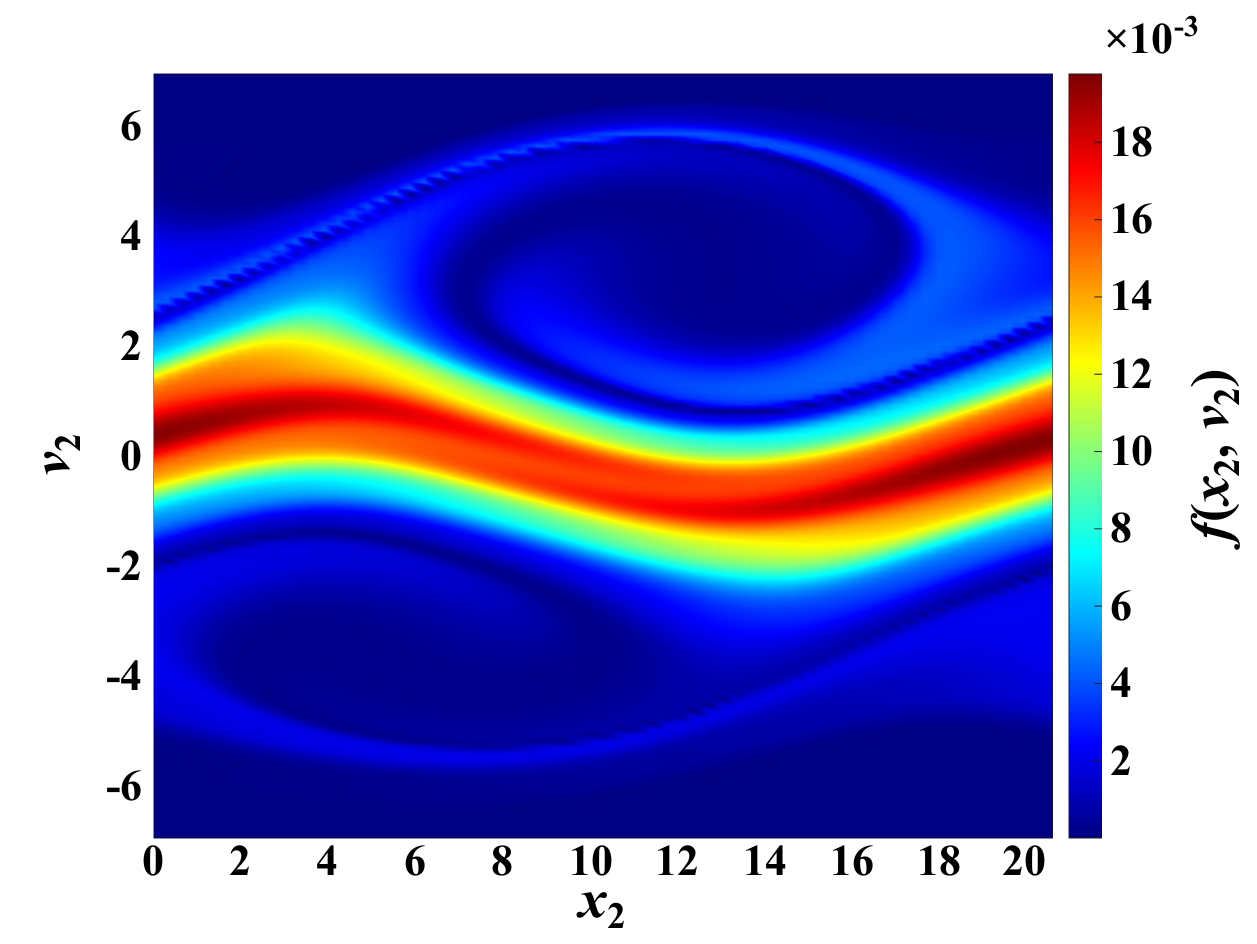}
        \caption{SC-NuF-CF2, $x_2$--$v_2$ slice.}
    \end{subfigure}

    \begin{subfigure}[b]{0.48\textwidth}
        \centering
        \includegraphics[width=\linewidth]{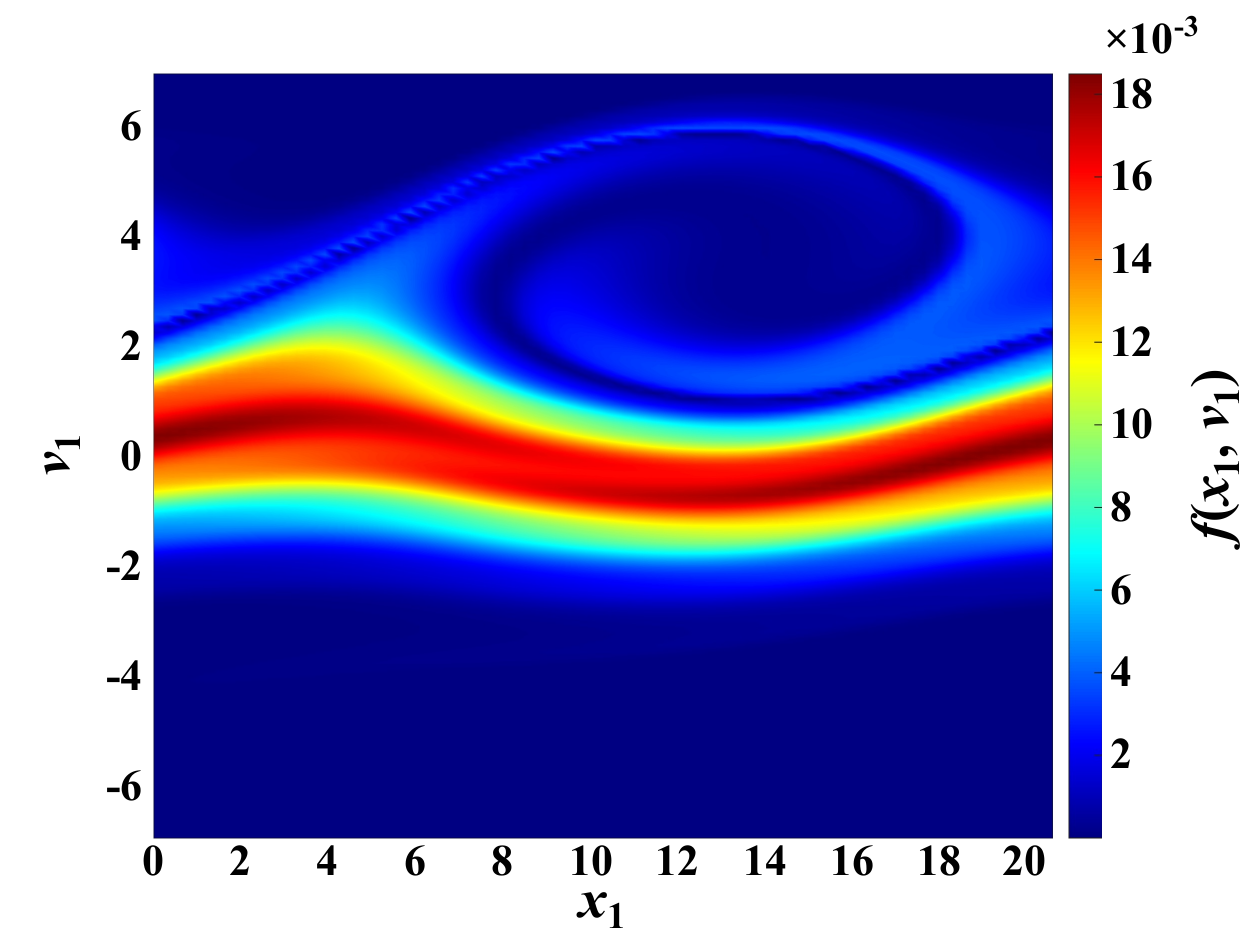}
        \caption{SC-NuF-CF3G, $x_1$--$v_1$ slice.}
    \end{subfigure}
    \hfill
    \begin{subfigure}[b]{0.48\textwidth}
        \centering
        \includegraphics[width=\linewidth]{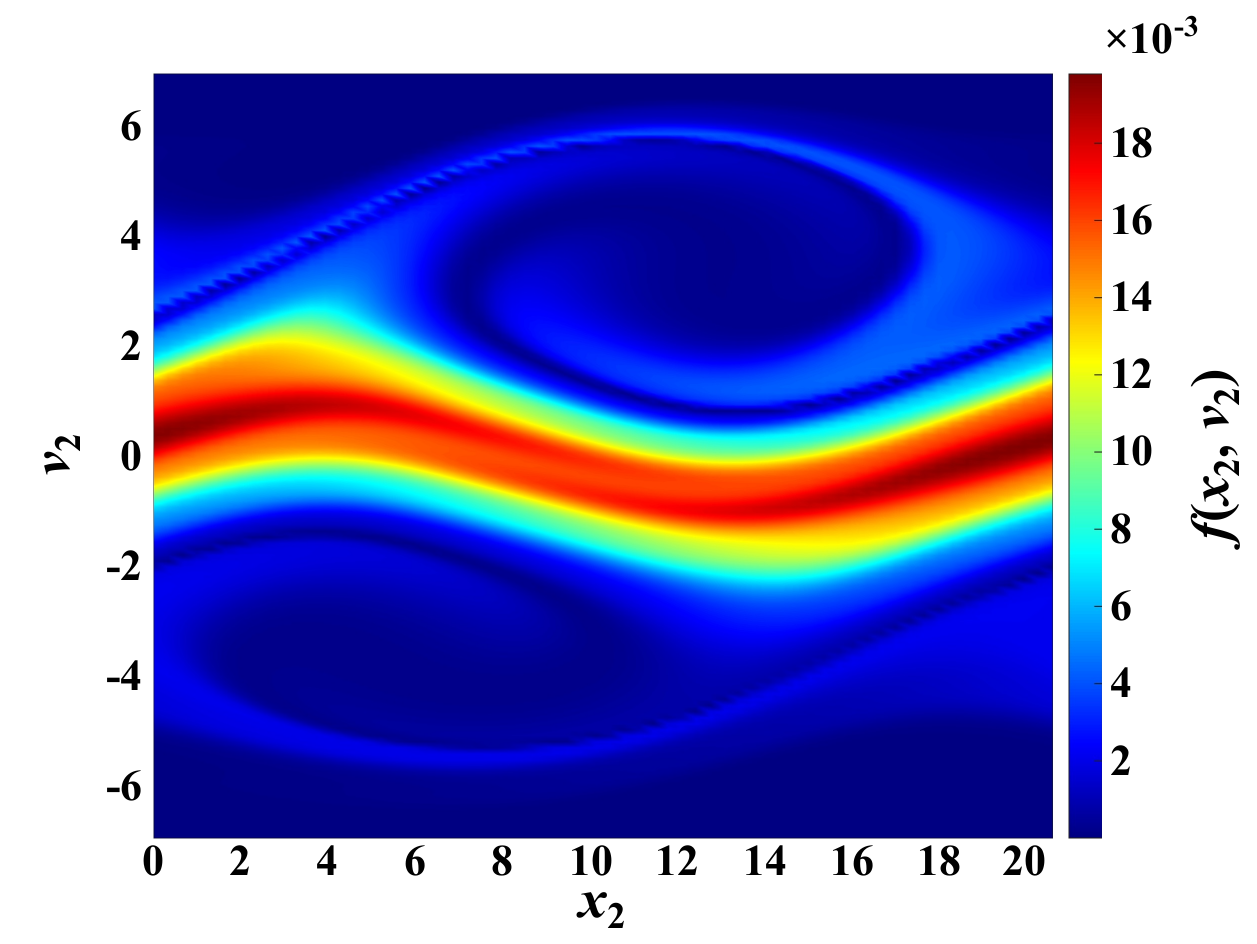}
        \caption{SC-NuF-CF3G, $x_2$--$v_2$ slice.}
    \end{subfigure}
    \caption{2D2V bump-on-tail instability.
    Surface plots of the numerical distribution function at $T=20$.
    The simulation uses $64^2\times 128^2$ phase-space points, with $64$ points
    in each spatial direction and $128$ points in each velocity direction, and a time step $\tau=0.2$.
    Left: the $x_1$--$v_1$ slice of $f(x_1,x_2,v_1,v_2)$
    at $x_2=0$ and $v_2=L_v/2$.
    Right: the $x_2$--$v_2$ slice of $f(x_1,x_2,v_1,v_2)$
    at $x_1=0$ and $v_1=L_v/2$.}
    \label{fig:2D2V_bump_on_tail_plots}
\end{figure}

\begin{example}[2D2V bump-on-tail instability]
We consider the 2D2V bump-on-tail instability problem on the
phase-space domain $[0,L_x]^2\times[-L_v,L_v]^2$ with $L_x=\frac{2\pi}{k}$ and $L_v=7$.
The initial distribution function is given by
\begin{equation}
    f_0(x_1,x_2,v_1,v_2)
    =
    \left[1+\alpha\cos(kx_1)+\alpha\cos(kx_2)\right]
    f_1(v_1)f_2(v_2),
\end{equation}
where
\begin{align*}
    f_1(v_1)
    =
    \sum_{i=1}^{2}
    \frac{\rho_i}{\sqrt{2\pi\theta_i}}
    \exp\left(-\frac{|v_1-u_i|^2}{2\theta_i}\right),\quad
    f_2(v_2)
    =
    \sum_{i=3}^{5}
    \frac{\rho_i}{\sqrt{2\pi\theta_i}}
    \exp\left(-\frac{|v_2-u_i|^2}{2\theta_i}\right).
\end{align*}
The parameters are
\begin{gather*}
    \rho_1=0.9,\qquad
    \rho_2=\rho_4=0.1,\qquad
    \rho_3=0.85,\qquad
    \rho_5=0.05,\\
    u_1=u_3=0,\qquad
    u_2=u_4=-u_5=4.5,\\
    \theta_1=\theta_3=1,\qquad
    \theta_2=\theta_4=\theta_5=0.25,
\end{gather*}
with $\alpha=0.05$ and $k=0.3$.
We apply SC-NuF-CF3G using $64^2\times 128^2$ phase-space points,
with $N_{x_1}=N_{x_2}=64$ and $N_{v_1}=N_{v_2}=128$,
and set the time step to $\tau=0.2$.

Figure~\ref{fig:2D2V_bump_on_tail_plots} shows the $x_1$--$v_1$ and
$x_2$--$v_2$ slices at $T=20$. Both exhibit deformed central populations,
rolled-up structures, and fine filaments. The $x_2$--$v_2$ slice additionally
contains a weaker structure at negative velocities, consistent with the
oppositely drifting populations in $f_2$. 
These features are qualitatively similar to those obtained
with the adaptive Hermite spectral method of Shao et al.~\cite{ShaoWangWu2026}
\end{example}

\section{Conclusion}
\label{sec:Conclusion}

We have developed a flow-map framework for nonlinear transports
based on CF-RKEI and initial-data pullback.
The classical exponential--pullback identity provides the basis for constructing
the self-consistent exact frozen flow by composing exact backward
characteristic maps generated by frozen linear combinations of stage
transport fields.
Its fully discrete realization, the self-consistent numerical flow,
combines field reconstruction with numerical integration of the frozen
characteristics.

A local defect identity quantifies the error introduced by freezing a
transport field. For a Lie--Euler VP step, the resulting estimate
separates the effects of incoming distribution error, field variation
and approximation, and characteristic integration. It explains the
role of solution gradients without requiring a uniform regularity
assumption on all numerical iterates.
For the Vlasov--Poisson system, we have shown that the self-consistent exact
frozen flow is a volume-preserving diffeomorphism under suitable
regularity assumptions.
Temporal self-convergence tests for the Vlasov--Poisson and
guiding-center models yield rates consistent with the CF-RKEI design
orders over the tested range.

Global nonlinear stability and convergence, including the PDE order of
the higher-order compositions, are not established here, and volume preservation is not
automatically inherited by its numerical characteristic maps.
Direct pullback of nonnegative initial data preserves pointwise
nonnegativity, but does not guarantee exact conservation of
quadrature-based mass, momentum, or energy.
Long-time simulations also exhibit drift in some monitored conserved
quantities.
These limitations indicate that both the fully discrete construction
and its analysis require further development.

\section*{Declaration of competing interest}
The authors declared that they have no conflicts of interest to this work. 
We declare that we do not have any commercial or associative interest that represents a conflict of interest in connection with the work submitted.

\section*{Acknowledgments}
The work of both authors was partially supported by the NSFC (No 12671463), 
Guangdong basic and applied basic research foundation [2025A1515012182], 
Guangdong and Hong Kong Universities “1+1+1” Joint Research Collaboration Scheme [2025A0505000014].

\bibliographystyle{elsarticle-num}
\bibliography{references}

\appendix
\section{CF RKEI tableaux}
\label{app:low-order-tableaux}
The corresponding Butcher tableaux of the CF1, CF2 and CF3G methods are presented in
Table~\ref{tab:butcher-low-order}.
\begin{small}

\begin{table}[htbp]

    \centering

    \begin{subtable}[t]{0.42\textwidth}
        \centering

        \begin{tikzpicture}[baseline=(B1.center)]

        \matrix (B1) [
            matrix of math nodes,
            nodes in empty cells,
            nodes={
                minimum width=1.05cm,
                minimum height=0.65cm,
                inner sep=1pt,
                anchor=center
            },
            column sep=0pt,
            row sep=0pt
        ] {
            0 &   \\
              & 1 \\
        };

        \draw[line width=0.4pt]
            (B1-1-1.north east)
            --
            (B1-1-1.north east |- B1-2-1.south east);

        \draw[line width=0.4pt]
            (B1-1-1.south west)
            --
            (B1-1-1.south west -| B1-1-2.south east);

        \end{tikzpicture}

        \caption{Butcher tableau for the CF1 method $(p=1)$.}
        \label{tab:butcher-p1}
    \end{subtable}
    \hfill
    \begin{subtable}[t]{0.42\textwidth}
        \centering

        \begin{tikzpicture}[baseline=(B2.center)]

        \matrix (B2) [
            matrix of math nodes,
            nodes in empty cells,
            nodes={
                minimum width=1.05cm,
                minimum height=0.65cm,
                inner sep=1pt,
                anchor=center
            },
            column sep=0pt,
            row sep=0pt
        ] {
            0        &           \\
            \tfrac12 & \tfrac12  \\
                     & 1         \\
        };

        \draw[line width=0.4pt]
            (B2-1-1.north east)
            --
            (B2-1-1.north east |- B2-3-1.south east);

        \draw[line width=0.4pt]
            (B2-2-1.south west)
            --
            (B2-2-1.south west -| B2-2-2.south east);

        \end{tikzpicture}

        \caption{Butcher tableau for the CF2 method $(p=2)$.}
        \label{tab:butcher-p2}
    \end{subtable}

    \vspace{1em}
    \begin{subtable}[t]{0.42\textwidth}
        \centering
    \begin{tikzpicture}[baseline=(B.center)]

    \matrix (B) [
        matrix of math nodes,
        nodes in empty cells,
        nodes={
            minimum width=1.05cm,
            minimum height=0.65cm,
            inner sep=1pt,
            anchor=center
        },
        column sep=0pt,
        row sep=0pt
    ] {
        0         &               &               &               \\
        \frac12   & \frac12       &               &               \\
        1         & -1            & 2             &               \\
                & \frac1{12}    & \frac13       & -\frac14      \\
                & \frac1{12}    & \frac13       & \frac5{12}    \\
    };

    \draw[line width=0.4pt]
        (B-1-1.north east)
        --
        (B-1-1.north east |- B-5-1.south east);

    \draw[line width=0.4pt]
        (B-3-1.south west)
        --
        (B-3-1.south west -| B-3-4.south east);

    \end{tikzpicture}

    \caption{Butcher tableau for the CF3G method $(p=3)$.}
    \label{tab:butcher-p3}
    \end{subtable}

    \caption{CF RKEI Butcher tableaux.}
    \label{tab:butcher-low-order}
\end{table}
\end{small}
\end{document}